\documentclass{amsart}

\usepackage{todonotes}
\usepackage{amscd,amsmath,amssymb,amsthm,mathrsfs,latexsym}
\usepackage[mathcal]{euscript}
\usepackage{stackrel}
\usepackage[all]{xy}
\usepackage{color,soul} 
\usepackage{hyperref} 
\hypersetup{
 colorlinks=true, 	 
 linktoc=all,     	 
 linkcolor=blue,  	 
 citecolor=blue,    
 filecolor=magenta,  
 urlcolor=black       
}

\usepackage{caption}
\usepackage{tikz-cd}

\def\ds{\displaystyle}

\usepackage{enumerate}

\newtheorem{thm}{Theorem}[section]
\newtheorem{cor}[thm]{Corollary}
\newtheorem{lem}[thm]{Lemma}
\newtheorem{prop}[thm]{Proposition}

\theoremstyle{definition}
\newtheorem{defn}[thm]{Definition}
\newtheorem{rmk}[thm]{Remark}
\newtheorem{claim}[thm]{Claim}
\theoremstyle{definition}
\newtheorem{ex}[thm]{Example}
\theoremstyle{remark}

\newcommand{\sm}{\wedge}
\newcommand{\Tor}{{\rm Tor}}
\newcommand{\Top}{\textrm{\rm Top}}
\newcommand{\Mor}{\textrm{Mor}}

\newcommand{\Bcom}{B_{\rm com}}
\newcommand{\ol}[1]{\overline{#1}}
\newcommand{\goesto}{\mapsto}
\newcommand{\xmaps}{\xrightarrow}
\newcommand{\srm}[1]{\stackrel{#1}{\maps}}
\newcommand{\srt}[1]{\stackrel{#1}{\to}}
\newcommand{\e}{\emph}
\def\co{\colon\thinspace}
\DeclareMathOperator{\colim}{colim}

\newcommand{\R}{{\mathbb R}}
\newcommand{\C}{{\mathbb C}}
\newcommand{\Z}{{\mathbb Z}}

\newcommand{\X}{{\it X}}

\newcommand{\Susp}{\Sigma}
\renewcommand{\P}{\mathcal{P}}
\newcommand{\Sym}{\mathrm{Sym}}
\newcommand{\Map}{\mathrm{Map}}
\newcommand{\wt}{\widetilde}

\newcommand{\K}{{\mathcal{K}}}
\renewcommand{\H}{{\mathcal{H}}}
\newcommand{\Rep}{{\rm Rep}}
\newcommand{\SL}{{\rm SL}}
\newcommand{\GL}{{\rm GL}}
\newcommand{\U}{{\rm U}}
\newcommand{\SU}{{\rm SU}}
\newcommand{\SO}{{\rm SO}}
\newcommand{\Sp}{{\rm Sp}}
\newcommand{\Spin}{{\rm Spin}}

\newcommand{\FI}{{\rm FI}}

\newcommand{\FIs}{{\rm FI\#}}
\newcommand{\n}{{\rm \bf n}}
\newcommand{\p}{{\rm \bf p}}
\renewcommand{\k}{{\rm \bf k}}
\renewcommand{\r}{{\rm \bf r}}
\newcommand{\s}{{\rm \bf s}}

\newcommand{\m}{{\rm \bf m}}

\newcommand{\0}{{\rm \bf 0}}

\newcommand{\Aut}{{\rm Aut}}

\newcommand{\Ind}{{\rm Ind}}

\renewcommand{\C}{{\mathbb C}}
\newcommand{\Q}{{\mathbb Q}}

\newcommand{\bbN}{\mathbb{N}}
\newcommand{\bbR}{\mathbb{R}}

\newcommand{\bbZ}{\mathbb{Z}}

\newcommand{\bbQ}{\mathbb{Q}}

\newcommand{\Hom}{{\rm Hom}}
\newcommand{\Comm}{{\rm Comm}}

\newcommand{\leqs}{\leqslant}
\newcommand{\geqs}{\geqslant}
\newcommand{\heq}{\simeq}

\newcommand{\maps}{\longrightarrow}
\newcommand{\lmaps}{\longleftarrow}
\newcommand{\injects}{\hookrightarrow}
\newcommand{\homeo}{\cong}
\newcommand{\surjects}{\twoheadrightarrow}
\newcommand{\isom}{\cong}
\newcommand{\cross}{\times}

\usepackage{mathtools} 

\def\ds{\displaystyle}

\title[Homological stability for commuting elements in Lie groups]{Homological stability for spaces of commuting elements in Lie groups}

\author{Daniel A. Ramras}
\address{Mathematical Sciences, Indiana University-Purdue University, Indianapolis, IN 46202}
\email{dramras@iupui.edu}

\author{Mentor Stafa}
\address{Department of Mathematics, Tulane University, New Orleans, LA 70118}
\email{mstafa@tulane.edu}

\date{\today}
\subjclass[2010]{Primary 57T10, 20C30; Secondary   57T35, 57T15, 55N91}
\keywords{representation space, representation stability, homological stability, Weyl group, $\FI_W$--module}
\thanks{D. Ramras was partially supported by the Simons Foundation (Collaboration Grant \#279007).}

\begin{document}

\begin{abstract}
In this paper we study homological stability for spaces $\Hom(\Z^n,G)$ of pairwise commuting $n$--tuples in a Lie group $G$. We prove that for each $n\geqs 1$, these spaces satisfy rational homological stability as $G$ ranges through any of the classical sequences of compact, connected Lie groups, or their complexifications. We prove similar results for rational equivariant homology, for character varieties, and for the infinite-dimensional analogues of these spaces, $\Comm(G)$ and $\Bcom G$, introduced by Cohen--Stafa and Adem--Cohen--Torres-Giese respectively. In addition, we show that the rational homology of the space of unordered commuting $n$--tuples in a fixed group $G$ stabilizes as $n$ increases.

Our proofs use the theory of representation stability -- in particular, the theory of $\FI_W$--modules developed by Church--Ellenberg--Farb and Wilson. In all of the these results, we obtain specific bounds on the stable range, and we show that the homology isomorphisms are induced by maps of spaces.

\end{abstract}

\maketitle

\tableofcontents

\section{Introduction}

Let $G$ be a compact, connected Lie group. The focus of this article is the space $\Hom(\Z^n,G)$ 
consisting of all group homomorphisms $\rho: \Z^n \to G$. 
The standard basis for $\Z^n$ defines an embedding of $\Hom(\Z^n, G)\injects G^n$, and  
$\Hom(\Z^n, G)$ has the resulting subspace topology.
 This space, then, consists of ordered commuting $n$--tuples in $G$.
We will consider various related spaces as well, such as  
the character variety, or representation space, $\Rep(\Z^n,G)=\Hom(\Z^n,G)/G$.

These spaces can also be interpreted as  
moduli spaces of flat connections on principal $G$--bundles over the $n$--torus,
as studied  by Witten~\cite{witten1998toroidal} and Kac--Smilga~\cite{kac2000Smilga}.
The cases of commuting pairs and triples were examined in detail in
the monograph by Borel, Friedman and Morgan~\cite{borel2002almost}. More recently, homological and homotopical properties of these spaces have been studied by a variety of authors, including Baird~\cite{baird2007cohomology}, Adem--Cohen~\cite{adem2007commuting}, Florentino--Lawton~\cite{florentino2014topology}, and Pettet--Souto~\cite{pettet2013souto}.

In three of the most classical cases, namely
$G = \U(r)$, $\SU(r)$, or $\Sp(r)$, the representation space $\Hom(\Z^n, G)$ is 
connected for all $n\geqs 1$. In fact, these are the only 
 semisimple compact connected Lie groups with this 
property~\cite{adem2007commuting,kac2000Smilga}. For instance 
Sjerve and Torres-Giese \cite{torres2008fundamental} show that 
 $\Hom(\Z^n, \SO(3))$ is disconnected for $n\geqs 2$.

The authors \cite{ramras2017hilbert} gave an explicit formula for 
the Poincar\'e series of the path component of the trivial representation 
$\Hom(\Z^n,G)_1$ using methods from \cite{cohen2016spaces}, and the second author \cite{stafa2017poincare}
gave a similar formula for the path component $\Rep(\Z^n,G)_1$:
\begin{equation}\label{eqn: Hom-PP}
P(\Hom(\Z^n,G)_1;q)=\frac{1}{|W|}  \left( \prod_{i=1}^r (1-q^{2d_i}) \right)
\left(\sum_{w\in W} \frac{ \det(1+qw)^n}{\det(1-q^2w)} \right),
\end{equation}
and
\begin{equation}\label{eqn: Rep-PP}
P(\Rep(\Z^n,G)_1;q)=\frac{1}{|W|} \sum_{w\in W} \det(1+qw)^n,
\end{equation}
where $T$ is a maximal torus, $W$ is the Weyl group, and the positive integers $d_1,\dots,d_r$ 
are the characteristic degrees of   $W$. (For complex reductive groups, the mixed Hodge structure of these character varieties was subsequently determined by Florentino--Silva~\cite{FS2017}.) 
Explicit computations based on these formulas indicated a stability pattern, for fixed $n$, as $G$ varies through one of the classical sequences of Lie groups (e.g. the unitary groups). The aim of this article is to prove that this phenomenon does in fact hold. We establish similar results for a wide range of spaces built from commuting tuples, summarized in Table~\ref{table: all sequences of spaces}.
These stability results can be seen as extending the well-known fact that the homology of $G$ itself stabilizes in each of the classical families of Lie groups.

In Section~\ref{sec: nil}, we extend our results to reductive Lie groups and to nilpotent discrete groups. In the case of homomorphism spaces, it is a theorem of Pettet and Souto~\cite{pettet2013souto} that if $G$ is the group of real or complex points of a (Zariski) connected, reductive algebraic group with maximal compact subgroup $K$, then $\Hom(\Z^n, G)$ deformation retracts to $\Hom(\bbZ^n, K)$. The corresponding result for $\Rep(\Z^n, G)$ was established by Florentino--Lawton~\cite{florentino2014topology} (note that when $G$ is reductive, $\Rep(\Z^n, G)$ should be interpreted as a GIT quotient), and generalizations to nilpotent groups discrete groups were obtained by Bergeron~\cite{bergeron2015topology}.
Moreover,
Bergeron and Silberman  showed in \cite{bergeron2016note} that when $G$ is compact and $\Gamma$ is nilpotent, the 
connected component of the identity in $\Hom(\Gamma, G)$ contains only abelian representations, and hence is the same as $\Hom(H_1 (\Gamma),G)_1$. 
These results allow us to extend most of the stability statements in the paper (Section~\ref{sec: nil}). In Section~\ref{sec: covers}, we show that the homology of the spaces considered here is invariant under finite coverings of Lie groups.

We approach  homological stability using the machinery of \e{representation stability}
in the sense of Church and Farb \cite{church2013representation}. In particular, we use J. Wilson's extension of this theory to the classical sequences of Weyl groups. 
Each of our  homological stability results is obtained by first proving representation stability for an associated sequence.
For instance, work of Baird identifies the homology of  $\Hom(\bbZ^n, G)_1$ as the fixed points of a certain Weyl group representation. 
We use representation  stability to control the behavior of this sequence of Weyl group representations as $G$ varies through one of the classical sequences of Lie groups, finding in particular that the ranks of the fixed-point subspaces (the isotypical components of the trivial representation) stabilize. Stability bounds are obtained using the theory of $\FIs$--modules introduced by Church--Ellenberg--Farb~\cite{church2015fi}, and Wilson's extension of this theory to Weyl groups of type B/C~\cite{Wilson-Math-Z}. 

We give a minimalist introduction to 
representation stability, using the language of $\FI_W$-- and $\FI_W\#$--modules,  
in Section \ref{FIW-modules}. We  show that if an $\FI_W$--module $V = \{V_i\}_{i\geqs 1}$ is finitely generated in stage $d$, then for $n\geqs d$,
there are canonical surjections $V_n^{W_n} \surjects V_{n+1}^{W_{n+1}}$ 
between the subspaces of invariants (Proposition~\ref{avg}). 
This statement has the following topological consequence:
if the degree--$k$ rational homology of an $\FI_W$--space $\{X_r\}_{r\geqs 1}$ is generated in stage $N$, then the maps $H_k (X_n/W_n) \to H_k (X_{n+1}/W_{n+1})$ are surjective for $n\geqs N$ (Theorem~\ref{quot-stab}).

The reader familiar with the theory of $\FIs$--modules may notice that the spaces $\Hom(\bbZ^n, G)$ \e{do not} have the structure of $\FIs$--spaces as $G$ varies. We exploit work of Baird~\cite{baird2007cohomology}, which shows that their \e{equivariant} homology does have this extended structure. This allows us to deduce stability bounds for equivariant homology (Section~\ref{sec: equivariant stability}), from which we derive bounds on ordinary homological stability (Section~\ref{sec: stable range for Hom}) using a comparison of Eilenberg--Moore spectral sequences.

 It is worth mentioning that homological stability results have a long history, going back to work of Arnold 
\cite{arnol1969cohomology} and Cohen \cite{cohen1972thesis} on braid groups, and work of Quillen on general linear groups~\cite{Quillen}. The methods we use here are quite different.

Before  describing the main results in this article, we make precise our terminology regarding homological stability.
A sequence of topological spaces $\{X_n\}_{n\geqslant 0}$ with maps
 $\phi_{n}: X_n \to X_{n+1} $ between them
is {\it strongly rationally homologically stable}  if for each $k\geqs 0$,
there exists   $N = N(k)$ such that
$$(\phi_n)_*\co H_k (X_n)\maps H_k (X_{n+1})$$
is an isomorphism for  $n\geqs N$. Since our main results are all for rational homology, we will often drop the coefficient group $\Q$ from the notation. While we focus on rational homology, the corresponding statements for rational cohomology are equivalent.

\subsection{Ordered commuting tuples}
Let $\{G_r\}_{r\geqs 1}$ denote one of the classical  families of 
compact, connected Lie groups -- namely $G_r = \SU(r)$, $\U(r)$, $\SO(2r+1)$,  $\Sp(r)$, or $\SO(2r)$ -- or the complexifications thereof. We have standard inclusions $G_r\injects G_{r+1}$ (see Section~\ref{Lie}).
 
\begin{thm}
Fix $k, n\geqs 0$ and let $\{G_r\}_{r\geqs 1}$ be one of the classical sequences of Lie groups listed above.
Then the standard inclusions $G_r\injects G_{r+1}$ induce isomorphisms 
$$H_k(\Rep(\Z^n,G_r)_1; \bbQ)\srm{\isom} H_k(\Rep(\Z^n,G_{r+1})_1; \bbQ)$$
for $r\geqslant k$, and
$$H_k(\Hom(\Z^n,G_r)_1; \bbQ)\srm{\isom} H_k(\Hom(\Z^n,G_{r+1})_1; \bbQ)$$
for $r  - \lfloor \sqrt{r} \rfloor \geqs k$, where $\lfloor \sqrt{r} \rfloor$ is the floor function.
\end{thm}

These results are proved in  Theorems~\ref{stability-wrt-r-Rep} and~\ref{Hom-bound}, and Corollary~\ref{cor: nil complex} respectively. We prove similar results for $G_r$--equivariant homology in Section~\ref{sec: equivariant stability}.  
 This result, and the others discussed below, also apply to the  groups $\Spin(r)$ (Example~\ref{ex: Spin}), and to the projectivizations of the above sequences (Example~\ref{ex: proj}), although for the projective groups there are no maps inducing the  isomorphisms.

\subsection{Unordered commuting tuples}
Let $G$ be a compact and connected Lie group (not necessarily of classical type).
The  symmetric group $S_n$ acts on $\Z^n$ by permuting the standard basis, and hence acts on
 $\Hom(\Z^n,G)$ and $\Rep(\Z^n,G)$ by permuting the entries of commuting $n$--tuples.
The space of
\textit{unordered pairwise commuting $n$--tuples} in $G$ is the quotient $\Hom(\Z^n,G)/S_n$, and the corresponding moduli space is 
$\Rep(\Z^n,G)/S_n.$ 
We have natural inclusions
$$\Hom(\Z^n,G)/S_n \injects \Hom(\Z^{n+1},G)/S_{n+1}$$
sending $[(g_1, \ldots, g_n)]$ to $[(g_1, \ldots, g_n, 1)]$, where $1\in G$ is the identity, and similarly for unordered representation spaces.
 
\begin{thm}\label{unordered-thm}
Let $G$ be the group of 
complex or real points of a reductive linear algebraic group,
defined over $\R$ in the latter case. Then for each $k\geqs 0$ the 
sequences
$$n\goesto H_k(\Hom(\Z^n,G)_1; \bbQ) \textrm{ and } n\goesto H_k(\Rep(\Z^n,G)_1; \bbQ)$$
satisfy uniform representation stability, with stable range $n\geqs 2k$. Moreover, the
natural maps
$$H_k(\Hom(\Z^n,G)_1/S_{n}; \bbQ)\maps H_k(\Hom(\Z^{n+1},G)_1/S_{n+1}; \bbQ)$$
and
$$H_k(\Rep(\Z^n,G)_1/S_{n}; \bbQ)\maps H_k(\Rep(\Z^{n+1},G)_1/S_{n+1}; \bbQ)$$
are isomorphisms for $n\geqs k$.
\end{thm}

This is proven in Theorem~\ref{thm: stability for fixed G} and Corollary~\ref{cor: stability in n nil}, while an analogue for equivariant homology appears in Theorem~\ref{thm: equiv.coh. stability Hom mod S_n}. We note that these results bear some similarity to homological stability for configuration spaces, where one stabilizes with respect to the size of the configurations. It would be interesting to know whether stability holds for \e{configurations} of commuting elements in $G$; that is, commuting tuples in which repetition is not allowed.

\subsection{Symmetric products and representation stability}

The above results are closely related to a general homological stability result for symmetric products. This stability result goes back at least to Steenrod~\cite{steenrod}; here we use representation stability to give a simple proof of rational homological stability for these spaces. (Steenrod in fact proved stability for \e{integral} homology, and this will allow us to deduce integral stability in several of the situations we study.)

Recall that the $n$--th symmetric product of a space $X$ is the quotient space
$\Sym^n X = X^n/S_n$, where the symmetric group $S_n$ acts
by permuting the factors. A basepoint in $X$ determines maps $\Sym^n X \injects \Sym^{n+1} (X)$ (adding the basepoint in the last factor), and the colimit of this sequence is $\Sym^\infty (X)$.
 


In general, if $X$ is path connected and $H_* (X; \bbQ)$ is finitely generated  in each degree, then we show (Proposition \ref{symm-prod}) that for each $k\geqs 0$, 
 $H_k (X^n; \bbQ)$ 
 forms a uniformly representation stable
sequence of $S_n$-representations. 
 %
In Corollary~\ref{symm-prod-cor}, we explain how homological stability for symmetric products then follows from the general theory. For spaces equipped with an involution, we obtain a similar representation stability result in Proposition~\ref{signed-symm-prod}.


In addition to the representation stability result for products, we also use work of Baird \cite{baird2007cohomology} to prove various related representation stability results, outlined in Table \ref{table: rep stability results}, that help us prove the homological stability results in Table \ref{table: all sequences of spaces}.

\subsection{Infinite-dimensional analogues}

We also study two constructions that combine, in different ways, the spaces  $\Hom(\Z^n,G)$ and $\Rep(\Z^n,G)$ for varying $n$.
The space $\Bcom G$, known as the \e{classifying space for commutativity in $G$}, is the
geometric realization of the simplicial space $\Hom(\Z^\bullet,G)$. Introduced by Adem--Cohen--Torres-Giese~\cite{adem2012commuting}, this space has been studied by a variety of authors~\cite{adem2015gomez, adem2017gomezlindtillman, AGV, gritschacher2018spectrum}. 
When $G=\U$ is the infinite unitary group, $\Bcom U$ represents \e{commutative complex $K$--theory}.  
The space $\Comm(G)$ is a subspace of the James reduced product $J(G)$, and 
$\Comm(G)/G$  is its image in $J(G)/G$. These spaces played an important role in the calculations of Poincar\'e series discussed above. We show that   these constructions
satisfy strong rational homological stability for all of the above sequences of Lie groups (Theorem~\ref{Hom-bound} and Corollary~\ref{cor: nil stable}). We note that the rational cohomology rings of $\Bcom \U$, $\Bcom \SU$ and $\Bcom \Sp$ were calculated
in \cite{adem2015gomez}, and each is polynomial.

 
\

In conclusion,  our main results   can 
be summarized by stating that the sequences of rational 
homology groups in Table \ref{table: rep stability results} (below) enjoy representation stability, while the sequences in Table \ref{table: all sequences of spaces} 
stabilize.  In these tables, the sequence $G_r$ can be any of the sequences discussed above, with maximal torus $T_r < G_r$; in fact, our results also extend to certain sequences of real reductive groups (Section~\ref{sec: nil}). Additionally, in Section~\ref{sec: nil} we  extend our results to finitely generated nilpotent groups in place of $\Z^n$.
We do not expect the bounds in Table \ref{table: all sequences of spaces} to be optimal;
in fact,   the bounds  $r- \lfloor \sqrt{r} \rfloor \geqs k$ can be improved to $r- \lfloor \sqrt{r} \rfloor +1\geqs k$ (except for $G_1 = \SO(2)$ --  see Remark~\ref{rmk: +1}).
Note that these bounds imply slightly weaker linear bounds.

\begin{table}[ht!]
\centering

\renewcommand{\arraystretch}{1.5}

\begin{tabular}{llr}
\hline
\textit{Module} & \textit{Groups acting} & \textit{Stable range}\\
\hline
\hline

$H_k(X^n)$ 						& $S_n$, permuting coordinates   & $n\geqslant 2k$\\
\hline
$H_k(T_r)$ 						& $W_r$, acting diagonally   & $r\geqslant 2k$*\\
$H_k(G_r/T_r)$ 					& $W_r$, acting by translation 			&   \\
$H_k(BT_r)$ 		& $W_r$, acting by functoriality 			& $r\geqs 2k$*\\
$H_k(J(T_r))$ 		& $W_r$, acting by functoriality  			 & $r\geqs 2k$* \\
\hline
\hline
$H_k(\Hom(\Z^n,G)_1)$ 		& $S_n$, permuting coordinates & $n\geqslant 2k$\\
$H_k(\Rep(\Z^n,G)_1)$ 		& $S_n$, permuting coordinates & $n\geqslant 2k$\\
\hline

\

\end{tabular}
\caption{Representation stability bounds.\\
*These bounds are established in all cases \e{except} $G_r = SU(r)$.}
\label{table: rep stability results}

\end{table}


\

\begin{table}[ht!]
\centering

\renewcommand{\arraystretch}{1.5}

\begin{tabular}{llr}
\hline
\textit{Homology sequence} & \textit{Space description} & \textit{Stable range}\\
\hline
\hline

$H_k(\Sym^r(X))$ 				& $r$--fold symmetric product of $X$    	& $r\geqslant k$\\
\hline
$H_k(\Hom(\Z^n,G_r)_1)$ 		& ordered commuting $n$--tuples in $G_r$ &  $r  - \lfloor \sqrt{r} \rfloor \geqs k$\\
$H_k(\Rep(\Z^n,G_r)_1)$ 		& $G_r$--character variety of $\Z^n$& $r\geqslant k$\\
$H_k(\Bcom (G_r)_1)$ 		& classifying space $|\Hom(\Z^\bullet,G_r)_1|$ & $r  - \lfloor \sqrt{r} \rfloor 																	\geqs k$  \\
$H_k(\Comm(G_r)_1)$ 			& James-type construction &  $r  - \lfloor \sqrt{r} \rfloor\geqs k$ \\
$H_k(\Bcom (G_r)_1/G_r)$ 	& classifying space $|\Hom(\Z^\bullet,G_r)_1/G_r|$ & $r\geqslant k$\\
$H_k(\Comm(G_r)_1/G_r)$ 		& James-type construction &  $r \geqs k$ \\
\hline
$H_k^{G_r}(\Hom(\Z^n,G_r)_1)$ 		& ordered commuting $n$--tuples in $G_r$ & $r\geqslant k$\\
$H_k^{G_r}(\Comm(G_r)_1)$ 			& James-type construction  & $r\geqslant k$\\
$H_k^{G_r}(\Bcom (G_r)_1)$ 		& classifying space $|\Hom(\Z^\bullet,G_r)_1|$ & $r\geqslant k$\\
\hline
\hline
$H_k(\Hom(\Z^n,G)_1/S_n)$ 		& unordered commuting $n$--tuples in $G$ & $n\geqslant k$\\
$H_k(\Rep(\Z^n,G)_1/S_n)$ 		& unordered $G$--character variety of $\Z^n$ & $n\geqslant k$\\
$H_k^G(\Hom(\Z^n,G)_1/S_n)$ 		& unordered commuting $n$--tuples in $G$ & $n\geqslant k$\\
\hline

\

\end{tabular}
\caption{Stability bounds for rational (equivariant) homology groups.}
\label{table: all sequences of spaces}

\end{table}

\newpage
 
\begin{rmk}\label{rmk: SO} 
While our methods naturally lead us to divide the special orthogonal groups into two sequences (and the stability bounds are most easily stated in this way),  the entire sequence $\{\Hom(\Z^n, \SO(m))\}_{m\geqs 2}$ in fact satisfies strong homological stability. 
The stabilization maps are induced by block sum with the identity matrix, and 
in the stable range, composing any two adjacent maps in the sequence
\begin{eqnarray*}H_k (\Hom(\Z^n, \SO(m-1))_1)\maps H_k (\Hom(\Z^n, \SO(m))_1)\qquad \qquad\qquad \qquad\\
\qquad \qquad \qquad \maps H_k (\Hom(\Z^n, \SO(m+1))_1)\maps H_k (\Hom(\Z^n, \SO(m+2))_1)
\end{eqnarray*}
yields an isomorphism; hence each of these maps is in fact an isomorphism. 
Similar remarks apply to the other functors considered in the paper, as well as to the complexifications $\SO(2m, \C)$ and to the Spin groups.
\end{rmk}

\subsection*{Acknowledgments} 
The authors would like to thank  Jenny Wilson and John Wiltshire-Gordon for helpful conversations regarding representation stability, Tom Baird for suggesting the strategy in Lemma~\ref{lem: translation=conjugation}, Sean Lawton for helpful conversations about Lie theory, Fred Cohen for his feedback on a technical aspect, Jeremy Miller for telling us about Steenrod's result \cite{steenrod}, Bernardo Villarreal for comments on the first version, and Sarkhan Badirli for his assistance with Matlab computations. We also thank Mahir Can and Soumya Banerjee for asking about stability for equivariant homology. Finally, we thank the referees, whose comments helped to improve the exposition.

\section{Homology of commuting tuples}\label{Hom}

In this section we review work of Baird~\cite{baird2007cohomology} on the (co)homology of the spaces $\Hom(\Z^n, G)$, and explain how Baird's results interact with homomorphisms between Lie groups.

For a topological group $G$ and a finitely generated discrete group $\pi$, a group
representation $\rho : \pi \to G$ is completely determined by its image on a set of generators
$g_1,\dots,g_n \in \pi.$ That is, the association of $\rho$ with the $n$--tuple
$(\rho(g_1),\dots,\rho(g_n)) \in G^n$ gives an inclusion 
$\Hom(\pi,G) \injects G^n$, and the subspace topology on $\Hom(\pi, G)$ is independent of the choice of generating set (see \cite[\S 2]{villarreal2017cosimplicial} for a direct proof). The group $G$ acts by conjugation on $\Hom(\pi,G)$, and the quotient space is denoted by $\Rep(\pi, G) = \Hom(\pi,G)/G$.

\subsection{The conjugation map}
 
Throughout this section, $G$ will be a compact and connected Lie group, $T\leqs G$ a maximal torus, and $W = N(T)/T$ its Weyl group. 
Work of Baird~\cite{baird2007cohomology} allows us to describe the homology of the identity components
$\Hom(\Z^n, G)_1$ and $\Rep(\Z^n, G)_1$ (and their unordered versions) as Weyl group invariants in the homology of some related spaces, and we will explain how various maps between spaces of commuting elements behave on homology.

 The conjugation map $G \times T \to G$  given by $(g,t)\mapsto gtg^{-1}$
defines a surjection onto $G$. Baird~\cite{baird2007cohomology} generalized this to a map
\begin{align*}\label{eqn: conj. GxTn map to Hom}
\begin{split}
G\times T^n & \to \Hom(\Z^n,G) \\
(g,t_1,\dots,t_n) &\mapsto (gt_1 g^{-1},\dots,gt_n g^{-1}).
\end{split}
\end{align*}
This map is certainly not a surjection in general (since its domain is always path connected). However, the image of
this map is precisely the path component of the trivial representation $\Hom(\Z^n,G)_1$,
as shown by Baird. 
This map factors through the quotient by the normalizer of $T$ to give a map
$$
G \times_{NT} T^n \to \Hom(\Z^n,G)_1,
$$
where $NT$ acts by right multiplication on $G$ and diagonally by 
conjugation on $T^n$. Moreover, the conjugation action of $T\leqs NT$  on $T^n$ is trivial, so we have a homeomorphism $G \times_{NT} T^n \srm{\isom} G/T \times_{W} T^n$, and we obtain a map
\begin{equation}\label{phi}
\phi = \phi_n \co G/T\times_W T^n \to \Hom(\Z^n,G)_1,
\end{equation}
which we call the \e{conjugation map}.
The stability properties of the spaces mentioned in the introduction 
are studied mainly via this map and others derived from it.

\begin{thm}[Baird] \label{Baird}
The map $\phi$ induces an isomorphism
$$H_* (G/T \times_{W} T^n; \bbQ)  \srm{\isom}  H_*(\Hom(\Z^n,G)_1; \bbQ),$$
and torsion in 
$H_*(\Hom(\Z^n,G)_1; \Z)$ has order dividing $|W|$. 
\end{thm}

Baird stated his result in cohomology, but by the Universal Coefficient Theorem, the homological and cohomological versions with rational coefficients are equivalent. 

One can use the map (\ref{phi}) to give an analogous model for the identity component $\Rep(\Z^n,G)_1$. The maps
\begin{equation}
\label{in}   T^n = \Rep(\bbZ^n, T) \xmaps{i=i_n} \Rep(\Z^n,G)_1
\end{equation}
induced by the inclusion $T\injects G$  are invariant with respect to the diagonal conjugation action of $W$ on $T^n$, and hence induce maps
\begin{equation}  
\label{j}  T^n/W \to \Rep(\Z^n,G)_1.
\end{equation}

\begin{thm}\label{Stafa}
The map (\ref{j}) is a homeomorphism. 
Consequently,
we have an isomorphism
$$H_* (T^n/W; \bbQ)  \srm{\isom}   H_* (\Rep(\Z^n,G)_1; \bbQ).$$ 
\end{thm}

This was first observed by Baird~\cite{baird2007cohomology}, and a proof is given in 
\cite[Theorem 5.1]{stafa2017poincare}.

\begin{rmk}\label{rmk: phi natural} The isomorphisms in Theorems~\ref{Baird} and~\ref{Stafa} are natural in the following senses. Let $f\co G\to G'$ be a map of Lie groups and assume there exist maximal tori $T\leqs G$, $T'\leqs G'$  such that $f(T) \subset T'$ and $f(NT) \subset NT'$. Then we have a commutative diagram 
\begin{center}
\begin{tikzcd}
G/T \cross_W T^n \arrow{r}{\phi} \arrow{d}  & 
	\Hom(\Z^n, G)_1 \arrow{d} 
	 \\
G'/T' \cross_{W'} (T')^n\arrow{r}{\phi}	 &
\Hom(\Z^n, G')_1, 
	 \\
\end{tikzcd}
\end{center}
where $W =NT/T$ and $W' = NT'/T'$ are the Weyl groups, and the vertical maps are induced by $f$. Similarly, we have 
a commutative diagram 
\begin{center}
\begin{tikzcd}
T^n/W \arrow{r}{\isom}  \arrow{d}  & 
	\Rep(\Z^n, G)_1 \arrow{d} 
	 \\
(T')^n/W'  \arrow{r}{\isom}  &
\Rep(\Z^n, G')_1, 
	 \\
\end{tikzcd}
\end{center}
where the horizontal maps are the homeomorphisms (\ref{j}), and again the vertical maps are induced by $f$.
\end{rmk}

\section{Homology of finite quotients}

It is well-known that (under suitable conditions) if $H$ is a finite group acting on a space $X$, then there is an isomorphism
\begin{equation}\label{eqn: fin-quot} H^*(X/H; \bbQ) \isom H^*(X; \bbQ)^H,\end{equation}
where the right-hand side denotes the subring of $H$--invariant elements. 
This applies in particular to Theorems~\ref{Baird} and~\ref{Stafa}.
Here we explain a homological version of (\ref{eqn: fin-quot}), and describe how these results interact with equivariant maps.

Recall that a space $X$ is semi-locally contractible if every open set $U\subset X$ has a covering by open sets $U_i\subset U$ for which the inclusion maps $U_i\injects U$ are null-homotopic. 

 \begin{defn} We say that a space $X$ is \e{good} if it is homotopy equivalent to a semi-locally contractible space.
 We say that an action of a finite group
 $H$ on a space $X$  is \e{good} if both $X$ and $X/H$ are good.
\end{defn} 

Note that every CW complex is locally contractible and hence also semi-locally contractible. In all of the situations we encounter in this paper, the actions will be good because $X$ and $X/H$ will have the homotopy types of CW complexes. This will often be proven using the fact that if $n\goesto X_n$ is a proper simplicial space in which each $X_n$ has the homotopy type  of a CW complex, then so does the geometric realization $|X_\bullet|$ (see May~\cite[Appendix]{May-EGP}).

\begin{prop}\label{coh-quot}
Consider a good action of a finite group $H$ on a space $X$, and let $\k$ be a field whose characteristic does not divide $|H|$. Let $\pi\co X\to X/H$ denote the quotient map. Then $\pi^* \co H^*(X/H; \k) \to H^*(X; \k)$ has image contained in the subspace of invariants $H^*(X; \k)^H$, and in fact $\pi^*$ induces  an isomorphism
$$H^*(X/H; \k) \srm{\isom} H^*(X; \k)^H.$$
Similarly, the map $\pi_* \co H_* (X; \k) \to H_* (X/H; \k)$ restricts to an isomorphism
$$H_* (X; \k)^H \srm{\isom} H_* (X/H; \k).$$
\end{prop}
\begin{proof}  
Bredon~\cite[Theorem 19.2]{bredon2012sheaf} proves a version of this result for sheaf cohomology, with coefficients in the constant sheaf $\k$, for arbitrary actions of finite groups. For semi-locally contractible spaces, the sheaf cohomology and the singular cohomology are naturally isomorphic by Sella~\cite{Sella}, and this fact also holds for spaces homotopy equivalent to  semi-locally contractible spaces because both sheaf cohomology and singular cohomology are homotopy invariant (for sheaf cohomology, see~\cite[Corollary II.11.13]{bredon2012sheaf}).
Hence for good actions, the result for sheaf cohomology is equivalent to the result for singular cohomology.\footnote{For our purposes, the more classical fact that sheaf cohomology agrees with singular cohomology for locally contractible, paracompact Hausdorff spaces is sufficient, since in all the situations where we apply this result the spaces in question have the homotopy types of CW complexes.} 

We now deduce the homological case from the cohomological case. By the Universal Coefficient Theorem for cohomology, there is a commutative diagram of the form
\begin{center}
\begin{tikzcd}
H^*(X/H) \arrow{r}{\isom} \arrow{d}{\pi^*}   & 
	\Hom(H_*(X/H), \k) \arrow{d}{(\pi_*)^*} 
	 \\
 H^*(X)^H \arrow{r}{\isom}	 &
 \Hom(H_*(X), \k)^H 
	 \\
\end{tikzcd}
\end{center}
where $\Hom( - , \k)$ denotes the $\k$--linear dual.
Note here that naturality of the isomorphism 
$$H^*(X) \srm{\isom} \Hom(H_* (X), \k)$$
in the Universal Coefficient Theorem  implies that this isomorphism is equivariant with respect to the action of $H$, and hence restricts to an isomorphism between the subspaces of $H$--fixed points. Since the map $\pi^*$ is an isomorphism, we conclude that 
$$(\pi_*)^*\co \Hom(H_*(X/H), \k) \to \Hom(H_*(X), \k)^H $$
is an isomorphism as well.

We need the following lemma. To simplify notation we set $V^* := \Hom(V, \k)$.

\begin{lem} Let $H$ be a finite group, and $\k$ a field of characteristic not dividing $|H|$. Let $V$ and $W$ be  representations of $H$ over $\k$, and let 
$p\co V\to W$ be an $H$--equivariant map. Then the induced map $p^*\co W^* \to V^*$ is $H$--equivariant, where $H$ acts on $V^*$ by $(h\cdot l) (v) := l(h^{-1} v)$. 
By equivariance,  $p$ and $p^*$ restrict to  maps $p^H \co V^H\to W^H$ and $(p^*)^H \co (W^*)^H \to (V^*)^H$.

The map $p^H$ is injective if and only if $(p^*)^H$ is surjective, and $p^H$ is surjective if and only if $(p^*)^H$ is injective. In particular, $p^H$  is an isomorphism if and only if $(p^*)^H$ is an isomorphism.
\end{lem}

\begin{proof} It suffices to show that the natural map $\phi\co (U^*)^H \to (U^H)^*$ sending an invariant linear functional $l$ to its restriction $l|_{U^H}$ is an isomorphism, because then $(p^*)^H$ and $(p^H)^*$ are naturally isomorphic, and in general a linear map is injective (respectively, surjective) if and only if its dual is surjective (respectively, injective).

The map $\phi$ is injective, since if
$l \in (U^*)^H$ is non-zero, then there exists $x\in U$ such that
 $l(x)\neq 0$, and then linearity and $H$--invariance of $l$ give
 $$l\left(\sum_{h\in H} h\cdot x\right) = \sum_{h\in H} l(h\cdot x) = |H| l(x) \neq 0$$
 (since the characteristic of $\k$ does not divide $|H|$).  
 To prove surjectivity of $\phi$, consider a linear functional $l\co U^H \to \k$. Choose a splitting $q\co U\to U^H$ of the inclusion $U^H\injects U$ and set
 $$\wt{l} :=  \dfrac{1}{|H|} \sum_{h\in H} h\cdot (l\circ q).$$
Then $\wt{l}$ is an $H$--invariant extension of $l$, so $\phi(\,\wt{l}\,) = l$.
\end{proof}

We now complete the proof of Proposition~\ref{coh-quot}.
The fact that $\pi_*$ restricts to an isomorphism $H_*(X)^H \to H_* (X/H)$ follows by applying the Lemma to the case $V = H_* (X)$, $W = H_* (X/H)$ (with trivial $H$--action) and $p = \pi_*$. 
\end{proof}

In order to describe how Proposition~\ref{coh-quot} interacts with equivariant maps of spaces, we will need the following construction.

\begin{defn}\label{alpha}
Let $H$ be a finite group, and let $V$ be a $\k[H]$--module, with $\k$ a field of characteristic not dividing $|H|$. The averaging map
$$\alpha\co V\to V^H$$
is the $\k$--linear map
defined by
$$\alpha (v) = \dfrac{1}{|H|} \sum_{h\in H} h\cdot v.$$
To simplify notation, we set  $\ol{v} := \alpha(v)$.
\end{defn}

We now list some basic properties of the averaging map.
 
\begin{lem}\label{avg-lem} Let $\alpha\co V\to V^H$ be the averaging map  of the $\k[H]$--module $V$ $($where $\k$ is a field of characteristic not dividing $|H|$$)$. Then
\begin{enumerate}
\item $\alpha|_{V^H}$ is the identity $($so $\alpha$ is idempotent with image $V^H$$)$;
\item $\alpha(h\cdot v) = \alpha (v)$ for all $h \in H$, $v\in V$;
\item $\alpha$ is the orthogonal projection of $V$ onto $V^H$ with respect to any $H$--invariant inner product $($that is, $\langle v - \alpha (v), \alpha (v)\rangle = 0$ if $\langle \,, \, \rangle$ is $H$--invariant$)$.
\end{enumerate}
\end{lem}

\begin{prop}\label{coh-quot2}
Let $H_1$ and $H_2$ be finite groups acting on spaces $X_1$ and $X_2$, and assume the actions are good. Let $f\co X_1\to X_2$ be a map that is equivariant with respect to a homomorphism $\phi\co H_1 \to H_2$ $($meaning that $f(h\cdot x) = \phi(h)\cdot f(x)$ for all $h\in H, x\in X$$)$. Let $\ol{f} \co X_1/H_1\to X_2/H_2$ be the map induced by $f$.
Let $\k$ be a field of characteristic not dividing $|H_1|$ or $|H_2|$.
Then under the isomorphisms in Proposition~\ref{coh-quot}, the map
$$\ol{f}^*\co H^*(X_2/H_2; \k) \to H^*(X_1/H_1; \k)$$ 
corresponds to 
$$f^*\co  H^*(X_2; \k)^{H_2} \to H^*(X_1; \k)^{H_1}$$ 
$($in particular, $f^*$ maps $H_2$--invariant classes to $H_1$--invariant classes$)$, while
the map 
$$\ol{f}_*\co H_*(X_1/H_1; \k) \to H_*(X_2/H_2; \k)$$
corresponds to 
$$\ol{ f_*} = \alpha\circ f_* \co H_*(X_1; \k)^{H_1} \to H_*(X_2; \k)^{H_2}.$$ 
\end{prop}
\begin{proof} Let $\pi_i \co X_i\to X_i/H_i$ be the quotient map. First we consider $\ol{f}^*$. 
Say $x_2\in H^* (X_2; \k)^{H_2}$. By Proposition~\ref{coh-quot}, there is a unique class $\ol{x}_2\in H^*(X_2/H_2; \k)$ satisfying $\pi_2^* \ol{x}_2  = x_2$. Then we have
$$\pi_1^* \ol{f}^* \ol{x}_2 = f^* \pi_2^* \ol{x}_2 = f^* x_2,$$
as desired.

Next, given $x_1 \in H_*(X_1)^{H_1}$, we want to show that $ \ol{f}_* (\pi_1)_* x_1  =(\pi_2)_* \alpha (f_* x_1)$. We have
\begin{align*}(\pi_2)_* \alpha (f_* x_1) & = \dfrac{1}{|H_2|} \sum_{h\in H_2} (\pi_2)_* h_* f_* x_1
=  \dfrac{1}{|H_2|} \sum_{h\in H_2} (\pi_2)_* f_* x_1\\
& = (\pi_2)_* f_* x_1 = \ol{f}_* (\pi_1)_*  x_1
 \end{align*}
as desired.
\end{proof}

\section{$\FI_W$--modules and $\FI_W\#$--modules}\label{FIW-modules}
 
Here we provide a quick introduction to the theory of $\FI_W$--modules, as developed by Church, Ellenberg, and Farb \cite{church2015fi} in type A, and later extended by J. Wilson \cite{wilson2014fiw}. 
More details can be found in these sources. 

Let $W_n$ denote the $n^{\textrm th}$ Weyl group of classical type A, B, C, or D. These correspond to the
symmetric group $S_n$ for type A, the signed permutation group $B_n = \Z_2 \wr S_n$ 
for both types B and C, and the even signed permutation group $D_n \leqs B_n$ for D. 
In this section, we write $\Z_2$ as the group $\{-, +\}$ with $+$ as identity element, so that elements of $\Z_2 \wr S_n$ are pairs consisting of a permutation and a list of $n$ signs; note that $\Z_2 \wr S_n$ is in bijection with the set of permutations $\sigma$ of the $2n$--element set $\{\pm 1, \ldots, \pm n\}$ satisfying $\sigma (\{-t, t\}) = \{\sigma(t), -\sigma(t)\}$.
The even signed permutation group $D_n$ can be seen
as the index 2 subgroup of $B_n$ obtained as the kernel of the projection 
$\Z_2 \wr S_n \to \Z/2\Z$ counting the number of minus signs modulo 2.
We have standard inclusions $j_{n} \co W_n \injects W_{n+1}$ sending a signed permutation $\sigma$ of $\n$ to the signed permutation of ${\bf n+1}$ that agrees with $\sigma$ on $\n\subset {\bf n+1}$ and is the identity on $\{\pm (n+1)\}$.

As we saw in Section~\ref{Hom}, the homology of spaces of commuting elements can be described 
using representations of Weyl groups and this perspective will be used throughout the paper.

\subsection{$\FI_W$--modules} For $W =$ A, B, C, or D,
the \textit{category $\FI_W$} is the category whose objects are the sets $\n=\{\pm 1,\dots,\pm n\}$
and $\0=\emptyset$, 
indexed by the natural numbers, and whose morphisms are the injections
$\phi: \m \to \n$ such that $\phi(-t)=-\phi(t)$ for all $t \in \m$ (here we use the convention that $--i = i$), together with an extra condition depending on the type ($A$, $B$, $C$ or  $D$).
Namely, the category $\FI_D$ is the subcategory defined by the additional condition that 
an isomorphism must reverse an even number of signs (that is, $|\{i \in \{1, \ldots, n\} \,:\, \phi(i) = -i\}|$ must be even), 
$\FI_A$ is the subcategory consisting of morphisms that
preserve all signs, and $\FI_{B} = \FI_{C}$ is the category with no further restriction on morphisms (following Wilson, we will use the notation $\FI_{BC} := \FI_{B} = \FI_{C}$).
This implies that $\Aut(\n)= S_n$ in type A,
 $\Aut(\n)= \Z_2 \wr S_n = B_n$ in types B and C,  and $\Aut(\n)= D_n$ in type D.

Note that
there are inclusions of categories $\FI_A \hookrightarrow \FI_D \hookrightarrow \FI_{BC}$.
The notation $\FI_W$ will stand for one of the above three categories.

\begin{rmk}\label{trans} Elementary arguments show that each of the categories $\FI_W$ has the property that $\Aut(\n)$ acts \e{transitively} on the set of morphisms $\Mor(\m, \n)$ (via composition). In particular, every morphism $\m\to \n$ can be written in the form $\sigma \circ i_{mn}$, where $i_{mn} \co \m \to \n$ is the (unique) morphism satisfying $i(j) = j$ for $1\leqs j\leqs m$ and $\sigma$ is a signed permutation lying in $\Aut(\n)$. Note here that $\Mor(\m, \n)$ is empty unless $\m \leqs \n$.
\end{rmk}

\begin{defn}[$\FI_W$--modules]\label{defn: FI_W module}
An \textit{$\FI_W$--module} $V$ over a ring $R$ is a  functor 
$$V\co \FI_W \to {{\rm \bf Mod}_R }$$
to the category of $R$--modules.  Given a morphism $\phi\co \n\to \m$ in $\FI_W$, we will sometimes abbreviate $V(\phi)$ to $\phi_*$ when $V$ is clear from context.

A sub--$\FI_W$--module of $V$ is an $\FI_W$--module $V'$ with $V'(\n) \leqs V(\n)$ a submodule for each $n$, and $V'(\phi) = V(\phi)|_{V'(\n)}$ for each morphism $\phi\co \n\to \m$ in $\FI_W$.

An \textit{$\FI_W$--space} $X$ is a functor 
$$X\co \FI_W \to \Top.$$
\end{defn}

Note that the homology of an $\FI_W$--space with coefficients in a ring $R$  is an $\FI_W$--module over $R$.

\begin{rmk} \label{FI-rmk} A \e{consistent sequence}   of $W_n$--representations is a sequence of $W_n$--representations $V_n$  together with structure maps $f_n \co V_n\to V_{n+1}$ that are  equivariant with respect to the inclusions $j_n\co W_n \injects W_{n+1}$. 
Remark~\ref{trans} implies that each $\FI_W$--module  $V$ is determined by its underlying consistent sequence of
$W_n$--representations $V_n:=V(\n)$,  with structure maps
$$(i_n)_* = V(i_{n})\co V_n \to V_{n+1},$$
where $i_n \co \n \to ({\bf n+1})$ satisfies $i_n (j) = j$ for $1\leqs j\leqs n$.
Similarly, an $\FI_W$--space $X$ is determined by its sequence of $W_n$--spaces $X_n = X(\n)$, together with the structure maps $(i_n)_*$. 
\end{rmk}

Not all consistent sequences of $W_n$--representations (or $W_n$--spaces) extend to $\FI_W$--modules (or $\FI_W$--spaces). However, Wilson~\cite[Lemma 3.4]{wilson2014fiw} provides a  characterization of those consistent sequence that \e{do} extend. Wilson's proof immediately extends to $\FI_W$--objects in any category, and we will apply the result to 
$\FI_W$--spaces in Section~\ref{Lie}. To state the result in full generality, we define a consistent sequence 
$C$ of $W$--representations in a category $\mathcal{C}$ to be a sequence of objects $C_n$ in $\mathcal{C}$ together with homomorphisms $W_n\to \textrm{Aut}_\mathcal{C} (C_n)$ (denoted $\tau\goesto \tau_*$) and morphisms  $\phi_{n}\co C_n\to C_{n+1}$ satisfying the equivariance condition $\phi_{n} \circ \tau_* = (j_{n} (\tau))_* \circ \phi_n$. Note that every functor $\FI_W\to \mathcal{C}$ has an underlying consistent sequence of $W_n$--representations.

\begin{lem}[Church-Ellenberg-Farb,Wilson] \label{lem: extension} 
Let $C$ be a consistent sequence of $W_n$--representations in the category $\mathcal{C}$. Then $C$ extends to an $\FI_W$--object $\wt{C} \co \FI_W\to \mathcal{C}$ $($with underlying consistent sequence $C$$)$ if and only if $\tau_* \circ \phi_{n-1} \circ \cdots \circ \phi_m =  \phi_{n-1} \circ \cdots \circ \phi_m$ for all $\tau\in W_n$ satisfying $\tau (i) = i$ for $i\leqs m$.
\end{lem}

As discussed in Section~\ref{Lie}, our key examples of $\FI_W$--modules will be produced from the classical sequences of Lie groups using Lemma~\ref{lem: extension}.

\begin{defn}[Finite Generation]\label{defn: finite generation}
We say that an $\FI_W$--module $V$ is \e{generated in stage $\leqslant n$} (or more briefly, stage $n$)
if for each $m\geqs n$, the union of the images of $V_n$ under $\FI_W$--morphisms $\n\to \m$ spans $V_m$. (We note that $V$ is generated in stage $\leqs n$ if and only if $V_n$ is not contained  in any proper sub--$\FI_W$--module of $V$.)

We say that $V$ is finitely generated if it is generated in stage $n$ for some $n$ and $V_n$ is finite-dimensional.
\end{defn}

We note that the term \e{degree} is often used in place of \e{stage} in the previous definition. We use the term stage to avoid confusion with (co)homological degree.

\begin{rmk}\label{fg-rmk}
All of the $\FI_W$--modules $V$ considered in this article will be finite-dimensional, in the sense that each $V_n$ is finitely generated over the coefficient ring $R$. In this setting, if $V$ is generated in stage $n$, then it is automatically finitely generated as well.
\end{rmk}

We will use the following result of Wilson repeatedly.

\begin{prop}[{\cite[Proposition 5.2]{wilson2014fiw}}]\label{prop: tensor}
If $V$ and $W$ are finite-dimensional $FI_W$--modules that are  generated in stages $m$ and $n$,
respectively, then $V \otimes W$ is generated in stage $($at most$)$ $m+n$. 
 \end{prop}

\subsection{Representation stability}

Now we give the precise definition of representation stability and uniform representation
stability. These notions were first defined by Church and Farb \cite{church2013representation}
for a sequence of $G_n$--representations $V_n$ (for some sequence of groups $G_n$), mainly for the purpose of studying
stability properties of representations of symmetric groups. 
Afterwards we will explain how these notions relate to homological stability. 
 
 A partition of a non-negative 
 integer $n$ is a sequence of non-negative integers
$$\lambda:=(\lambda_1\geqslant  \lambda_2\geqslant  \cdots \geqslant\lambda_l \geqslant 0)$$
with $\sum \lambda_i =n$. We write $\lambda \vdash n$ or $|\lambda|= n$ to indicate that $\lambda$ is a partition of $n$. Note that for $n=0$, we allow the empty partition. 
There is a well-known bijection between partitions $\lambda \vdash n$ and irreducible representations of  the symmetric group $S_n$ over fields of characteristic zero; we denote the representation corresponding to the partition 
$\lambda \vdash n$  by $V_\lambda.$ For instance if 
$\lambda=(n)$ then $V_\lambda$ is the \textit{trivial} one-dimensional representation of $S_n.$
If $\lambda:=(\lambda_1,\dots,\lambda_t) \vdash k$ is a partition of $k$, 
then for any $n \geqslant k + \lambda_1$ we may define a partition
$
\lambda[n]:=(n-k,\lambda_1,\dots,\lambda_t).
$
Now we can define the irreducible $S_n$--representation $V(\lambda)_n$ as
$$
V(\lambda)_n := V_{\lambda[n]}.
$$
Note that for each $\lambda$, $V(\lambda)_n$ is  defined only for sufficiently large $n$ (namely $n\geqs |\lambda|+\lambda_1$).
This process associates  an infinite sequence of irreducible representations of symmetric groups to each partition $\lambda$ of a non-negative integer. In particular, for the empty partition $\lambda = \emptyset$, this sequence is $n\goesto V(\emptyset)_n = V_{(n)}$, the sequence of trivial one-dimensional representations of $S_n$. 

Over fields of characteristic zero, the hyperoctahedral group $B_n = \Z_2 \wr S_n$ has irreducible representations indexed
by double partitions $\lambda = (\lambda^+,\lambda^-)$ of $n$,
where $\lambda^+ \vdash m$ and $\lambda^- \vdash n-m$ for some $m\leqs n$. As before, we write 
$V_\lambda$ for the representation associated to $\lambda$.
In general, for a double partition $\lambda=(\lambda^+,\lambda^-)$ of $k$ and for
$n \geqslant k + \lambda_1^+$ we define the partition 
$\lambda[n]:=((n-k,\lambda^+),\lambda^-)$. Now we can define the irreducible
$B_n$--representation
$$
V(\lambda)_n := V_{\lambda[n]}.
$$
Once again, when $\lambda^+ = \lambda^- = \emptyset$, we obtain the sequence of trivial one-dimensional representations of $B_n$. 

Finally, we consider the case of the even-signed permutation group $D_n\leqs B_n$. For each 
double partition $\lambda$ as above, we set
$$V(\lambda)_n := \textrm{Res}^{B_n}_{D_n} V_{\lambda[n]}.$$
Wilson \cite[Proposition 3.30]{wilson2014fiw} showed that if $V$ is a
finitely generated $\FI_D$--module, then $V_n$ is the restriction of a
$B_n$--representation for sufficiently large $n$, and hence (over a field of characteristic zero) 	admits a (unique) decomposition as a sum of representations of the form $V(\lambda)_n$. (It should be noted that $V(\lambda)_n$ is not always irreducible as a $D_n$--module.) Once again, the empty double partition gives rise to the sequence of 1--dimensional trivial representations of $D_n$.

\begin{defn}[Representation stability] \label{def: rep stability}
Let $V$ be an $\FI_W$--module over a field $\k$ of characteristic zero. 
We say that $V$ is
\textit{representation stable} if it satisfies the following conditions:
\begin{enumerate}
\item[(I)] {\it Injectivity}: The maps $(i_n)_*: V_n \to V_{n+1}$ are injective, 
							for all sufficiently large $n$;
\item[(II)] {\it Surjectivity}: The image $(i_n)_*(V_n)$ generates $V_{n+1}$ as a $\k [W_{n+1}]$--module, 
							for all sufficiently large $n$;
\item[(III)] {\it Multiplicities}: For all sufficiently large $n$, there exists an isomorphism of $W_n$--representations
$$
V_n \isom \bigoplus_{\lambda} c_{\lambda,n} V(\lambda)_n,
$$
and for each $\lambda$, the multiplicity $c_{\lambda,n}$ of $V(\lambda)_n$ is eventually independent 
of $n$. 
\end{enumerate}
\end{defn}

\begin{rmk} The decompositions in (III) above are unique if they exist, and so the multiplicities $c_{\lambda, n}$ are well-defined. In types A, B, and C, such a decomposition always exists, since all irreducible representations of $W_n$ are of the form $V(\lambda)_n$ in these cases (over fields of characteristic zero).
\end{rmk}

\begin{defn}[Uniform Representation Stability]
Let $V=\{V_n,\phi_n\}$ be a representation stable $\FI_W$--module with
$c_{\lambda,n}$ constant for all $n\geqslant N_\lambda.$ 
Then $V$ is called \textit{uniformly representation stable} 
if $N=N_\lambda$ can be chosen independently of $\lambda$.
In this case, we say that $V$ has \e{stable range} $n\geqs N$.
\end{defn}

Wilson~\cite{wilson2014fiw} shows that an $\FI_W$--module $V$ is
uniformly representation stable if and only if it is finitely generated. Here we are mainly
interested in the ``if'' part of the statement.

\begin{thm}[{{\cite[Theorem 4.27, 4.28]{wilson2014fiw}}}]\label{thm: fin. gen. implies uniform rep. stability}
Let $\k$ be a field of characteristic 0 and let $V$ be a finitely generated
$\FI_W$--module over $\k$. Then $V$ is uniformly representation stable.
\end{thm}

Given an $\FI_B$--module $V$, composing $V\co \FI_B \to  {{\rm \bf Mod}_R }$ with the inclusion 
$\FI_D\injects \FI_B$ yields the restricted $\FI_D$--module $V|_D$.
The following corollary is immediate from the definition of $V(\lambda)_n$ in type D.

\begin{cor}\label{stable-range-D}
Let $V$ be a an $\FI_B$--module that is uniformly representation stable with stable range $n\geqs N$. Let $V|_D$ be the restriction of $V$ to an $\FI_D$--module $($that is, the composition of the functor $V$ with the inclusion $\FI_D\injects \FI_B$$)$. Then $V|_D$ is also uniformly representation stable with stable range $n\geqs N$.
\end{cor}

\begin{rmk}\label{rmk: iso-comp} Since the trivial representation of $W_n$ corresponds to $V(\emptyset)_n$, 
Theorem~\ref{thm: fin. gen. implies uniform rep. stability} implies that in a finitely generated $\FI_W$--module, the dimensions of the isotypical components of the trivial representation eventually stabilize. We will deduce a stronger version of this statement in Proposition~\ref{avg}.
\end{rmk}

\subsection{$\FI_W \#$--modules}

We now introduce extensions of the categories $\FI_A = \FI$ and $\FI_{BC}$, due to Church--Ellenberg--Farb~\cite{church2015fi} and Wilson~\cite{wilson2014fiw}, respectively. These categories can be used to establish bounds on the stable range of a finitely generated module.

Define $\FI_{BC} \#$ to be the category whose objects are the based sets $\n_0 = \n\cup \{0\}$, $n=1, 2, \ldots$
(with $0$ as the basepoint), and whose morphisms are based maps
$$
\phi\co \n_0 \to \m_0
$$
that are injective on $\n \setminus \phi^{-1} (0)$ and satisfy $\phi(\{i, -i\}) = \{\phi(i), -\phi(i)\}$ (we set $-0 = 0$).  

The category $\FI_{A}\#$ is the subcategory of $\FI_{BC} \#$ consisting of those morphisms that preserve signs (that is, $\phi\co \n_0 \to \m_0$ lies in $\FI_{A}\#$ if $\phi(i) \in \{0,1, \ldots, m\}$ for each $i\in\{1, \ldots, n\}$). This category is equivalent to the category $\FIs$ introduced in~\cite{church2015fi}, whose objects are finite sets and whose morphisms $S\to T$ are \e{partially defined} injections
$$S \supset A \stackrel[\isom]{\psi}{\maps} B \injects T.$$
Note that each morphism $\phi\co \n_0 \to \m_0$ in $\FI_{BC}\#$  restricts to a partially defined injection $\{1, \ldots, n\}\supset A \stackrel[\isom]{\phi'}{\maps} B\subset \{1, \ldots, m\}$, where $A = \{i: \phi(i)\neq 0\}$. The equivalence of categories $\FI_A\#\to \FIs$ carries $\n_0$ to $\{1, \ldots, n\}$ and takes $\phi\co \n_0 \to \m_0$ to $\phi'$. (Composition in $\FI\#$ is defined so as to make this a functor.)

As observed by Wilson, there is no natural analogue of these extensions in type D, so we do not define $\FI_D \#$. As such, the symbol $\FI_W \#$ will refer to one of the above two categories.

There are canonical embeddings of categories $\FI_W \injects \FI_W\#$, sending $\n\goesto \n_0$ and sending a morphism $\phi\co \n\to \m$ to the unique base-point preserving extension $\phi_0 \co \n_0 \to \m_0$.

\begin{defn}[$\FI_W \#$--modules]
An $\FI_W\#${--\it module} over a ring $R$ is a functor
$$V: \FI_W\# \to {{\rm \bf Mod}_R }$$
 from $\FI_W\#$ to the category of $R$--modules.
\end{defn}

Note that each $\FI_W\#$--module has an underlying $\FI_W$--module, defined via restriction along the embedding $\FI_W \injects \FI_W\#$.

The following result is due to Church--Ellenberg--Farb~\cite{church2015fi} in the case of $\FI_A\#$--modules, and due to Wilson~\cite{wilson2014fiw} 
in the case of $\FI_{BC}\#$--modules.

\begin{thm}\label{sharp} Let $V$ be an $\FI_W\#$--module over a field of characteristic zero. If $V$ is finitely generated in stage $k$, then $V$ is uniformly representation stable with stable range 
$n\geqs 2k$.
\end{thm}

We are mainly interested in stability for subspaces of invariants, and in this case a better bound can be obtained. 
First let us recall the branching rules for induced 
representations of symmetric and hyperoctrahedral groups (see \cite{geck2000characters}).

Given a representation $V$ of $S_n$ over a ring $R$, let $V\boxtimes R$ denote the external tensor product of $V$ with the trivial representation of $S_k$ (so $V\boxtimes R$ is a representation of $S_n\cross S_k$). We will use similar notation for the hyperoctahedral groups $B_n$.

\begin{lem}\label{lem: branching rule for S_n}
Let $\k$ be a field of characteristic zero.
For each partition $\lambda$  of $n$, 
the $S_{n+k}$--representation induced by $V_\lambda$ is  
$$
\Ind_{S_n \times S_k}^{S_{n+k}} V_\lambda \boxtimes \k \cong \bigoplus_{\mu} V_\mu,
$$
where the direct sum is taken over those partitions $\mu$ of $n+k$ obtained from $\lambda$  by adding one box to each of $k$ different columns $($of the corresponding Young tableaux$)$.
\end{lem}

\begin{lem}\label{lem: branching rule for B_n} Let $\k$ be a field of characteristic zero.
For each double partition $\lambda=(\lambda^+,\lambda^-)$  of $n$,  
the $B_{n+k}$--representation induced by $V_{(\lambda^+,\lambda^-)}$ is  
$$
\Ind_{B_n \times B_k}^{B_{n+k}} V_{(\lambda^+,\lambda^-)} \boxtimes \k \cong 
		\bigoplus_{\mu^+} V_{(\mu^+,\lambda^-)},
$$
where the direct sum is taken over those partitions $\mu^+$ of $n+k$ obtained from $\lambda^+$  by adding one box to each of $k$ different columns $($of the corresponding Young tableaux$)$.
\end{lem}

We learned the following result from John Wiltshire-Gordon.

\begin{prop}\label{prop: stability isotypical component}
In types A, B, and C, if an $\FI_W$--module $V$ is generated in stage $d$, and $V$ extends to an $\FI_W \#$--module, then for $n\geqs d$ we have
$$\dim (V_n^{W_n}) = \dim (V_d^{W_d}).$$
\end{prop}
\begin{proof}
A complete classification of $\FI \#$--modules is given in \cite[Theorem 4.1.5]{church2015fi}, 
whereas  Wilson gives a classification of $\FI_B \#$--modules in \cite[Theorem 3.7]{Wilson-Math-Z}. 
In particular, an $\FI_W\#$--module $V$ can   
always
be written as 
\begin{equation} \label{eqn: decomp} \ds V \cong \bigoplus_{n \geqslant 0} M_W(U_n),
\end{equation}
where
$U_n$ is a representation of $W_n$, and the $\FI_W \#$--module $M_W (U_n)$ satisfies
$$M_W(U_n)_r = \Ind_{W_n \times W_{r-n}}^{W_{r}} U_n \boxtimes \Q$$ 
if $r\geqslant n$ and $M_W(U_n)_r = 0$ otherwise; moreover, $M_W (U_n)$ is generated in stage $n$.
The trivial representation corresponds to the partition $\lambda=(n)$ of $n$, whose corresponding Young tableaux is simply $n$ boxes aligned horizontally.
Using the  
branching rules above, one sees that if $U_n$ is an irreducible representation of
$W_n$ and 
$$\Ind_{W_n \times W_{k}}^{W_{n+k}} U_n \boxtimes \Q$$
contains a copy of the trivial representation of $W_{n+k}$, then $U_n$ must be trivial itself.
Next, observe that for any $W_n$--representations $A$ and $B$,
$$\Ind_{W_n \times W_{k}}^{W_{n+k}} (A\oplus B) \boxtimes \Q
\isom \left(\Ind_{W_n \times W_{k}}^{W_{n+k}} A \boxtimes \Q\right)
\oplus \left(\Ind_{W_n \times W_{k}}^{W_{n+k}} B \boxtimes \Q\right).$$
Decomposing an arbitrary representation $U_n$ of $W_n$ into irreducibles, we now see that the dimension of 
$$\left(M_W (U_n)_{n+k}\right)^{W_{n+k}} \isom \left(\Ind_{W_n \times W_{k}}^{W_{n+k}} U_n \boxtimes \Q\right)^{W_{n+k}}$$
agrees with that of $U_n^{W_n} = \left( M_W (U_n)_n\right)^{W_n}$ (for $k\geqs 0)$. This establishes the lemma for  modules of the form $M_W (U_n)$, and the general case 
then follows from the decomposition (\ref{eqn: decomp}).
\end{proof}

\subsection{Representation stability and homological stability}

\begin{defn} An  $\FI_W\#$--space is a functor $X\co \FI_W\#\to \Top$.
\end{defn}

Recall from Remark~\ref{FI-rmk} that   an $\FI_W$--space $X$ is determined by the sequence of $W_n$--spaces $X_n  = X(\n)$, together with the $j_n$--equivariant maps $(i_n)_*\co X_n\to X_{n+1}$ induced by the morphisms $i_n$. We will denote the induced maps $X_n/W_n \to X_{n+1}/W_{n+1}$ by $\ol{i}_n$.

The goal of this section is to establish the following result, which will be used to prove the homological stability results to follow.

\begin{thm}\label{quot-stab}
Let $n\goesto X_n$ be a consistent sequence of good $W_n$--spaces
whose degree $k$ rational homology extends to a finitely generated
$\FI_W$--module $H_k (X; \bbQ)$.  Then the sequence of quotient spaces
$$ \cdots \xmaps{\ol{i}_{n-1}} X_n/W_n \xmaps{\ol{i}_n} X_{n+1}/W_{n+1} \xmaps{\ol{i}_{n+1}} X_{n+2}/W_{n+2} \xmaps{\ol{i}_{n+2}}  \cdots$$
is strongly rationally homologically stable in degree $k$. 

If, moreover, the $\FI_W$--module $H_k (X; \Q)$ is generated in stage $d$, then the maps 
$(\ol{i}_n)_*$ are surjective for $n\geqs d$.

In types A, B, and C, if the $\FI_W$--module $H_k (X; \Q)$ is generated in stage $d$, and extends to an $\FI_W \#$--module, then the above sequence is strongly rationally homologically stable for $n\geqs d$. In type $D$, if the $\FI_D$--module $H_k (X; \Q)$ is generated in stage $d$, and is the restriction of an $\FI_{BC} \#$--module, then once again the above sequence is strongly rationally homologically stable for $n\geqs d$.
\end{thm}

In order to prove Theorem~\ref{quot-stab}, we need a general observation regarding equivariant maps between representations.  We will use the averaging operator $\alpha : V \to V^H$, denoted $\alpha(v) = \ol{v}$, from Definition~\ref{alpha}.

\begin{prop}\label{avg} Let $j\co H_1 \to H_2$ be a homomorphism between finite groups, and consider a $j$--equivariant map $\phi\co V_1 \to V_2$, between $\bbQ H_i$--modules. 
If $\phi (V_1)$ generates $V_{2}$ as a $\bbQ H_{2}$--module, then the averaging map
$$V_1^{H_1} \to V_{2}^{H_{2}}$$
defined by $v\goesto \ol{\phi (v)}$
is \e{surjective}.

Consequently, if $V$ is a finitely generated $\FI_W$--module over a field of characteristic zero, and $V$ is generated in stage $d$ and has stable range 
$n\geqs N$, then the averaging map $V_n^{W_n} \to V_{n+1}^{W_{n+1}}$ is surjective for $n\geqs d$ and is an isomorphism for $n\geqs N$.  

Moreover, if $V$ is a finitely generated $\FI_W\#$--module over a field of characteristic zero, and $V$ is generated in stage $d$, then the averaging map $V_n^{W_n} \to V_{n+1}^{W_{n+1}}$ is an isomorphism for $n\geqs d$.
\end{prop}

\begin{proof}
This follows from Wilson~\cite[Proposition 4.16]{wilson2014fiw}, which states that if an $\FI_W$--module $V$ is generated in stage $d$, then its \e{surjectivity degree} (\cite[Definition 4.12]{wilson2014fiw}) is at most $d$. Taking $a=0$ in the definition of surjectivity degree, this says that the natural map
$$I_n\co (V_n)_{W_n} \to  (V_{n+1})_{W_{n+1}}$$ 
between coinvariants (see~\cite[p. 288]{wilson2014fiw})  is surjective for  $n\geqs d$.
The composition $V^{W} \injects V \surjects  (V)_{W}$ is an isomorphism for every $\Q[W]$--module $V$,  and under these isomorphisms the averaging map agrees with the map $I_n$.
\end{proof}

\begin{rmk} In types A, B, and C, Proposition~\ref{avg} admits an elementary proof.
A straightforward verification, using Lemma~\ref{avg-lem}, shows that 
\begin{equation}\label{avg2}\ol{\phi (v)} = \ol{\phi (\ol{v})}.\end{equation}
By hypothesis, for each $v\in V_{n+1}^{H_{n+1}}$ there exist $x_i \in V_n$ and $h_i \in H_{n+1}$ such that 
$v = \sum_i h_i \cdot \phi_n (x_i)$.
Since $v = \ol{v}$, we find (using (\ref{avg2}))
$$v = \ol{\sum_i h_i \cdot \phi_n (x_i)} = \sum_i \ol{h_i \cdot  \phi_n (x_i)} = \sum_i \ol{\phi_n (x_i)}  = \sum_i \ol{\phi_n (\ol{x_i})} =  \ol{\phi_n \left( \sum_i\ol{x_i}\right)}.$$
Since $\sum_i\ol{x_i}\in V_n^{H_n}$, this proves the surjectivity statement. In types A, B, and C, the consequences for finitely generated $\FI_W$-- and $\FI_W\#$--modules  follow from Remark~\ref{rmk: iso-comp} and Proposition~\ref{prop: stability isotypical component}. 
For $\FI_D$--modules 
that are restrictions of $\FI_B$--modules,  
we may instead apply Corollary~\ref{stable-range-D}.
\end{rmk}

\begin{proof}[Proof of Theorem~\ref{quot-stab}] By Proposition~\ref{coh-quot2}, the maps 
$$H_k (X_n/W_n; \bbQ) \xmaps{(\ol{\phi}_n)_*} H_k (X_{n+1}/W_{n+1}; \bbQ)$$ 
are isomorphic to the averaging maps 
$$H_k (X_n; \bbQ)^{W_n} \maps H_k (X_{n+1}; \bbQ)^{W_{n+1}}.$$
If $H_k (X; \bbQ)$ is generated in stage $d$ as an $\FI$--module, then 
Proposition~\ref{avg} shows that this map is surjective for $n\geqs d$.
Moreover, 
Theorem~\ref{thm: fin. gen. implies uniform rep. stability} (together with the fact that $V(\emptyset)_n$ is the trivial representation for each $n$) shows that there exists $N$ such that the domain and range of this map have the same dimension (as $\bbQ$--vector spaces) for $n\geqs N$, and hence are isomorphisms (note that $N\geqs d$).
If $H_k (X; \bbQ)$ extends to an $\FI_W \#$--module, then by  Proposition~\ref{prop: 
stability isotypical component} we can take $N = d$.
\end{proof}

\section{Direct powers and symmetric products}\label{symm-sec}

In this section, we discuss some simple examples of representation stability and  deduce  a homological stability result for symmetric products.  In subsequent sections, 
the representation stability results proven here will be applied to spaces of commuting elements in Lie groups. 

\subsection{Direct powers} We wish to describe an $\FIs$--structure on the sequence of direct powers $X^n$ of a based space $(X, x_0)$. To do this, we need to introduce some notation. Given a partially defined function $S \subset A \srt{g} X$, where $S$ is a (finite) set, let $g_0 \co S\to X$ denote the extension defined by setting $g_0 (s) = x_0$ for all $s\notin A$.

\begin{defn} Given a based space $(X, x_0)$, let $\P (X) = \P(X, x_0)$ denote the $\FIs$--space defined by sending a finite set $S$ to the $|S|$--fold product 
$$X^S = \Map(S, X)$$
 and sending a partially defined injection $S\supset A \srt{\phi} B \subset T$ to the map $X^S \to X^T$ given by sending $f\co S\to X$ to $(f \circ \phi^{-1})_0$.
\end{defn}

It is an exercise to check that this defines a functor out of the category $\FIs$. 
The corresponding $\FI_A \#$--space takes $f\co \m_0 \to \n_0$ to the map 
$f_*\co X^m\to X^n$ defined by
$$f_*(x_1, \ldots, x_m) = (y_1, \ldots, y_n),$$
where $y_i = x_{f^{-1} (i)}$ if $f^{-1} (i)$ exists, and $y_i = x_0$ otherwise.
Applying the functor  $H_* ( - ; \bbQ)$ or $H^* ( - ; \bbQ)$ then yields an $\FIs$--module (in the latter case, this relies on the canonical isomorphism between $\FIs$ and $\FIs^{\rm op}$; see~\cite[Section 4]{church2015fi}). Since an $\FIs$--module restricts to an $\FI$--module, we see that $H_* ( - ; \bbQ)$ and $H^* ( - ; \bbQ)$ are $\FI$--modules.
In both cases, the underlying consistent sequences of $S_n$--representations are determined by the natural actions of $S_n$ on $X^n$ (permuting the factors), along with the maps $(x_1, \ldots, x_n)\goesto (x_1, \ldots, x_n, x_0)$ (for the homological case) or the maps   $(x_1, \ldots, x_{n+1})\goesto (x_1, \ldots, x_n)$ (for the cohomological case).

\begin{defn}
We say that $X$ is of finite (rational) type if $H_k (X; \bbQ)$ is finite-dimensional for each $k \geqslant 0$.
\end{defn}

\begin{prop}\label{symm-prod} If $X$ is path connected and of finite type, then the $\FIs$--module $H_k (\P (X); \bbQ)$ is finite-dimensional and is generated  in stage $k$. Consequently, the underlying sequence of symmetric group representations is uniformly representation stable with stable range $n\geqs 2k$.  
\end{prop}
\begin{proof} 
By the K\"unneth Theorem,  $H_* (X^n; \bbQ)$ is isomorphic to the $n$--fold tensor product of $H_*(X; \bbQ)$ with itself. Under this isomorphism, the $\FIs$--module structure
on $H_k (\P (X); \bbQ)$
 is described by essentially the same formulas as for $X^n$ itself, with sequences of points in $X$ replaced by $n$--fold tensors and the basepoint $x_0$ replaced by the class $[x_0]\in H_0 (X;\bbQ)$.
Now,  $H_k (X^n; \bbQ)$ is generated by simple tensors of the form $a = a_1 \otimes a_2 \otimes \cdots \otimes a_n$, with $a_i \in H_{|a_i|} (X; \bbQ)$, and $\sum_i |a_i| = k$. If $n>k$, then we must have $|a_i| = 0$ for some $i$, meaning that $a_i = q [x_0]$ for some $q\in \bbQ$, and hence  $a$ is in the image of one of the structure maps $H_k (X^{n-1}; \bbQ) \to H_k (X^n; \bbQ)$ defining the $\FI$--module structure on $H_k (\P (X); \bbQ)$. This shows that $H_k (\P (X); \bbQ)$ is generated in stage $k$, 
and the result now follows from   Theorem~\ref{sharp}.
\end{proof}

\begin{lem} If $X$ is  semi-locally contractible, then the same holds for $\Sym^n (X)$ for each $n\geqs 1$ $($in other words, $X^n$ is a good $S_n$--space$)$.
\end{lem}

 The proof is similar to the proof that the symmetric product construction is homotopy invariant. 

Recall that the infinite symmetric product 
$\Sym^\infty (X)$ is simply the colimit of the sequence of maps between the finite symmetric products. Theorem~\ref{quot-stab} now yields the following corollary.

\begin{cor}\label{symm-prod-cor}
Let $X$ be a good, path connected space of finite type. 
Then the maps 
$$H_k (\Sym^n X; \bbQ) \maps H_k (\Sym^{n+1} X; \bbQ) \maps H_k (\Sym^\infty (X); \bbQ)$$
are isomorphisms for $n\geqs k$.
\end{cor}

\begin{rmk}\label{rmk: Steenrod symmetric prod}
In the case where $X$ has the homotopy type of a CW complex, the statement of Corollary \ref{symm-prod-cor} was shown by Steenrod~\cite[Eq. (22.7)]{steenrod} (using other methods) with $\Q$  replaced by the ring of integers. 

Following Steenrod, we use the compactly generated topology on $X^n$ (as in Steenrod~\cite{Steenrod-convenient}) when forming the symmetric product. 
We note that if $X$ is Hausdorff, this does not affect the weak homotopy type of $\Sym^n (X)$. More generally, we claim that if $X$ is Hausdorff and $G$ is a compact group acting on $X$, then the identity map $k(X)/G\to X/G$ is a weak equivalence, where $k(X)$ denotes $X$ with the compactly generated topology. 
Since homotopy groups are defined in terms of maps out of compact spaces,
 it suffices to show that a subset of $k(X)/G$ is compact if and only if it is compact in $X/G$, and that the two subspace topologies on such sets coincide.
Compact sets in $k(X)$ and in $X$ coincide by Steenrod~\cite[Theorem 3.2]{Steenrod-convenient}, and since the quotient maps $k(X)\to k(X)/G$ and $X\to X/G$ are proper (Bredon~\cite[Theorem I.3.1]{Bredon}), the same holds for $k(X)/G$ and $X/G$.
Finally, if $K\subset k(X)/G$ is compact, then the identity map to $K\subset X/G$ is a homeomorphism, since $X/G$ is Hausdorff (ibid.).
\end{rmk}

\subsection{Signed direct powers}\label{sdp-sec}

Consider a space $(X, x_0)$ equipped with an involution $\tau\co X\to X$ 
fixing the basepoint $x_0$. 
We wish to describe an  action of   $B_r = \Z_2 \wr S_r$,  on $X^r$ extending the permutation action of $S_r$. 
For ease of notation, in this section we view $\Z_2$ as the group $\{0,1\}$, with $0$ as identity, and we write the group operation as $+$. (One should think of the sign associated to $\epsilon\in \Z_2$ as $(-1)^\epsilon$.)
In this notation, the subgroup  $D_r\leqs B_r$ of even-signed permutations is given by
$$D_r = \{((t_1, \ldots, t_r), \sigma) \in B_r \,:\, |\{i \,:\, t_i = 1\}| \textrm{ is even.}\}$$
Using the involution $\tau$, we can now endow $X^r$ with the action of $\Z_2 \wr S_r$ given by 
$$(t_1, \ldots, t_r, \sigma) \cdot (x_1, \ldots, x_r) = (\tau^{t_1} (x_{\sigma^{-1}(1)}), \ldots, \tau^{t_r} (x_{\sigma^{-1}(r)})),$$
where $\tau^0 := \textrm{Id}_X$ (and $\tau^1 := \tau$).
  
 Every  morphism $\m_0\to \n_0$ in the category $\FI_{BC}\#$ factors uniquely as $\nu\circ f$, where $f\co \m_0\to \n_0$ is a morphism in $\FI_{BC} \#$ that preserves signs (that is, $f(\{0, \ldots, m\}) \subset \{0, \ldots n\}$), and $\nu\co \n_0\to \n_0$ is a bijection that satisfies $\nu(\{\pm i\}) = \{\pm i\}$ for $1\leqs i \leqs n$ and is the identity on the complement of the image of $f$ (that is, $\nu$ acts as negation of some subset of the image of $f$ and acts as the identity on the remaining elements). 
 
 We will refer to such maps $f$ and $\nu$ as \e{unsigned partial injections} and \e{partial negations}, respectively. If $f'\co \n_0\to \p_0$ is another unsigned partial injection and $\nu'\co \p_0\to \p_0$ is another partial negation that restricts to the identity on the complement of the image of $f'$, then we have
 $$(\nu'\circ f') \circ (\nu \circ f) = \nu'' \circ (f'\circ f),$$ 
 where $\nu''$ is the partial negation
 $$\nu'' (i) = \begin{cases}
     \nu'\circ f'\circ  \nu\circ f (j), 	& \mbox{if $f'\circ f (j) = i$}, \\
     i, & \mbox{if $i$ is not in the image of $f'\circ f$.} \\
\end{cases}$$ 
Note that, by definition, $\nu''$  restricts to the identity on the complement of the image of $f'\circ f$.
 
\begin{lem} \label{ext} Let $(X, x_0)$ be a based space equipped with an involution $\tau\co X\to X$ fixing $x_0$. 
Then $\P(X)$ extends to an $\FI_{BC}\#$--space $($and, by restriction, to an $\FI_{BC}$--space and an $\FI_{D}$--space$)$, whose structure maps 
are determined as follows: For an unsigned partial injection $f\co \m_0\to \n_0$ 
we set
$$f_* (x_1, \ldots, x_m) = (y_1, \ldots, y_n),$$
where 
$$y_i = \begin{cases}
     x_j 	& \mbox{if $f(j) = i$}, \\
     x_0 & \mbox{if $f^{-1} (i) = \emptyset$,} \\
\end{cases}$$
and for a partial negation $\nu\co \n_0 \to \n_0$ we set
$$\nu_* (x_1, \ldots, x_n) = (y_1, \ldots, y_n),$$
where 
$$y_i = \begin{cases}
     x_i 	& \mbox{if $\nu(i) = i$}, \\
     \tau(x_i) & \mbox{if $\nu(i) = -i$.} \\
\end{cases}$$
\end{lem}
\begin{proof} One checks, by a (tedious) computation, that these assignments satisfy
$$\nu'_*\circ f'_* \circ \nu_* \circ f_* = \nu''_* \circ f'_* \circ f_*,$$ 
where $\nu''$ is the partial negation defined above.
\end{proof}

We denote the $\FI_{BC}\#$--space constructed in Lemma~\ref{ext} by $\P_\tau (X) = \P_\tau (X, x_0)$.

\begin{rmk} For a general morphism $\phi \co \m_0\to \n_0$, the  $\FI_{BC}\#$--space $\P_\tau (X)$ satisfies
$$\phi_* (x_1, \ldots, x_m) = (y_1, \ldots, y_n)$$
where
$$y_i = \begin{cases}
     x_j, 	& \mbox{if $\phi(j) = i$}, \\
      \tau (x_j), 	& \mbox{if $\phi(j) = -i$}, \\	 
     x_0 & \mbox{if $\phi^{-1} (i) = \emptyset$.} \\
\end{cases}$$
\end{rmk}

\begin{prop}\label{signed-symm-prod} Let $X$ be a path connected, semi-locally contractible space of finite type. If $\tau\co X\to X$ is an involution fixing $x_0\in X$, 
then the homology groups
 $H_k (X^r; \bbQ)$ 
 form a uniformly representation stable
sequence of $B_r$--representations, 
  with stable range $r\geqs 2k$. 
Consequently, the maps  
$$H_k (X^r/B_r ; \bbQ) \maps H_k (X^{r+1}/B_{r+1} ; \bbQ)$$
$($induced by inserting the basepoint in one factor$)$
are isomorphism for $r\geqs k$. 

The same statement holds with $D_r$ in place of $B_r$.
  \end{prop}
\begin{proof} By Proposition~\ref{symm-prod}, $H_k (\P_\tau (X); \bbQ)$ is finitely generated as an $\FI$--module. It follows immediately that it is also finitely generated as an $\FI_D$--module and as an
$\FI_{BC}$--module, since the structure maps for these enhanced modules include the $\FI$ structure maps. Representation stability in type B/C now follows from Theorem~\ref{thm: fin. gen. implies uniform rep. stability}, and in type D it then follows from Corollary~\ref{stable-range-D}.
The last part follows from Theorem~\ref{quot-stab}.
\end{proof}

\begin{rmk}\label{rmk: signed sym product = sym product} Say $X$ is a space with involution $\tau$.  
Since $(\Z_2)^r$ is normal in $B_r$, we have a homeomorphism $X^r/B_r \isom \Sym^r(X/\bbZ_2)$, where $\bbZ_2$ acts via $\tau$. Hence in type B/C the homological stability statement in Proposition~\ref{signed-symm-prod} is a special case of the one for ordinary symmetric products. 
It follows from Steenrod~\cite[Eq. (22.7)]{steenrod} that in type B/C, the stability result in Proposition \ref{signed-symm-prod} also holds integrally
(see Remark \ref{rmk: Steenrod symmetric prod}).
\end{rmk}

 \section{$\FI_W$--modules arising from Lie groups}\label{Lie}

 Let $\{G_r\}_{r\geqs 1}$ denote one of the following classical sequences of Lie groups:
 \begin{itemize} 
 \item $G_r = \SU(r)$ or $G_r = \U(r)$ (type A);
 \item $G_r = \SO(2r+1)$ (type B);  
 \item $G_r = \Sp(r)$ (type C);
 \item or $G_r = \SO(2r)$ (type D).  
 \end{itemize}
 In this section, we review the structure of the maximal tori and Weyl groups in these sequences and construct two finitely-generated $\FI_W$--modules associated to each sequence.
 We will work exclusively with rational (co)homology in this section and in all subsequent sections (any field of characteristic zero would suffice), so we will drop the coefficients from the notation for simplicity.

We begin by specifying the \e{standard inclusions} $G_r\injects G_{r+1}$. For $G_r = \SU(r)$, $\U(r)$, $\SO(2r)$, and $\SO(2r+1$) these inclusions are given by $A\goesto A\oplus I$, where $I$ denotes an identity matrix of size 1 (in type A) or 2 (in types B and D); so our convention is to put the additional 1's in the lower right corner of the matrix. 

Following Brocker--tom Dieck~\cite{BTD}, we view $\Sp(r)$ as the group of $2r\cross 2r$ block matrices
$$C=C(A,B)=\left[ \begin{array}{rrr} A & -\ol{B}  \\ B & \ol{A} \end{array} \right]$$
such that $C\in \U(2n)$. (Here $A$ and $B$ are arbitrary $n\cross n$ complex matrices, and $\ol{A}$ and $\ol{B}$ are their entry-wise complex conjugates.) The standard inclusion $\Sp(r)\injects \Sp(r+1)$ is the homomorphism
$$\left[ \begin{array}{rrr} A & -\ol{B}  \\ B & \ol{A} \end{array} \right] \goesto \left[ \begin{array}{rrr} A\oplus 1 & -\ol{B}\oplus 0  \\ B\oplus 0 & \ol{A}\oplus 1 \end{array} \right],$$
where $0$ and $1$ are viewed as $1\cross 1$ complex matrices.

Next we make a choice of maximal torus $T_r = T(G_r) \leqs G_r$ in each of our groups, with the property that in each sequence, the standard inclusion maps $T_r$ to $T_{r+1}$, and we describe the associated Weyl groups $NT_r/T_r$ and their actions on the maximal tori. This discussion will mostly follow~\cite{BTD}, and we refer the reader there for further details. Our choices are as follows:

\begin{itemize}
\item $T(\U(r))$ is set of diagonal unitary matrices;
\item $T(\SU(r)):= T(\U(r))\cap \SU(r)$;
\item $T(\SO(2r)) := \SO(2)\oplus \cdots \oplus \SO(2)$;
\item $T(\SO(2r+1)) := 1\oplus \SO(2)\oplus \cdots \oplus \SO(2)$;
\item  $T(\Sp(r)):=\{ C(D,0) \,:\, D\in T(\U(r))\} = \{D\oplus \ol{D} \,:\, D\in T(\U(r))\}$.
\end{itemize}
Note that~\cite{BTD} uses the torus $\SO(2)\oplus \cdots \oplus \SO(2)\oplus 1\leqs \SO(2r+1)$; our choice above ensures that the standard inclusion $\SO(2r+1) \injects \SO(2r+3)$ carries $NT(SO(2r+1))$ into $NT(SO(2r+3))$.  

In each type, there are isomorphisms of Weyl groups $NT_r/T_r \isom W_r$, where $W_r$ is the abstract Weyl group associated to the type (as defined in Section~\ref{FIW-modules}). 
 We briefly specify these isomorphisms.
(In the sequel,  these isomorphisms are treated as identifications, so that $W_r$ refers to either the abstract group  or to $NT_r/T_r$).

{\bf Type A:} In type A, $NT_r$ is generated by $T_r$ together with the signed permutation matrices $P$ of determinant one, and the desired isomorphism is provided by sending $P$ to its underlying (unsigned) permutation of the standard basis for $\C^n$.  The action of $W_r = S_r$ on $T(\U(r)) \isom (S^1)^r$ is simply given by permuting the coordinates of the torus (and as $r$ varies, this in fact gives us the the $\FI_A\#$--space $\P(S^1)$ considered in Section~\ref{symm-sec}). This action restricts to the determinant-one sub-torus $T(\SU(r))\subset T(\U(r))$. 

{\bf Types B and D:} For the special orthogonal groups, we will use the notation for $B_r$ and $D_r$ from Section~\ref{sdp-sec}. For $G_r = \SO(2r+1)$, the isomorphism $W_r = B_r\to NT_r/T_r$ sends $\sigma \in S_r\subset W_r$ to the class represented by the permutation matrix that permutes the ordered pairs of standard basis vectors $p_i = (e_{2i}, e_{2i+1})$ ($i=1, \ldots, r$) according to $\sigma$ (preserving the ordering within the pairs) and fixes $e_1$, while  the element $((0,\ldots,0, 1), e)$  (where $e\in S_r$ is the identity) maps to the class of $\textrm{diag}(-1,1, \ldots, 1, -1)$.  These elements generate $B_r$, so this determines the map. The case of $G_r = \SO(2r)$ is similar: $\sigma\in S_r$ maps to the class of the permutation matrix that permutes the pairs $q_i = (e_{2i-1}, e_{2i})$ ($i=1, \ldots, r$) according to $\sigma$, while, for instance, $((1,1,0,\ldots, 0), e)$ maps to the class of $\textrm{diag}(1,-1, 1, -1, 1, \ldots, 1)$.
Similarly, for $\SO(2r)$ and $\SO(2r+1)$, we have $T_r \isom (\SO(2))^r$. 

In both of the above cases, the Weyl group acts on $T_r \isom (\SO(2))^r$ by permuting the factors and negating the angle of rotation (this is seen by conjugating matrices in $\SO(2)$ by  
$\textrm{diag}(1,-1)$).
If we identify $\SO(2)$ with $S^1$ in the usual way, then signed permutations in $W_r$ act via permutations and complex conjugation.

{\bf Type C:} In our notation $NT(\Sp(r))$ is generated by $T(\Sp(r))$ together with matrices of the form $P\oplus P$, with $P$ a permutation matrix, together with the matrices 
$$C_j=\left[ \begin{array}{rrr} A_j & B_j  \\ B_j & A_j \end{array} \right],$$
for $j=1, \ldots r$, where $B_j$ is the $r\cross r$ matrix with a 1 in position $(j,j)$ and all other entries zero, and $A_j = I - B_j$ . The isomorphism $W_r = B_r \to NT(\Sp(r))/T(\Sp(r))$ sends $\sigma\in S_r$ to the class of $P_\sigma \oplus P_\sigma$ (where $P_\sigma$ is the permutation matrix associated to $\sigma$) and sends $(\epsilon_1,  \ldots, \epsilon_r, e)$ ($\epsilon_i\in \{0,1\}$) to the class of the product $C_1^{\epsilon_1} \cdots C_r^{\epsilon_r}$.
Note  that there are canonical homeomorphisms $T(\Sp(r)) \srm{\isom} T(\U(r)) \isom (S^1)^r$, sending $C(D,0)$ to $D$, and the Weyl group acts by permutations and complex conjugation on these circle factors: Explicitly,
$$(P_\sigma \oplus P_\sigma)\cdot   C(\textrm{diag}(\lambda_1, \ldots, \lambda_n),0)   = C(\textrm{diag}(\lambda_{\sigma^{-1}(1)}, \ldots, \lambda_{\sigma^{-1}(r)}), 0),$$
and
$$C_j\cdot C(\textrm{diag}(\lambda_1, \ldots, \lambda_n),0) = C(\textrm{diag}(\lambda_1, \ldots, \ol{\lambda_j}, \ldots, \lambda_n), 0).$$

The following result is proven by a case-by-case inspection.
 
\begin{lem}\label{NT} The standard inclusions  $G_r \injects G_{r+1}$ map $NT_r$ to $NT_{r+1}$, and the induced maps $NT_r/T_r\to NT_{r+1}/T_{r+1}$ agree  with the standard inclusions $j_r\co W_r\injects W_{r+1}$.
\end{lem}

\begin{prop} \label{prop: T is FI}
For each $n\geqs 0$, the $W_r$--spaces $T^n_r$ form a consistent sequence with respect to the standard inclusions $T_r\injects T_{r+1}$, and  these sequences extend  to  $\FI_W$--spaces. 

For $G_r = \U(r)$, the $\FI_A$--space $r\goesto T^n_r$ is isomorphic to $\P((S^1)^n)$ $($with the identity element in $(S^1)^n$ as basepoint$)$, and hence extends to an $\FI_A \#$--space. 

For $G_r = \SO(2r+1)$ or $\Sp(r)$, the $\FI_{W}$--space $r\goesto T^n_r$ is isomorphic to $\P_\tau ((S^1)^n)$, where $\tau$ is complex conjugation, and hence extends to an 
$\FI_{BC} \#$--space. 

For $G_r = \SO(2r)$, the $\FI_D$--space $r\goesto T^n_r$ is isomorphic to the restriction $\P_\tau ((S^1)^n)|_D$.
\end{prop} 
\begin{proof} Consistency of the sequences follows from Lemma~\ref{NT}.

 Except in the case $G_r = \SU(r)$, we have
$$T_r^n \isom ((S^1)^r)^n \isom ((S^1)^n)^r,$$
where the second homeomorphism is defined by 
$$((z_{11}, \ldots, z_{1r}), \ldots, (z_{n1}, \ldots, z_{nr}))\goesto 
((z_{11}, \ldots, z_{n1}), \ldots, (z_{1r}, \ldots, z_{nr})).$$
From the above descriptions of the Weyl group actions and Lemma~\ref{NT},
we see that this homeomorphism is equivariant with respect to the signed permutation actions of $W_r$, where on the left the action is the diagonal action on $(T_r)^n$ induced by the action of $W_r$ on $T_r\isom (S^1)^r$, and on the right the action is exactly that occurring in the definition of $\P((S^1)^n)$, where $(S^1)^n$ is considered as a space with involution $z\goesto \ol{z}$ (complex conjugation in each coordinate). 
The result now follows from  Corollary~\ref{symm-prod-cor} in the case $G_r = \U(r)$, and  from Proposition~\ref{signed-symm-prod} in the other cases. 

For $G_r = \SU(r)$, Lemma~\ref{lem: extension} shows that $r\goesto T(\SU(r))^n$ is a sub--$\FI$--space of $r\goesto T(\U(r))^n$.
(Note that in this case $T_r^n \isom ((S^1)^{r-1})^n$, and we no longer have an $\FI\#$--space). 
\end{proof}

\begin{cor}\label{cor: T-rep-stable}
For each $n\geqs 0$, the $\FI$--modules $\{H_k (T^n_r)\}_{r\geqs 1}$ 
are generated in stage $k$, and, except possibly in the case $G_r = \SU(r)$, are uniformly representation stable for 
 $r\geqs 2k$.
\end{cor}
\begin{proof} In all cases except for $G_r = \SU(r)$, this follows from Proposition~\ref{prop: T is FI} together with Propositions~\ref{symm-prod} and~\ref{signed-symm-prod}.

For $G_r = \SU(r)$, we need to show that
the $\FI$--module $\{H_k (T^n_r)\}_{r\geqs 1}$ is generated in stage $k$.
Note that the action of $W_{r-1} \leqs W_r$ on $T_r$ agrees with the usual permutation action on 
$(S^1)^{r-1}$. The   homeomorphism
$$T_r^n \isom ((S^1)^{r-1})^n \isom ((S^1)^n)^{r-1},$$
now endows $((S^1)^n)^{r-1}$ with a $W_r$--action, which when restricted to $W_{r-1}$ gives the defining action for $\P((S^1)^n)$. In fact, if we restrict the $\FI$--space $\r \goesto T_r^n$ along the functor $\FI\to \FI$ defined by sending $\r$ to ${\bf r+1}$ and sending $\phi\co \r\to \s$ to the unique extension $\wt{\phi} \co {\bf r+1} \to {\bf s+1}$ satisfying $\wt{\phi}(r+1) = s+1$, we obtain the $\FI$--space $\P((S^1)^n)$. For each $k\geqs 0$, the $k^{\textrm th}$ homology of this restricted $\FI$--space is generated in stage $k$, and so the same is automatically true for the original $\FI$--space. 
\end{proof}

Next we consider the flag manifolds $G/T$.
For each of the above sequences of Lie groups, the spaces $G_r/T_r$ admit left actions of $W_r$, defined by $[n]\cdot gT_r = gn^{-1} T_r$. The standard inclusions induce maps $G_r/T_r \to G_{r+1}/T_{r+1}$ making these sequences consistent, but these sequences \e{do not} satisfy the conditions in Lemma~\ref{lem: extension} and hence do not extend to $\FI_W$--spaces.  Nevertheless, we will show that the associated consistent sequences in \e{homology} do in fact extend to $\FI_W$--modules. Moreover, we will see that these  $\FI_W$--modules are in fact dual to certain algebraically-defined co--$\FI_W$--modules, which were shown by
Church--Ellenberg--Farb~\cite[Theorem 5.1.5]{church2015fi} and Wilson~\cite[Corollary 6.5]{wilson2014fiw} to be finitely generated.

Let $G$ be a compact Lie group with maximal torus $T\leqs G$, and let $ET\to BT$ and $EG\to BG$ denote the (functorial) simplicial models for the universal principal bundles (as defined, for instance, in~\cite{segal1968csss}).
The  principal action  of $NT\leqs G$ on $EG$  descends to an action of $W = NT/T$ on $EG/T$.
We will call this the translation action of $W$ on $EG/T$.

\begin{lem}\label{lem: translation=conjugation}
Consider the homotopy equivalence $i \co ET/T \to EG/T$ induced by the inclusion $T\injects G$. 
Then the conjugation action of $W$ on $H^*(ET/T)$ and the translation action of $W$ on $H^*(EG/T)$ coincide under the isomorphism $i^*$,
and similarly for homology.
\end{lem}
\begin{proof}  
Recall that for a Lie group $H$, the simplicial model for $EH$ is the geometric realization of the simplicial space  whose space of $k$--simplices is $H^{k+1}$. The face maps are given by deletion of elements from a string, while the degeneracies are given by repetition. The principal action of $H$ on $EH$ is given by right-multiplication. 

Note that the inclusions $ET/T \to ENT/T\to EG/T$ are both homotopy equivalences. The space $EN$  also admits a conjugation action of
$N$, which descends to a conjugation action of $W$ on $ENT/T$, and the map $ET/T \to ENT/T$ is $W$--equivariant (when $ET/T$ has the conjugation action). Moreover, the principal action of $N$ on $EN$ also descends to a $W$--action on $ENT/T$, making the map  $ENT/T\to EG/T$ equivariant. Hence to prove the lemma, it will suffice to show that for every $[n]\in W$, the conjugation and translation actions of $[n]$ on $ENT/T$ are homotopic to each other, and we prove this by exhibiting a simplicial homotopy between these maps. Recall (see~\cite{may1992simplicial} for instance) that a simplicial homotopy between maps $f,g\co ENT\to ENT$  consists of a collection of maps $h_{ki} \co (NT)^{k+1} \to NT^{k+2}$ for $k=0,1,2,\ldots$ and $0\leqs i\leqs k$, satisfying a collection of identities (these identities simply amount to saying that the $h_{ki}$ combine into a simplicial map $ENT \cross \Delta^1\to ENT$, where $\Delta^1$ is the standard simplicial 1--simplex).
We define 
$$h_i (n_0,...,n_k) = (n_0 n, ..., n_i n, n^{-1} n_i n, ..., n^{-1} n_k n);$$
one readily checks that all of the identities hold.
One checks that this map is well-defined on equivalence classes modulo the principal action of $T$ (given by right-multiplication in all coordinates), and hence descends to a simplicial homotopy between the translation   and conjugation actions of $[n]$ on $ENT/T$, as desired.
\end{proof}

Borel~\cite[Proposition 29.2(a)]{Borel57} showed that for a compact, connected Lie group $G$ with maximal torus $T \leqs G$, the ring $H^*(G/T)$ is isomorphic to the cokernel of the map $(i^*)^+\co H^*(BG)^+ \to H^*(BT)$ induced by the inclusion $i\co T\injects G$ (here $H^*(BG)^+ $ is the ideal consisting of elements in non-zero degrees).  The conjugation action of $NT$ on $BG$ is homologically trivial because for every element $g\in G$, the conjugation map $c_g\co BG\to BG$ induced by $g$ is nullhomotopic (thought of as an automorphism of the groupoid whose nerve is $BG$, the functor $c_g$ is isomorphic to the identity via a continuous natural transformation).
This implies that the image of $(i^*)^+$ is contained in the $W$--invariants $H^*(BT)^W$, and by Baird~\cite[Theorem B.2]{baird2007cohomology}, we also have an isomorphism
$H^*(G/T) \isom H^*(BT)/(H^*(BT)^W)^+$, where again $+$ denotes the ideal of non-zero degree elements. By comparing dimensions, it follows that the image of $(i^*)^+$ is exactly $(H^*(BT)^W)^+$.

We now give a functorial version of these results.

\begin{prop}\label{prop: G/T}
For each compact, connected Lie group $G$ and maximal torus $T\leqs G$, the map $f \co G/T \to BT$ classifying the principal $T$--bundle $G\to G/T$ is surjective in rational cohomology $($and, dually, injective in rational homology$)$, and the kernel of $f^*$ is precisely $(H^*(BT)^W)^+$.
\end{prop}
\begin{proof}
First we prove that $f^* \co H^*(BT)\to H^*(G/T)$ is surjective.
It suffices to prove surjectivity in degree 2, because $H^*(G/T)$ is generated in degree 2 (since $H^*(BT)$ is generated in degree 2, this follows from either Borel or Baird's result). 
We have a map of fibrations
\begin{equation}\label{eqn: G/T fib}
\begin{tikzcd}
T  \arrow{r}{=} \arrow{d}{i} &  T\ar{d}  \\
G \arrow{r} \arrow{d} 	& EG \ar{d}   \\
G/T \ar{r}{f} & EG/T \heq BT,
\end{tikzcd}
\end{equation}
which gives rise to a map between the Serre spectral sequences for these fibrations.
Consider the differentials $d_2\co E_2^{0,1}\to E_2^{2,0}$ in the two spectral sequences: on the right, this differential is an isomorphism (since $EG$ is contractible), and on the left, it is surjective since $H^2(G)=0$ (in general, $H^*(G)$ is an exterior algebra concentrated in odd degrees -- see Reeder~\cite{reeder1995cohomology}, for instance). We thus have a commutative diagram
\begin{center}
\begin{tikzcd}
H^1(T)   \arrow[d, tail, twoheadrightarrow, "d_2"] \arrow[r, leftarrow, "="]&  H^1(T)
\arrow{d}[swap]{\isom}{d_2}  \\
H^2(G/T)  \arrow[r, leftarrow, "f^*"]	& H^2(BT),  \\
\end{tikzcd}
\end{center}
and it follows that $f^*$ is surjective in degree 2 (and hence in all degrees).  

Since we know that there is an isomorphism $H^*(G/T)\isom H^*(BT)/(H^*(BT)^W)^+$, and these graded vector spaces are finite-dimensional in each degree, to show that $\ker (f^*) = (H^*(BT)^W)^+$ it suffices to check that $(H^*(BT)^W)^+\leqs \ker(f^*)$. As an ungraded $W$--representation, $H^*(G/T)$ is   the regular representation (see, for instance, 
Baird~\cite[Theorem B.1]{baird2007cohomology}), and hence $W$ acts freely on $H^*(G/T)^+$, so $f^*((H^*(BT)^W)^+) = 0$.
\end{proof}

\begin{prop} \label{prop: H(G/T) is FI}
The consistent sequences of $W_r$--modules $r\goesto  H_k (G_r/T_r)$ extend to finitely generated $\FI_W$--modules.
\end{prop}
\begin{proof} Let $1\cross W_s\subset W_{r+s}$ denote the subgroup of elements fixing $1, \ldots, r\in \s_0$. By Lemma~\ref{lem: extension}, we need to prove that the action of $1\cross W_{s}$ on the image of the map
$$H_k (G_r/T_r) \maps H_k (G_{r+s}/T_{r+s})$$
induced by the standard inclusion is trivial. 

Note that in Diagram (\ref{eqn: G/T fib}) above, we may take the map $G\to EG$ to be the inclusion of one fiber of the fibration $EG\to EG/G$ (giving a specific choice for the classifying map $f$).
This choice gives us a commutative diagram
\begin{equation}\label{eqn: f natural}
\begin{tikzcd}
 G_r/T_r   \arrow{r}{f_r} \arrow{d} &  EG_r/T_r\ar{d}  \\
 G_{r+s}/T_{r+s}   \arrow{r}{f_{r+s}} 	&  EG_{r+s}/T_{r+s}   
\end{tikzcd}
\end{equation}
relating the classifying maps from Proposition~\ref{prop: G/T}, and moreover this choice ensures that the maps $f_r$ and $f_{r+s}$ are $W_r$-- and $W_{r+s}$--equivariant (respectively), where on the right the Weyl group actions are induced by the principal actions of $NT_r\leqs G_r$ and $NT_{r+s} \leqs G_{r+s}$ on the universal bundles.

Since $(f_{r+s})_*$ is injective, to complete the proof it suffices to show that $1\cross W_{s}$ acts trivially on the image of $H_k (EG_r/T_r)$ in $H_k (EG_{r+s}/T_{r+s})$. By Lemma~\ref{lem: translation=conjugation}, this is equivalent to showing that $1\cross W_{s}$ acts trivially on the image of $H_k (BT_r)$ in $H_k (BT_{r+s})$, where now $1\cross W_{s} \leqs W_{r+s}$ is acting by conjugation. But the image of $T_r$ in $T_{r+s}$ is fixed point-wise under conjugation by $1\cross W_{s}$, and the same is true after applying the bar construction.

Finally, we show that $H_k (G_r/T_r)$ is finitely generated as an $\FI_W$--module. For $k=0$, connectedness of $G_r$ implies that the $\FI_W$--module $H_k (G_r/T_r)$ is constant, with value the trivial representation.
For $k>0$, Proposition~\ref{prop: G/T} and commutativity of Diagram (\ref{eqn: f natural}) yield an isomorphism of $\FI_W$--modules 
$$H_k (G_r/T_r) \srm{\isom} (H^k(BT_r)/(H^k(BT_r)^{W_r}))^*,$$
where on the right, $W_r$ acts by conjugation.
It follows from Proposition~\ref{prop: T is FI} that $(H^k(BT_r)/(H^k(BT_r)^W))^*$ is isomorphic, as a co--$\FI_W$--module, to the degree--$k$ part of the diagonal coinvariant algebra $\bbQ[x_1, \ldots, x_r]/\mathcal{I}_r$, with the $\FI_W$--module structure coming from the signed permutation action of $W_r$ on the variables and the natural projections $x_1 \goesto x_1, \cdots, x_r\goesto x_r, x_{r+1} \goesto 0$. Wilson~\cite[Theorem 6.1]{wilson2014fiw} shows $(\bbQ[x_1, \ldots, x_r]/\mathcal{I}_r)^*$ is finitely generated as an $\FI_W$--module (in each degree).
\end{proof}

\section{Stability for commuting elements in compact Lie groups}\label{stab}

In this section we prove several of our main results regarding the
 homology of spaces of commuting elements in compact, connected Lie groups. 
All coefficients in this section are, implicitly,  rational (any field of characteristic zero would suffice).


\subsection{Stability for classical sequences of Lie groups}

Throughout this section, $G_r$ will again denote one of the classical Lie groups: $\SU(r)$ or $\U(r)$ (type A); $\SO(2r+1)$ (type B); $\Sp(r)$ (type C); or $\SO(2r)$ (type D);
and $T_r = T(G_r)\leqs G_r$ will be the maximal torus defined in Section~\ref{Lie}, with Weyl group $W_r$.

Fix  positive integers $n$ 
and   $k$. In Section~\ref{sec: stable range for Hom}, we will establish homological stability for the sequences $r\goesto \Hom(\Z^n, G_r)_1$. Here we consider the conjugation quotients $Rep(\Z^n, G_r)_1$, which turn out to be simpler to analyze.


\begin{thm}\label{stability-wrt-r-Rep}
The sequences 
$r\goesto \Rep(\Z^n,G_r)_1$
satisfy strong rational homological stability.
In homological degree $k$, stability holds for $r \geqs k$. 
\end{thm}
\begin{proof} By Theorem~\ref{Stafa}, these spaces are homeomorphic to $T_r^n/W_r$, and the homeomorphisms commute with the stabilization maps induced by the standard inclusions (Remark~\ref{rmk: phi natural}). 
By Corollary~\ref{cor: T-rep-stable}, the sequence $r\goesto H_k (T_r^n)$ is generated in stage $k$ as an $\FI_W$--module, and in all cases except $G_r = \SU(r)$, this module extends to $\FI_W \#$.  For $G_r \neq \SU(r)$, the result now follows from Theorem~\ref{quot-stab}. 

When $G_r = \SU(r)$,
Theorem~\ref{quot-stab} still tells us that the the maps
$$H_k (T_r^n)^{W_r} \isom H_k (T_r^n/W_r) \maps H_k (T_{r+1}^n/W_{r+1}) \isom H_k (T_r^n)^{W_r}$$ 
are surjective for $r\geqs k$. It remains to prove injectivity for $r\geqs k$.
The inclusions $\SU(r)\injects \U(r)$ restrict to inclusions $i = i_r\co T_r \injects T_r':= T(\U(r))$ between the diagonal maximal tori, and the action of $W_r$ on $T_r$ is just the restriction of its action on $T_r'$. For $r\geqs k$, we thus have a commutative diagram
\begin{center}
\begin{tikzcd}
H_k (T^n_r)^{W_r}  \arrow{d}{i_*} \arrow[r, tail, twoheadrightarrow]  & H_k (T^n_{r+1})^{W_{r+1}}   \arrow{d}{i_*}    \\
H_k ((T'_r)^n)^{W_r}   \arrow{r}{\isom} & H_k ((T'_{r+1})^n)^{W_{r+1}}, 
\end{tikzcd}
\end{center}
and it will suffice to show that $i_* \co H_k (T^n_r) \to H_k ((T'_r)^n)$ is injective. This can be seen by direct computation using the K\"unneth Theorem, or from the fact that the Serre spectral sequence for the fibration sequence $T_r \to T_r'\xmaps{\det} S^1$ collapses (for dimension reasons) at the $E^2$ page.
\end{proof}

\begin{rmk}\label{rmk: rep space integral stability}
In types A, B, and C, the homological stability result in Theorem \ref{stability-wrt-r-Rep} holds integrally
by Steenrod~\cite[Eq. (22.7)]{steenrod}.
\end{rmk}

\subsection{A note on $\pi_2(\Rep(\Z^n,\U(r)))$ and $\pi_2(\Rep(\Z^n,\SU(r)))$} 
Here we make a short note on the second homotopy group of these representation spaces by showing
that they are non-trivial, which is in contrast with a result of Florentino, Lawton and Ramras in the case of free group character varieties~\cite[Theorem 5.12]{FLR}.

We will make use of a result of Lawton and Ramras~\cite{Lawton-Ramras}, who show that the
universal cover of $\Rep(\Z^n,\U(r))$ is $\Rep(\Z^n,\R \times \SU(r)) \cong \Rep(\Z^n,\R) \times \Rep(\Z^n,\SU(r)).$ Therefore, there is an 
isomorphism $$\pi_2(\Rep(\Z^n,\U(r))) \cong \pi_2(\Rep(\Z^n,\SU(r))).$$
Moreover, if $\pi_1(G)$ is finite, it was shown by Biswas, Lawton and Ramras \cite{BLR} that $\Rep(\Z^n,G)_1$ is simply connected, and by the Hurewicz theorem we have 
$$\pi_2(\Rep(\Z^n,G)_1)\cong H_2(\Rep(\Z^n,G)_1; \Z).$$
One such Lie group is $\SU(r)$. 
By Theorem \ref{stability-wrt-r-Rep} we know that the second homology of $\Rep(\Z^n,\SU(r))$
stabilizes for $r\geqslant 2.$ That is, for all $r\geqslant 2$ we have
$$H_2(\Rep(\Z^n,\SU(2));\Q)\cong H_2(\Rep(\Z^n,\SU(r));\Q)$$
(and in fact, this isomorphism hold integrally by Remark~\ref{rmk: rep space integral stability}).
From \cite[Ex. 7.1]{stafa2017poincare} the Poincar\'e series
of $\Rep(\Z^n,\SU(2))$ is $((1+s)^n+(1-s)^n)/2,$ and we can see that
the coefficient of $s^2$ is $n\choose 2$. Therefore we have the following

\begin{prop}
The homotopy groups $\pi_2(\Rep(\Z^n,\U(r)))\isom\pi_2(\Rep(\Z^n,\SU(r)))$ have rank ${n\choose 2}.$ 
\end{prop}

We note that Lawton--Ramras~\cite[Theorem 5.3]{Lawton-Ramras} actually shows that 
$$\pi_2 (\Rep(\Z^2,\SU(2))) \isom \pi_2 (\Rep(\Z^2,\U(2))) \isom \bbZ.$$


\subsection{Stability with respect to number of commuting elements}\label{sec: n-stability}

Fix a compact, connected Lie group $G$ (not necessarily of classical type) with maximal torus $T$ and Weyl group $W$.
We now study how the homology of spaces of commuting $n$--tuples in $G$ varies as we increase  $n$.

Observe that the maps 
$$ T^n = \Rep(\bbZ^n, T) \srm{i} \Rep(\Z^n,G)_1$$
and
$$\phi \co G/T\times_W T^n \to \Hom(\Z^n,G)_1$$
 from Section~\ref{Hom} are $S_n$--equivariant, where on the left, $S_n$ acts by permuting the coordinates of $T^n$ and acts trivially on $G/T$, and on the right, $S_n$ acts by permuting the coordinates of $\Z^n$. Hence the induced maps
$$i_*\co   H_*(T^n/W)  \srm{\isom} H_*(\Rep(\Z^n,G)_1)$$
and
$$\phi_* : H_*(G/T \times_W T^n)   \srm{\isom} H_* (\Hom(\Z^n,G)_1),$$
which are isomorphisms by Theorems~\ref{Baird} and~\ref{Stafa}, are  isomorphisms of $\bbQ [S_n]$--modules.
We now extend these maps to isomorphisms of $\FIs$--modules.

First we explain the $\FIs$--module structure on the domains of $i_*$ and $\phi_*$.
The diagonal action of $W$ on $T^n$  commutes with the permutation action of $S_n$. Moreover, the inclusions $i_n\co T^n\injects T^{n+1}$, defined by $(t_1, \ldots, t_n)\goesto (t_1, \ldots,t_n, 1)$, are $W$--equivariant. It follows that $W$ in fact acts on the $\FIs$--space $\P(T)$  through maps of $\FIs$--spaces. Hence $\P(T)$ has the structure of a $W$--object in $\FIs$--spaces, or, equivalently, an $\FIs$--object in $W$--spaces. Applying the quotient space functor then yields an $\FIs$--space $n\goesto T^n/W$, whose underlying consistent sequence is given by the permutation actions and the maps induced by the inclusions $i_n$.

To understand the $\FIs$--module structure on the domain of $\phi_*$,
first note that given two $\FIs$--spaces $X$ and $Y$, there is a direct product $\FIs$--space $X\cross Y$ defined on objects by $\n_0 \goesto X(\n_0) \cross Y(\n_0)$ and on morphisms by sending 
$\phi\co \n_0 \to \m_0$ to $X(\phi)\cross Y(\phi)$. In particular, considering $G/T$ as a constant $\FIs$--space, we obtain an $\FIs$--space $G/T\cross \P(T)$, with underlying consistent sequence $\n_0\goesto G/T\cross T^n$. Again, $W$ acts through maps of $\FIs$--spaces, so the consistent sequence $\n_0\goesto G/T\cross_W T^n$ extends to an $\FIs$--space by the same reasoning as above.

 In order to endow the ranges of $i_*$ and $\phi_*$  with the structures of $\FIs$--modules, note that there is a canonical isomorphism of categories $\FIs \xmaps{\isom} \FIs^{\textrm{op}}$ which is the identity on objects and takes a partial injection 
$$S \supset A \stackrel[\isom]{}{\xmaps{\psi}} B \injects T$$
to the morphism $S\to T$ in the opposite category corresponding to
$$T \supset B \stackrel[\isom]{}{\xmaps{\psi^{-1}}} A \injects S.$$
We may thus consider $\P(\bbZ)$ as an $\FIs^\textrm{op}$--object in groups, and applying the contravariant functors $\Rep$ and $\Hom$ (followed by homology) makes the ranges of $i_*$ and $\phi_*$ into $\FIs$--modules.

Tracing the definitions yields the following result.

\begin{lem}\label{lem: i* phi*} The maps  $i_*$ and $\phi_*$ induce isomorphisms of $\FIs$--modules.
\end{lem}

Next, if $X$ is an $\FIs$--object in $W$--spaces, where $W$ is a finite group, then
since the isomorphisms 
$$H_k (X_n)^W \srm{\isom} H_k (X_n/W)$$
in Proposition~\ref{coh-quot} are natural with respect to equivariant maps, they in fact provide an isomorphism of $\FIs$--modules. 

\begin{lem}\label{lem: FIs-iso} The $\FIs$--modules
$$\n_0\goesto H_k (\Rep(\Z^n, G)_1) \,\,\,\textrm{ and }\,\,\, \n_0\goesto H_k (\Hom(\Z^n, G)_1)$$
are isomorphic $($respectively$)$ to the $\FIs$--modules
$$H_k(\P(T))^W \,\,\,\textrm{ and }\,\,\, H_k(G/T\cross \P(T))^W$$
\end{lem}

Since the structure maps for the $\FIs$--modules 
$$H_k(\P(T)) \,\,\,\,\,\textrm{ and }\,\,\,\,\,   H_k(G/T\cross \P(T))$$
are $W$--equivariant, these are in fact $\FIs$--objects in the category of $\Q[W]$--modules. 
In general, passage to $W$--invariants defines a functor from $\Q [W]$--modules to $\Q$--modules, so the $W$--invariants of an $\FIs$--module over $\Q [W]$ form an $\FIs$--module over $\Q$, and similarly for $\FI$--modules. We have the following general fact regarding generators for these fixed-point submodules.

\begin{lem}\label{W-invt-fin-gen}
Let $H$ be a finite group, and let $V\co \FI \to\k [H]\textrm{--}{\bf Mod}$ be an $\FI$--module over the group ring $\k [H]$, where $\k$ is a field of characteristic not dividing $|H|$. Let $R\co\k [H]\textrm{--}{\bf Mod}\to \k\textrm{--}{\bf Mod}$ denote the 
forgetful functor, and assume that $R\circ V$ is generated in stage $k$. Then the $\FI$--module $V^H$ is generated in stage $k$ as well.
\end{lem}
\begin{proof}
Say $x\in V_n^H$, with $n>k$. We must show that $x$ can be written as a sum of images of elements in $V_k^H$ under the structure maps for $V^H$.
By hypothesis, there exist $x_i\in V_n$, $y_i\in V_k$, and $\phi_i\co \k \to \n$ such that 
$$x = \sum_i x_i\,\,\,\,\, \textrm{ and } \,\,\,\,\,x_i = {\phi_i}_* (y_i).$$
Consider the $H$--invariant elements $\ol{x_i}$ and $\ol{y_i}$ (as in Definition~\ref{alpha}).
Since $x = \ol{x}$, linearity of the averaging map implies that $x =  \sum_i \ol{x_i}$.
Since each ${\phi_i}_*$ is a map of $\k [H]$--modules, we have
$${\phi_i}_* (\ol{y_i}) =\ol{{\phi_i}_* (y_i)} = \ol{x_i},$$
so
$$x =  \sum_i \ol{x_i} =   \sum_i {\phi_i}_* (\ol{y_i}).$$
But $\ol{y_i}\in V_k^H$, so this shows that $x$ is in the $\k$--linear subspace of $V_n^H$ spanned by the images of the maps 
$${\phi_i}_*\co V_k^H \maps V_n^H,$$
as desired.
\end{proof}

\begin{thm}\label{thm: stability for fixed G}
Let $G$ be a compact, connected Lie group, and fix $k\in\bbN$. The sequences of $S_n$--representations 
$$n\goesto H_k(\Rep(\Z^n, G)_1) 
\,\,\,\text{ and }\,\,\,
n\goesto H_k(\Hom(\Z^n, G)_1)
$$
are generated in stage $k$, and are uniformly representation stable for 
 $n\geqs 2k$. Moreover,
the sequences $\{\Rep(\Z^n, G)_1/S_n\}_{n\geqs 1}$ and $\{\Hom(\Z^n, G)_1/S_n\}_{n\geqs 1}$ satisfy strong rational homological stability, and in homological degree $k$, stability in fact holds for $n\geqs k$.
\end{thm}
\begin{proof} By Lemmas~\ref{lem: i* phi*} and~\ref{lem: FIs-iso} together with Theorems~\ref{sharp} and~\ref{quot-stab},
it suffices to show that the $\FIs$--modules $H_k (\P(T))^W$ and $H_k (G/T \cross \P(T))^W$ are generated in stage $k$. Proposition~\ref{symm-prod} tells us that $H_k (\P(T))$ is finitely generated in stage $k$, and the same holds for $H_k (\P(T))^W$ by Lemma~\ref{W-invt-fin-gen}. 
Next, we have a decomposition of  $\FIs$--modules
\begin{equation}\label{decomp} H_k(G/T \times \P(T)) \isom \bigoplus_{i=0}^k H_i (G/T) \otimes H_{k-i} (\P(T)).
\end{equation} 
Since the $\FI$--module $H_{k-i} (\P(T))$ is generated in stage $k-i$, and $H_i (G/T)$ is constant, we see that each term in this decomposition is generated in stage $k$ or earlier. Hence $H_k(G/T \times \P(T))$ is generated in stage $k$ as well, and  Lemma~\ref{W-invt-fin-gen} completes the proof.
\end{proof}

\section{Infinite-dimensional constructions}\label{sec: B(2,G) stability}

Fix a compact and connected Lie group $G$. As   observed in Section~\ref{sec: n-stability}, the sequence  
$\{\Hom(\Z^n,G)\}_{n\geqslant 0}$ extends to an $\FI$--space. 
In this section we will consider several infinite-dimensional constructions associated to this $\FI$--space. Once again, all coefficients in this section are rational.

First, note that $\{\Hom(\Z^n,G)\}_{n\geqslant 0}$ also has the structure of a simplicial space, in which the structure maps are the restrictions of those on the bar construction of $G$. 
Following~\cite{adem2012commuting}, we denote the geometric realization of this simplicial space  by
$$
\Bcom G := |\Hom(\Z^\bullet,G)|.
$$
The inclusion $\Hom(\Z^n,G) \subseteq G^n$ implies that the space $\Bcom G$
is a subspace of the classifying space $BG$ of $G$.

We define $\Bcom(G)_1\subset \Bcom G$ to be the subspace
$$
\Bcom (G)_1 := |\Hom(\Z^\bullet,G)_1|.
$$
It is important to note that for $G=\SU(n)$, $\U(n)$, or $\Sp(n)$,
we have $$\Bcom G=\Bcom G_1$$ (since all of the representation spaces are connected in these cases). 

The   spaces $\Bcom G$ and $\Bcom(G)_1$ were introduced by 
Adem, Cohen and Torres-Giese~\cite{adem2012commuting}
and used by Adem, G\'omez \cite{adem2015gomez} and Adem, G\'omez, Lind, Tillman~\cite{adem2017gomezlindtillman}
to define \emph{commutative $K$--theory}.
In particular, $\Bcom \U := \colim_n \Bcom \U(n)$ represents (reduced) commutative complex $K$--theory.

\

The other construction we will study is an analogue of the James reduced product,
and was introduced by Cohen and Stafa \cite{cohen2016spaces}.
Recall that for a CW-complex $X$ with a non-degenerate basepoint $*$, the James reduced product $J(X)$
of $X$ is defined as the quotient space
$$
J(X) : = \bigg(\coprod_{n \geq 0} X^n \bigg)/\sim,
$$
where $\sim$ is the equivalence relation generated by $(\dots,*,\dots) \sim (\dots,\hat*,\dots)$,
which omits the basepoint. Equivalently, this is the free topological monoid generated by $X$ with the basepoint $*$
acting as the identity element, endowed with the weak topology above. 

Define the space $\Comm(G)$ as the quotient space
$$
\Comm(G) : = \bigg(\coprod_{n \geq 0} \Hom(\Z^n,G)\bigg)/\sim,
$$
where $\sim$ is the same relation as above. Note that $\Comm(G) \subseteq J(G).$
We will focus here on the subspace 
$$\Comm(G)_1 : = \bigg(\coprod_{n \geq 0} \Hom(\Z^n,G)_1\bigg)/\sim,$$
and as above, for $G=\SU(n)$, $\U(n)$, or $\Sp(n)$, we have 
$$\Comm (G)_1=\Comm (G).$$

Now let $T \leqs G$ be a maximal torus with Weyl group $W$. 
The conjugation maps 
$$\phi_n \colon G/T\times_{W} T^n \to \Hom(\Z^n,G)_1 $$ 
can be assembled to give maps
$$
\phi' \colon G/T\times_{W} BT  \to \Bcom (G)_1
$$
and
$$
\phi'' \colon G/T\times_{W_r} J(T) \to \Comm(G)_1.
$$

The next result was proven for $\Bcom  (G)$  in  \cite[Theorem 6.1]{adem2012commuting} and for $\Comm(G)$ in
\cite[Theorem 7.1]{cohen2016spaces}.

\begin{thm}\label{thm: phi}
The maps 
$$
\phi' \colon G/T\times_{W} BT  \to \Bcom ({G})_1
$$
and
$$
\phi'' \colon G/T\times_{W} J(T) \to \Comm(G)_1
$$
induce isomorphisms in rational $($co$)$homology.
\end{thm}

Now let $\{G_r\}_{r\geqs 1}$ be one of the classical sequences of Lie groups, with maximal tori $T_r \leqs G_r$ (as in Section~\ref{Lie}) and Weyl groups $W_r$. 
For a fixed positive integer $k$, consider the consistent sequences of $W_r$--representations
$$
  \{H_k(G_r/T_r\times BT_r)\}_{r\geqslant 1}, \text{ and }  
  \{H_k(G_r/T_r\times J(T_r))\}_{r\geqslant 1},
$$
with structure maps induced by those for the consistent sequences
$\{G_r/T_r\}_{r\geqs 1}$ and $\{T_r\}_{r\geqs 1}$ (note here that both constructions $B$ and $J$ are functorial).
The argument in the proof of Proposition~\ref{prop: H(G/T) is FI} shows that these sequences extend to $\FI_W$--modules.

First we show that these $\FI_W$--modules are
representation stable, which implies strong homological stability for $\Bcom ({G_r})_1$ and $\Comm(G_r)_1$.  Note that except in the case $G_r = \SU(r)$, 
the $\FI_W$--modules $\r\goesto B(T_r)$ and
$\r\goesto J(T_r)$ extend to $\FI_W\#$--modules (since this holds before applying the functors $B$ and $J$ by Proposition~\ref{prop: T is FI}).

\begin{prop}\label{prop: stability BT}
The $\FI_W$--modules $\r \goesto H_k (BT_r)$ are generated in 
stage $k$.  Hence these modules are uniformly representation stable for $r\geqs 2k$ $($except possibly in the case $G_r = \SU(r)$$)$.
\end{prop}
\begin{proof}
 For $G_r = \U(r)$, $\Sp(r)$, $\SO(2r)$, or $\SO(2r+1)$, the Weyl group $W_r$ acts
on $T_r\isom (S^1)^r$ via permutations and complex conjugation, so $S_r < W_r$ acts on $BT_r\isom (BS^1)^r$ by permuting the factors.
We have an isomorphism
$$H^*(BT_r) \isom \Q[x_1,\dots,x_r],$$
where $|x_i|=2$,
 and in cohomology this action just permutes the set of generators $x_1,\dots,x_r$; the action in homology is simply dual to this action. 
 It can easily be  seen
(e.g. \cite[Example 1.4]{wilson2014fiw}) that the degree--$k$ submodule of the $\FI$--module $\r\goesto \Q[x_1,\dots,x_r]$, with the prescribed
action of the symmetric group, is finitely generated in stage $k$. Alternatively, we can view $BT_r$ as $(BS^1)^r\heq (\C P^\infty)^r$,
and for increasing $r$ its degree $k$ rational cohomology is generated in stage 
$k$ by Proposition~\ref{signed-symm-prod}. 

 For $G_r = \SU(r)$, the subgroup $S_{r-1}\leqs W_r \isom S_{r}$  consisting of permutations fixing $r$  acts on $T(\SU(r)) \isom (S^1)^{r-1}$ by permuting the factors, and hence acts on $$H^*(BT(\SU(r))) \isom \Q[x_1,\dots,x_{r-1}]$$
by permuting the generators. Finite generation in this case now follows by the same argument as in Corollary~\ref{cor: T-rep-stable}.
\end{proof}

\begin{prop}\label{prop: stability J}
The $\FI_W$--modules $\r\goesto H_k(J(T_r))$ are generated in stage $k$.
Hence these modules are uniformly representation stable for $r\geqs 2k$ $($except possibly in the case $G_r = \SU(r)$$)$.
\end{prop}
\begin{proof} The homology of the James reduced product $J(T_r)$ is the tensor
algebra $\mathcal{T} \left(\wt{H}_* (T_r)\right)$ generated by the reduced homology of the maximal torus. Define
$$V_r:=\widetilde H_*(T_r)=\bigoplus_{i=1}^r H_i(T_r).$$
The action of the Weyl group preserves homological 
degree in $H_*(T_r)$ and the tensor grading in the direct sum decomposition
$$
\mathcal{T}[V_r]= \bigoplus_{n \geqslant 0} V_r^{\otimes n} = \bigoplus_{n \geqslant 0} 
\left( \bigoplus_{i=1}^r H_i(T_r)\right)^{\otimes n}.
$$
Then we have
$$
H_k(J(T_r)) \isom \bigoplus_{n = 0}^m 
\left( \bigoplus_{\Sigma i_j=k} H_{i_1}(T_r) \otimes \cdots \otimes H_{i_n}(T_r)\right),
$$
and this extends to a decomposition of $\FI_W$--modules.
Each factor $\{H_{i_j}(T_r)\}_{r\geqs 1}$ is generated in stage $i_j$ 
by Corollary~\ref{cor: T-rep-stable}, so by Proposition~\ref{prop: tensor}, $\{H_k(J(T_r))\}_{r\geqs 1}$ is generated in stage $k =\Sigma i_j$.
\end{proof}



Recall from Theorem~\ref{Stafa}  
the homeomorphism  
\begin{equation}\label{eqn: j7} T^n/W \srm{\homeo} 
\Hom(\Z^n,G)_1/G = \Rep(\Z^n, G)_1.
\end{equation}
Similarly, we can obtain the following homeomorphisms.

\begin{prop}\label{prop: homeo}
Let $G$ be a compact Lie group. Then there are homeomorphisms 
$$
\Comm(G)_1/G\cong J(T)/W,
$$
and
$$
\Bcom  (G)_1/G \cong BT/W.
$$
\end{prop}
\begin{proof}
The first homeomorphism was shown in \cite[Theorem 1.2]{stafa2017poincare}. 
The second homeomorphism follows from the homeomorphism (\ref{eqn: j7}),
which yields a simplicial homeomorphism between the simplicial spaces
$n\goesto T^n/W$ (whose geometric realization is $BT/W$) and $n\goesto \Rep(\Z^n, G)_1$ (whose realization is $\Bcom  (G)_1/G$).
\end{proof}

 In Section~\ref{sec: stable range for Hom}, we establish homological stability for the sequences
$$r\goesto \Bcom (G_r)_1 \,\,\, \textrm{ and }\,\,\, r\goesto \Comm(G_r)_1.$$
Here we consider the conjugation quotients, which are simpler to handle.

\begin{thm}\label{thm: stability for BcomG/G and ComG/G}
The sequences
$$r\goesto \Bcom (G_r)_1/G_r \,\,\, \textrm{ and }\,\,\, r\goesto \Comm(G_r)_1/G_r$$
satisfy strong rational homological stability. In homological degree $k$, stability holds for $r\geqslant k$.
\end{thm}
\begin{proof} The proof is similar to that of Theorem~\ref{stability-wrt-r-Rep}.
We will show that the actions of $W$ on $J(T)$ and on $BT$ are good. 
Then by Proposition~\ref{prop: homeo} we have
$$H_*(\Bcom  (G)_1/G)\isom H_* (BT)^W,$$
 and
$$H_*(\Comm(G)_1/G)\isom H_*(J(T))^W,$$ 
and except in the case of $\SU(r)$, the result will follow from Propositions~\ref{prop: stability J} and~\ref{prop: stability BT} together with Theorem~\ref{quot-stab} 
Note that $J(T)$ is a CW complex, and May~\cite[Appendix]{May-EGP} shows that $BT$ has the homotopy type of a CW complex, since it is the geometric realization of a proper simplicial space in which each level is CW complex. The space $BT/W$ has the homotopy type of a CW complex for the same reason: this space is the geometric realization of a simplicial space $n\goesto T^n/W$, and $T^n/W$ is triangulable for each $n$ (since the action of $W$ is algebraic, the quotients are semi-algebraic by Schwarz~\cite{Schwarz}).
We claim that $J(T)/W$ also has the homotopy type of a CW complex. This space is the colimit of the diagram formed by all maps $T^n/W\to T^{n+1}/W$ given by inserting the identity element in one coordinate. These maps are cofibrations, so this colimit is homotopy equivalent to the corresponding homotopy colimit. 

 In the case $G_r = \SU(r)$, following the argument in~\ref{stability-wrt-r-Rep}, it suffices to prove injectivity of the maps $BT(\SU(r))\to BT(\U(r))$ 
and $JT(\SU(r))\to JT(\U(r))$ 
 induced by the inclusion $\SU(r)\injects \U(r)$.

For the classifying spaces, the determinant map $T'_r\to S^1$ induces a simplicial map $BT(\U(r)) \to BS^1$. This map is a level-wise Hurewicz fibration and hence (by~\cite[Theorem 12.7]{May-GOILS}) a fibration on geometric realizations, with fiber $BT(\SU(r))$. Since the homology of $BS^1\heq \C P^\infty$ is concentrated in even degrees, the Serre spectral sequence for this fibration collapses at the $E^2$ page, and hence the inclusion of the fiber is injective in homology, as desired.

For the James constructions, recall that we saw in the proof of Theorem~\ref{stability-wrt-r-Rep} that $T_r\injects T_{r+1}$ is injective on homology, and since we are working rationally this map admits a splitting. To obtain the homology of the James construction, we take the tensor algebras, and functoriality  of this construction implies that we again have a split injection between the tensor algebras. 
\end{proof}

\begin{rmk}\label{rmk:  Bcom integral stability}
For $G=\U(r)$, Proposition~\ref{prop: homeo} implies that
\begin{equation}\label{cg}
\Bcom (G)_1/G  \cong   BT/W   \isom   \Sym^r(BS^1).
\end{equation}
Hence in this case the homological stability result in Theorem \ref{thm: equiv.coh. stability Hom, Comm,Bcom} holds integrally, as explained in  Remark~\ref{rmk: Steenrod symmetric prod}. In light of Remark~\ref{rmk: signed sym product = sym product}, the same holds in types B and C. 

We note that in (\ref{cg}) we are using the homeomorphism $BT\homeo (BS^1)^r$, 
which holds when the product $(BS^1)^r$ is given the compactly generated topology~\cite[Corollary 11.6]{May-GOILS}. 
Since geometric realizations of simplicial Hausdorff spaces are Hausdorff~\cite{Pazzis}, using the compactly generated topology 
does not affect the weak homotopy type of the symmetric product
(as discussed in Remark~\ref{rmk: Steenrod symmetric prod}).
\end{rmk}

\section{Stability for equivariant (co)homology}\label{sec: equivariant stability}

In this section we prove that the $G_r$--equivariant (co)homology of 
$\Hom(\Z^n,G_r)_1$, $\Comm(G_r)_1$, and $(\Bcom G_r)_1$, 
where $G_r$ acts by conjugation, also stabilizes for sufficiently large $r$.
In the case of $\Hom(\Z^n,G)_1$, equipped with the permutation action of the 
symmetric group $S_n$, we show that the $G$--equivariant (co)homology forms 
a representation stable sequence of $\bbQ[S_n]$--modules. We emphasize the cohomological statements in this section, which are needed in the next section.

We will need a result from Baird \cite{baird2007cohomology} regarding 
rational equivariant (co)homology. 
As naturality will be crucial for us, we explain in detail the maps involved in this isomorphism. Given a left $G$--space $X$, the inclusion $X^T \injects X$ induces a map
\begin{equation}\label{eqn: X^T X} EG\cross_T X^T \maps EG\cross_G X,
\end{equation}
where $EG$ is a universal principal (right) $G$--bundle. (We will sometimes denote the homotopy orbit space $EG\cross_G X$ by $X_{hG}$.)
The Weyl group $W=NT/T$ acts on $EG\cross_T X^T \isom EG/T \cross X^T$ via $([e], x)\cdot [n] = ([e\cdot n], n^{-1} \cdot x)$, and the map (\ref{eqn: X^T X}) is invariant under this action, yielding an induced map
\begin{equation}\label{eqn: iota} \iota\co (EG\cross_T X^T)/W \maps EG\cross_G X
\end{equation}
sending $[e, x]$ to $[e,x]$.

\begin{thm}[{\cite[Theorem 3.5]{baird2007cohomology}}]\label{thm: Baird G-equiv cohom thm}
Let $G$ be a compact and connected Lie group acting on a paracompact Hausdorff 
space $X$ such that for every $x\in X$, there exists a maximal torus $T(x)\leqs G$ such that $T(x) \leqs G_x$ $($where $G_x\leqs G$ denotes the stabilizer subgroup$)$. 
Then for each maximal torus $T\leqs G$,
the map $\iota$
induces  an isomorphism 
$$\iota^*\co H_G^*(X) \srm{\isom} H^*\left(\left(EG\cross_T X^T\right)/W\right)$$
in rational cohomology. 
\end{thm}

 The action of $W$ on $H^*(EG/T)$ in Theorem~\ref{thm: Baird G-equiv cohom thm} is induced by the \e{principal action} of $NT\leqs G$ on $EG$, which descends to an action of $W = NT/T$ on $EG/T$.
Recall (Lemma~\ref{lem: translation=conjugation}) that up to homotopy, this action  agrees with the  conjugation action of $W$ on $BT$.

\begin{rmk}\label{rmk: natural} The map $\iota$ in Theorem~\ref{thm: Baird G-equiv cohom thm} is natural in the following sense. Say $h\co G\to G'$ is a continuous homomorphism and 
$h(T) \leqs T'$ for some maximal tori $T\leqs G$ and $T'\leqs G'$. If $X$ and $X'$ are $G$ and $G'$--spaces (respectively), and $f\co X\to X'$ is an $h$--equivariant map (meaning that $f(g\cdot x) = h(g)\cdot f(x)$ for all $g\in G$, $x\in X$) satisfying
$$f(X^T) \subset (X')^{T'},$$
then we have a commutative diagram
  \begin{center}
\begin{tikzcd} 
(EG\cross_T X^T)/W \arrow[d] \arrow{r}{\iota} 
  & EG\cross_G X \arrow[d]\\
	 (EG'\cross_{T'} (X')^{T'})/W   \arrow{r}{\iota}
  & EG'\cross_{G'} X'
\end{tikzcd}
\end{center}
in which the vertical maps are induced by $h$ and $f$.
\end{rmk}

For completeness, we sketch the proof of Theorem~\ref{thm: Baird G-equiv cohom thm}.

\begin{proof}[Proof of Theorem~\ref{thm: Baird G-equiv cohom thm}]
To begin, let $E$ be a right $G$--space and let $H\leqs G$ be a subgroup. Then 
$G$ acts on $E\cross G/H$ via $(e, [g])\cdot g' = (e\cdot g', [(g')^{-1} g])$,
and we have a homeomorphism
\begin{equation}\label{eqn: ExG/H}
\psi \co E/H \srm{\isom} E\cross_G (G/H)  
\end{equation}
induced by the inclusion $E\injects E\cross (G/H)$, $e\goesto (e, [1])$ (the inverse of (\ref{eqn: ExG/H}) is the map induced by $(e, [g])\goesto [e\cdot g])$.
Next, let $W = NH/H$, and let $Y$ be a left $W$--space. Then the induced homeomorphism
\begin{equation}\label{eqn: ExG/HxY}
\psi\cross \textrm{Id}_Y\co  E\cross_H Y  =  E/H\cross Y \srm{\isom} \left(E\cross_G (G/H)\right) \cross Y,
\end{equation}
is  $W$--equivariant, 
where $[n]\in W$ acts on $(E\cross_H Y)/W$ via $[e, y]\cdot [n] = [e\cdot n, n^{-1} \cdot y]$ and on $\left(E\cross_G (G/H)\right) \cross Y$ via $(e, [g], y) \cdot [n] = (e, [gn], n^{-1} \cdot y)$. Hence $(\ref{eqn: ExG/HxY})$ descends to a homeomorphism
\begin{equation}\label{eqn: homeo2} (E\cross_H Y)/W \srm{\isom} (\left(E\cross_G (G/H)\right) \cross Y)/W \isom 
 E\cross_G (G/H\cross_W Y).
 \end{equation}

Now, consider   
a left $G$--space $X$. Then the fixed point space $X^H$ is invariant under $NH$, and hence inherits an action of $W$. Setting $Y = X^H$ in (\ref{eqn: homeo2}) gives a homeomorphism
\begin{equation}\label{eqn: homeo3} (E\cross_H X^H)/W \srm{\isom} (\left(E\cross_G (G/H)\right) \cross X^H)/W \isom 
 E\cross_G (G/H\cross_W X^H).
 \end{equation}
The map 
\begin{equation}\label{eqn: phi}\phi\co G/H\cross_W X^H \maps X\end{equation}
induced by $([g], x)\goesto (g\cdot x)$ is $G$--equivariant (where in the domain of $\phi$, $G$ acts by left translation on $G/H$, and trivially on $X^H$)
so we have an induced map
\begin{equation}\label{eqn: phiG} \phi_G \co E\cross_G \left(G/H\cross_W X^H\right) \maps E\cross_G X.
\end{equation}
Composing (\ref{eqn: homeo3}) and (\ref{eqn: phiG}) yields a map
$$(E\cross_H X^H)/W \maps E\cross_G (G/H\cross_W X^H) \maps E\cross_G X$$
sending $[e, x]$ to $[e, x]$, so it remains only to show that (\ref{eqn: phiG}) induces an isomorphism in rational cohomology. The proof of
Baird~\cite[Theorem 3.3]{baird2007cohomology} shows that the underlying $G$--equivariant map 
(\ref{eqn: phi}) is an isomorphism in rational cohomology. When $E=EG$, the domain and range of (\ref{eqn: phiG}) fiber over $EG/G = BG$, and by comparing the Serre spectral sequences for these fibrations we see that  (\ref{eqn: phiG}) induces an isomorphism in rational cohomology, as desired.
\end{proof}

When $X$ is $\Hom(\Z^n,G)$, $\Comm(G)$, 
or $\Bcom G$, with $G$ acting by conjugation, we have the following corollary of Theorem~\ref{thm: Baird G-equiv cohom thm}. 
 
\begin{cor}\label{cor: G-equiv. cohomology of hom, comm, b_com}
For each compact, connected Lie group $G$, the inclusion $T\injects G$ of a maximal torus induces maps
\begin{enumerate}
\item $(EG/T \cross T^n)/W \maps EG\cross_G \Hom(\Z^n, G)_1 $
\item $(EG/T \cross J(T))/W \maps EG\cross_G \Comm(G)_1$ 
\item $(EG/T \cross BT)/W \maps EG\cross_G \Bcom(G)_1$,
\end{enumerate}
each of which induces an isomorphism in rational cohomology:
\begin{enumerate}
\item $H^*_G(\Hom(\Z^n,G)_1) \srm{\isom}   H^*(BT\cross T^n)^{W}$,
\item $H^*_G(\Comm(G)_1) \srm{\isom}  H^*(BT \cross J(T))^{W}$,
\item $H^*_G(\Bcom (G)_1) \srm{\isom} H^*(BT \cross BT)^{W}$.
\end{enumerate}
\end{cor}

\begin{proof} Item (1) is due to Baird \cite[Corollary 4.4]{baird2007cohomology}, 
and follows from Theorem~\ref{thm: Baird G-equiv cohom thm} and Proposition~\ref{coh-quot} by taking $X = \Hom(\Z^n, G)_1$. The key point is that each $n$--tuple in $\Hom(\Z^n, G)_1$ lies in a maximal torus (note also that $\Hom(\Z^n, G)_1^T = T^n$, since if $(g_1, \ldots, g_n)$ is  fixed by all $t\in T$, then each $g_i$ lies in the centralizer  $Z(T)$, which is just $T$ itself). Note that the action of $W$ on $BT\cross T^n$ is good; the proof is similar to the argument in the proof of Theorem~\ref{thm: stability for BcomG/G and ComG/G}.

The other two cases are similar, by taking $X = \Comm(G)_1$ or $X = \Bcom(G)_1$ in Theorem~\ref{thm: Baird G-equiv cohom thm}. We begin by showing that both of these spaces are paracompact and Hausdorff, and that the actions of $W$ on $BT\cross J(T)$ and $BT\cross BT$ are good. First, $\Comm(G)_1$ is a closed subspace of the CW complex $J(G)$, and in general closed subspaces of paracompact spaces are paracompact. 
Pazzis~\cite{Pazzis} implies that the simplicial space $\Bcom(G)_1$ is Hausdorff.
To see that  $\Bcom(G)_1$ is paracompact, note that this space, being a geometric realization, may be written as the colimit of its skeleta, which are all compact. Moreover, the inclusion of one skeleton into the next is a closed embedding, since the skeleta are compact Hausdorff. In general, it follows from Michael's theory of selections~\cite{Michael-selection} that colimits of closed embeddings between paracompact Hausdorff spaces are paracompact (see~\cite{nLab-colim}). 
The proof that the action on $BT\cross BT$ is good is similar to the case of $BT\cross T^n$. For $BT\cross J(T)$, note that since $T$ is a CW complex, so is $J(T)$.
The quotient space $(BT\cross J(T))/W$ is the geometric realization of a simplicial space $n\goesto (T^n\cross J(T))/W$. We showed in the proof of Theorem~\ref{thm: stability for BcomG/G and ComG/G} that $J(T)/W$  has the homotopy type of a CW complex, and a similar argument applies to $(T^n\cross J(T))/W$. We conclude, using May~\cite[Appendix]{May-EGP}, that $(BT\cross J(T))/W$ has the homotopy type of a CW complex as well.
 
Next, we need to verify that each stabilizer in $\Comm (G)_1$ and in $\Bcom  (G)_1$ contains a maximal torus. But this follows immediately from the fact that each $n$--tuple in $\Hom(\Z^n, G)_1$ lies in a maximal torus.
To complete the proof, we need to check that
$\Comm(G)_1^{T}=J({T})$ and $\Bcom (G)_1^{T}=B{T}$.
Each point in $\Comm (G)_1$ has a unique non-degenerate representative $(g_1, \ldots, g_n)$ with $g_i \neq 1$, and if this point is fixed by all $t\in T$, then as above we find that $g_i\in Z(T) = T$ for each $i$. The equality $\Bcom (G)_1^{T} = BT$ follows similarly, using the fact that fixed points commute with geometric realization.
\end{proof}

\begin{thm}\label{thm: equiv.coh. stability Hom, Comm,Bcom}
Each of the following sequences of spaces satisfies strong rational $($co$)$homological stability:
\begin{align*}
r&\goesto EG_r\cross_{G_r} \Hom(\Z^n, G_r)_1,\\
r&\goesto EG_r\cross_{G_r} \Comm(G_r)_1,\\
r&\goesto EG_r\cross_{G_r} \Bcom(G_r)_1.
\end{align*}
In $($co$)$homological degree $k$, stability holds for $r\geqslant k$.
\end{thm}
\begin{proof} We begin by considering the homomorphism spaces.
Using the K\"unneth Theorem and Corollary \ref{cor: G-equiv. cohomology of hom, comm, b_com} we have isomorphisms
$$H^{G_r}_k(\Hom(\Z^n,G_r)_1)\cong \bigoplus_{i+j=k}(H_i({BT_r}) \otimes H_j(T_r^n)  )^{W_r}$$
that are natural in $r$, and by Lemma~\ref{lem: translation=conjugation}, the action of $W_r$ on $H_i({BT_r})$ is that induced by conjugation.
Each term $H_i({BT_r}) \otimes H_j(T_r^n) $ in the direct sum is an 
$\FI_W$--module finitely generated in stage $\leqslant i+ j= k$
by Corollary~\ref{cor: T-rep-stable} and Proposition~\ref{prop: stability BT}, along with
Proposition~\ref{prop: tensor}. By Proposition~\ref{prop: T is FI}, when $G_r = \U(r)$, $\Sp(r)$, or $\SO(2r+1)$, the $\FI_W$--modules in question are in fact $\FI_W\#$--modules, and when $G_r = \SO(2r)$, the $\FI_D$--modules  in question are restrictions of the $\FI_B$--modules associated to $\SO(2r+1)$.
Theorem~\ref{quot-stab} now implies that the  $W_r$--invariants in these modules stabilize for $r\geqslant k$. 

Finally, we address the case $G_r = \SU(r)$. 
Let $T_r' = T(\U(r)) \leqs \U(r)$ denote the diagonal maximal torus.
As in the proof of 
Corollary~\ref{cor: T-rep-stable}, it suffices to show that the maps
$$H_* (BT_r \cross T_r^n) \maps H_* (BT'_r \cross (T'_r)^n)$$
(induced by the inclusion $T_r \injects T_r'$) are injective.  In that proof we established injectivity of $H_*(T_r^n)\to H_*((T'_r)^n)$, and injectivity of $H_* (BT_r) \to H_* (BT'_r)$ was established in the proof of Proposition~\ref{prop: stability BT}. 
The K\"unneth decomposition now completes the proof in this case.

Similar arguments apply to $\Comm(G_r)_1$ and $\Bcom (G_r)_1$, using Propositions~\ref{prop: stability J} and~\ref{prop: stability BT}.
\end{proof}

\begin{rmk}\label{rmk: Hom, Bcom integral equivariant stability}
For $G=\U(r)$, $\Sp(r)$, or $\SO(2r+1)$, Corollary \ref{cor: G-equiv. cohomology of hom, comm, b_com} implies that
\begin{enumerate}
\item $H_*^G(\Hom(\Z^n,G)_1) {\isom}   H_*(BT\cross T^n)^{W} {\isom}   H_*(\Sym^r(BS^1 \cross (S^1)^n))$, and
\item $H_*^G(\Bcom (G)_1) {\cong} H_*(BT \cross BT)^{W} {\isom}   H_*(\Sym^r(BS^1 \cross BS^1))$.
\end{enumerate}
\end{rmk}

\begin{thm}\label{thm: equiv.coh. stability Hom mod S_n}
Let $G$ be a compact and connected Lie group.
The sequence of $S_n$--representations
$$
n\goesto H^G_k(\Hom(\Z^n,G)_1)  
$$
is uniformly representation stable with stable range $n\geqslant 2k.$
Consequently, the sequence
$$n\goesto \Hom(\Z^n,G)_1/S_n$$
satisfies strong $G$--equivariant rational homological stability, and in homological degree $k$, stability holds for $r\geqs k$.
\end{thm}
\begin{proof}
Consider the action of the symmetric group $S_n$ on $EG \times_G \Hom(\Z^n,G)_1$ 
that is trivial on $EG$ and permutes the coordinates of $n$--tuples in $\Hom(\Z^n,G)_1.$
The latter action of $S_n$ commutes with the conjugation action of $G$, giving a well-defined $S_n$--action on the homotopy orbit space.

Recall that the conjugation map
$$\phi \co  G/T\times_W T^n \maps \Hom(\Z^n,G)_1$$
is both $S_n$--equivariant and $G$--equivariant (where on the left, $S_n$ acts trivially on $G/T$ and by permutations on $T^n$, while $G$ acts by left-translation on $G/T$ and trivially on $T^n$). 
Since $\phi$ is $G$--equivariant, it  induces a map of fibration sequences
\begin{center}
\begin{tikzcd}
 G/T\times_W T^n \arrow{d}{\phi} \arrow{r} &  (G/T\times_W T^n)_{hG}\arrow{d}{\phi_{hG}} \arrow{r}
 	& BG \ar{d}{=}   \\
\Hom(\Z^n,G)_1  \arrow{r}  & (\Hom(\Z^n,G)_1)_{hG} \arrow{r} & BG,
\end{tikzcd}
\end{center}
where in the middle column we use the notation $X_{hG}:=EG\cross_G X$.
By Theorem~\ref{Baird}, $\phi$ induces an isomorphism in rational (co)homology, and a comparison of the Serre spectral sequences for these fibrations shows that the induced map $\phi_{hG}$ between homotopy orbit spaces is also an isomorphism in rational 
(co)homology. Note that $S_n$--equivariance of $\phi$ implies $S_n$--equivariance of $\phi_{hG}$. 

Next, we have an $S_n$--equivariant homeomorphism 
$$(G/T\times_W T^n)_{hG} \isom ((G/T\times T^n)_{hG})/W = (EG\cross_G (G/T\times T^n))/W,$$
where on the right $W$ acts trivially on $EG$.
Furthermore, 
we have an $S_n$--equivariant homeomorphism \cite[eq. (28)]{baird2007cohomology}
$$
EG \times_G (G/T \times T^n) \cong EG \times_T T^n = EG/T \cross T^n
$$
given by $[e, gT, t]\goesto [e\cdot g,  t]$ (with inverse $[e, t]\goesto [e, T, t]$). This homeomorphism is also $W$--equivariant if we give $EG\times_T T^n$ the action $[e, t]\cdot [n] = [en, n^{-1} t n]$. 

Lemma~\ref{lem: translation=conjugation} now gives 
an isomorphism of $\FI$--modules 
$$\{H_k^G (\Hom(\Z^n, G))\}_{n\geqs 1} \isom \{H_k (BT \cross T^n)^W\}_{n\geqs 1},$$
where $W$ acts on $BT\heq EG/T$ by conjugation.
So it will suffice to show that $\{H_k (BT \cross T^n)^W\}_{n\geqs 1}$ is generated in stage $k$ (note that both of these are in fact $\FIs$--modules).

By Lemma~\ref{W-invt-fin-gen}, it will suffice to show that the $\FI$--module
$\{H_k (BT \cross T^n)\}_{n\geqs 1}$ is generated in stage $k$.
The K\"unneth Theorem gives
$$H_k (BT \cross T^n) \isom \bigoplus_{i+j=k} H_i(BT)\otimes H_j(T^n),$$
which is a decomposition of $\FI$--modules (where the $\FI$--structure on $ H_i(BT)$ is trivial). By Corollary~\ref{cor: T-rep-stable}, each term in this direct sum decomposition is generated in stage $k$, and hence the same is true for the sum.

For the last statement of the Theorem, note that we have a homeomorphism
$$EG\cross_G (\Hom(\Z^n, G)/S_n) \isom (EG\cross_G \Hom(\Z^n, G))/S_n,$$
where on the right, $S_n$ acts trivially on $EG$. The statement now follows from 
Proposition~\ref{coh-quot2},
because $EG\cross_G \Hom(\Z^n, G)$ has the homotopy type of a CW complex; indeed it is the geometric realization of a simplicial space $k\goesto (G^{k+1} \cross \Hom(\Z^n, G))/G$, and each of these quotients is triangulable 
by Schwarz~\cite{Schwarz} (as in the proof of Theorem~\ref{thm: stability for BcomG/G and ComG/G}).
\end{proof}

\begin{rmk} In Sections~\ref{sec: nil} and~\ref{sec: covers} we extend the homological stability results in this article in several directions.  Remarks~\ref{rmk: eqvt11} and~\ref{rmk: eqvt12} discuss extensions of the results in the present section.
\end{rmk}

\section{Stability bounds for classical sequences of Lie groups}\label{sec: stable range for Hom}

In this section, we derive homological stability bounds for $\{\Hom(\Z^n, G_r)\}_{r\geqslant 1}$, where $n$ is fixed,
using the Eilenberg--Moore spectral sequences associated to the fibrations
\begin{equation}\label{eq: EM} \Hom(\Z^n, G_r)\maps \Hom(\Z^n, G_r)_{hG_r}\maps BG_r 
\end{equation} 
(and similarly for $\Bcom  (G_r)$ and $\Comm(G_r)$ in place of $\Hom(\Z^n, G_r)$).
We refer to McCleary~\cite{mccleary2001ssbook} and Smith~\cite{smith_Eilenberg-Moore-SS}  for background on the 
Eilenberg--Moore spectral sequence (in particular, see \cite[Theorem 6.1]{smith_Eilenberg-Moore-SS}.) We note that these sources place the spectral sequence in the second quadrant, so that the differentials are of ``cohomological type" and the total cohomological degree of $E_2^{p,q}$ is $p+q$. In order to simplify notation in the arguments to follow, we will reindex this as a first quadrant spectral sequence by reflecting across the  vertical axis (that is, the $q$--axis).

\begin{thm}\label{thm: EM} Let $E\to B$ be a fibration, with $B$ simply connected, and consider a map $f\co X\to B$. Let $X \stackrel[B]{\bigtimes}{} E$ denote the pullback of $E$ along $f$. If all four spaces are of finite type, then there is a first quadrant spectral sequence with
$$E_2^{p,q} = \Tor^{H^*(B; \Q)}_{p,q} (H^*(X; \Q); H^*(E; \Q))$$
converging to $H^*(X \stackrel[B]{\bigtimes}{} E; \Q)$. The differential on the $m^{\textrm th}$ page has the form
$$d_m \co  E_m^{p,q}\maps E_2^{p-m,q-m+1}.$$
A commutative diagram of pullback squares induces a map of spectral sequences, and on the $E_2$ page this map agrees with the induced map between $\Tor$ groups.
\end{thm} 
 
We will only need the last statement for the case when $X$ is a point, and we spell out the statement in more detail below (Corollary~\ref{cor: EM}).

Some additional comments are in order. First, convergence means that for each $k\geqs 0$, 
there exists $M = M(k)$ such that for $m\geqs M(k)$, the groups
$E_m^{p,q}$ with $q=p+k$ form the associated graded group of a filtration on 
$H^{k} (X \stackrel[B]{\bigtimes}{} E;\Q)$. Next, to understand the groups $\Tor^{H^*(B; \Q)}_{p,q} (H^*(X; \Q); 
H^*(E; \Q))$, we consider $H^*(X; \Q)$ as a graded module over the 
graded ring $H^*(B;\Q)$ via the map $f^*$. Given a graded 
module $M$ over a graded ring $R$, there exists a resolution of $M$ by 
grading-preserving maps between graded free $R$--modules (here freeness just 
refers to the underlying ungraded module). Tensoring such a resolution with 
$H^*(E;\Q)$ (over $H^*(B; \Q)$) yields a graded chain complex. If we consider the 
sub-chain complex consisting of elements in grading $q$, then the $p^{\textrm th}$ 
homology of this complex is independent of the chosen resolution, and is the 
group $\Tor^{H^*(B; \Q)}_{p,q} (H^*(X; \Q); H^*(E; \Q))$.

Taking $X$ to be a point yields the following special case of Theorem~\ref{thm: EM}.

\begin{cor}\label{cor: EM} Let $F\to E\to B$ be a fibration, with $B$ simply connected and 
all spaces of finite type. Then there is a first quadrant spectral sequence with
$$E_2^{p,q} = \Tor^{H^*(B; \Q)}_{p,q} (\Q; H^*(E; \Q))$$
converging to $H^*(F; \Q)$. The differential on the $m^{\textrm th}$ page has the form
$$d_m \co  E_m^{p,q}\maps E_2^{p-m,q-m+1}.$$

A commutative diagram
  \begin{equation}\label{fib-map} 
\begin{tikzcd} 
F \arrow[d] \arrow[rr] 
 & & F'\arrow[d]\\
	  E \arrow[rd]  \arrow{rr}{f}
	&  & E' \arrow[ld]  
	 \\
&B  
	 \\
\end{tikzcd}
\end{equation}
of fibrations over $B$ induces a 
map between the associated spectral sequences, and on the $E_2$ page this map agrees with the map
$$\Tor^{H^*(B; \Q)}_{p,q} (\Q; H^*(E'; \Q)) \maps \Tor^{H^*(B; \Q)}_{p,q} (\Q; H^*(E; \Q))$$
induced by the map of $H^*(B; \Q)$--modules $f^*\co H^*(E'; \Q) \to H^*(E; \Q)$.
\end{cor}

The $\Tor$ groups in the spectral sequences used below are computed over the  rational cohomology rings $H^*(BG_r)$.
We now recall the structure of these rings and of the maps
\begin{equation}\label{HBG} (Bi)^*\co H^*(BG_{r+1})\maps H^*(BG_r)\end{equation}
induced by the standard inclusions $i\co G_r \injects G_{r+1}$. 

\begin{prop}\label{prop: BG coh}
We have the following isomorphisms:
\begin{enumerate}
	\item $H^*(B\U(r))\isom \bbQ[c_1, \ldots, c_r]$, with $|c_i| = 2i$.
	\item $H^*(B\SU(r))\isom \bbQ[c_2, \ldots, c_r]$, with $|c_i| = 2i$.
	\item $H^*(B\Sp(r))\isom \bbQ[p_1, \ldots, p_r]$, with $|p_i| = 4i$.
	\item $H^*(B\SO(2r+1)) \isom \bbQ[p_1, \ldots, p_r]$, with $|p_i| = 4i$.
	\item For $r>1$, $H^*(B\SO(2r))\isom   \bbQ[p_1, \ldots, p_{r-1}, y_{2r}]$, with $|p_i| = 4i$, $|y_{2r}| = 2r$.
\end{enumerate}
The maps (\ref{HBG}) have the following behavior on the above polynomial generators:
\begin{enumerate}
	\item For $G_r = \U(r)$ or $\SU(r)$, we have $c_i \goesto c_i$ for $i<r$ and $c_{r+1}\goesto 0$.
	\item For $G_r = \Sp(r)$ or  $\SO(2r+1)$,   we have $p_i \goesto p_i$ for $i<r$ and $p_{r+1}\goesto 0$.
	\item For $G_r = \SO(2r)$   and $r>1$, we have $p_i \goesto p_i$ for $i<r-1$ and $p_{r}, y_{2r}\goesto 0$.
\end{enumerate}
\end{prop}
\begin{proof} This follows from the arguments in~\cite[\S III.3]{Toda-Mimura}.
Briefly, one first calculates  $H^* (G_r)$ by analyzing the spectral sequence for the fibration 
$$G_r\maps G_{r+1} \maps G_{r+1}/G_r,$$
finding that it is exterior on generators in one degree less than the above polynomial generators for $H^*(BG_r)$. The structure of these spectral sequences also determines the map $H^*(G_{r+1})\to H^*(G_r)$. The spectral sequences for the fibrations 
$$G_r\maps EG_r \maps BG_r$$
then determine the structure of $H^*(BG_r)$, and a comparison of these spectral sequences for  $r$ and $r+1$ shows that the maps (\ref{HBG}) are determined by the corresponding maps  $G_r\injects G_{r+1}$ between the fibers of these fibrations.
\end{proof}

\begin{rmk} In each of the above cases, we can rename the polynomial generators as $x_1, \ldots, x_r$ so that their degrees are non-decreasing. This will be implicit in the arguments to follow.
For $G_r = \SO(2r)$, there is sometimes an ambiguity in this choice of ordering, since $|y_{2r}|= |p_{r/2}|$ when $r$ is even. This will not affect the arguments. 
\end{rmk}

\begin{lem}\label{lem: EM}
The map
$$H^k_{G_{r+1}} (\Hom(\Z^n, G_{r+1})_1) \maps H^k_{G_{r}} (\Hom(\Z^n, G_{r+1})_1)$$
induced by the standard inclusion $G_r\injects G_{r+1}$ is an isomorphism for 
$k\leqs 2r$, and similarly for $\Bcom(-)$ or $\Comm(-)$ in place of $\Hom(\Z^n, -)$.
\end{lem}
\begin{proof} We will study the Eilenberg--Moore spectral sequence associated to the pullback diagram
\begin{center}
\begin{tikzcd}
EG_{r}\cross_{G_{r}} \Hom(\Z^n, G_{r+1})_1 \arrow{r}  \arrow{d}   
	& EG_{r+1}\cross_{G_{r+1}} \Hom(\Z^n, G_{r+1})_1 \arrow{d} 
	 \\
BG_{r} \arrow{r}	 &BG_{r+1}
\end{tikzcd}
\end{center}
induced by the standard inclusion $i\co G_r \injects G_{r+1}$.

To describe the $E_2$ page of this spectral sequence, we need a graded resolution of $H^*(BG_{r}; \Q)$ as a graded module over $H^*(BG_{r+1}; \Q)$. In all cases other than $G_r = \SO(2r)$, we see that $H^*(BG_{r}; \Q)$ is simply the quotient of $H^*(BG_{r+1}; \Q)$ by the ideal generated by the polynomial generator in the highest grading; for ease of notation we denote this generator by $x$. This gives us a 2-step free resolution
\begin{equation}\label{eq: res}
0 \lmaps H^*(BG_{r}; \Q)\stackrel{(Bi)^*}{\lmaps} H^*(BG_{r+1}; \Q) \stackrel{x\cdot}{\lmaps} \Sigma^{|x|} H^*(BG_{r+1}; \Q)
\end{equation}
where $\Sigma^d$ is the operator on graded modules that adds $d$ to all gradings, and the right-hand map in (\ref{eq: res}) is multiplication by $x$. 
Note that the shift in grading in the first term of this sequence makes multiplication by $x$ grading-preserving. To compute the $\Tor$ groups on the $E_2$ page of the spectral sequence, we tensor this resolution (over $H^*(BG_{r+1})$) with 
 $H^*_{G_{r+1}}(\Hom(\Z^n, G_{r+1})_1; \Q)$, which gives a graded chain complex concentrated in degrees $p=0$ and $p=1$. 
 The shift in grading implies that $E^{1,q}_2 = 0$ for $q \leqs 2r+1$, so that for $k \leqs 2r$ the only group contributing to the line of total cohomological degree $k$ (namely the line $q=p+k$) is $E_2^{0,k} \isom H_{G_{r+1}}^k (\Hom(\Z^n, G_{r+1})_1; \Q)$. Moreover, there is no room for non-trivial differentials in the spectral sequence, so this gives the desired isomorphism.

For $G_r = \SO(2r)$, let 
 $A = \Q[p_1, \ldots, p_{r}, y_{2r+2}] \isom H^*(BG_{r+1})$. 
 For $z\in A$, set  $A\tilde{z}:=\Sigma^{|z|} A$, and  write elements in $A\tilde{z}$ in the form $a\cdot\tilde{z}$ (so $|a\cdot\tilde{z}| = |a|+|z|$).
Similarly, we define  $A\,\tilde{z}\wedge \tilde{w}: =\Sigma^{|z|+|w|} A$, and  elements in $A\,\tilde{z}\wedge \tilde{w}$ are written in the form $a\cdot \tilde{z}\wedge \tilde{w}$.
 We can resolve $\Q[p_1, \ldots, p_{r-1}, y_{2r}]$ over $A$ as follows:
\begin{center}
\begin{tikzcd}[column sep=3ex, nodes={inner sep=2pt}] 
 0  
 &  \Q[p_1, \ldots, p_{r-1}, y_{2r}] \ar{l}  \ar[r, leftarrow, "\,\,\alpha"]
 &  A \oplus  \Sigma^{2r}A \ar[r, leftarrow, "\,\,\beta"]
&  ( A\tilde{y}_{2r+2} \oplus   A\tilde{p}_r) {\oplus} ( \Sigma^{2r} A \tilde{y}_{2r+2} \oplus\Sigma^{2r}A\tilde{p}_r)
\\
&&& (A\, \tilde{p}_r \wedge\tilde{y}_{2r+2} ) {\oplus} \Sigma^{2r} (A\,\tilde{p}_r\wedge\tilde{y}_{2r+2}) \ar[u, "\gamma"]
\end{tikzcd}
\end{center}
The maps $\alpha$, $\beta$, and $\gamma$, are defined by
$$\alpha (a, b) = (Bi)^*(a)+((Bi)^*(b))y_{2r},$$
$$\beta ((a\cdot \wt{y}_{2r+2}, a'\cdot \wt{p}_r ), (b\cdot \wt{y}_{2r+2}, b'\cdot \wt{p}_r )) = ( a y_{2r+2} + a' p_r   , b y_{2r+2} + b' p_r    ),$$
and
$$\gamma  (a \cdot \wt{p}_r \wedge\wt{y}_{2r+2}, b\cdot \wt{p}_r \wedge\wt{y}_{2r+2} )
	= ((ap_r \cdot \wt{y}_{2r+2}, -a y_{2r+2} \cdot \wt{p}_r), (bp_r \cdot \wt{y}_{2r+2}, -b y_{2r+2} \cdot \wt{p}_r) ).$$
The shift in grading again implies that $E^{1,q}_2 = 0$ for $q \leqs 2r+1$, 
and also that $E^{2,q}_2 = 0$ for $q \leqs 6r+1$,
so  for $k \leqs 2r$ the only group contributing to the line of total cohomological degree $k$ (namely the line $q=p+k$) is $E_2^{0,q} \isom H_{G_{r+1}}^k (\Hom(\Z^n, G_{r+1}); \Q)$. Moreover, there is no room for non-trivial differentials into the groups $E^{0,q}_r$ for $q \leqs 2r$.
\end{proof}

In the arguments below, we will resolve $\Q = H^*(\textrm{pt}; \Q)$ as a module over $H^*(BG_r)$
using the \e{Koszul complex}. 
We now recall the details of this construction, following 
Lang~\cite[Section XXI.4]{Langalgebra}. Lang works  in the ungraded setting, 
so we will explain how to include gradings. Let $A=\Q[x_1, \ldots, x_r]$, with 
grading satisfying $|x_1| \leqs |x_2| \leqs \cdots \leqs |x_r|$ (we give constant polynomials grading zero).
We now recall the definition of the (augmented) Koszul 
complex of $A$. This is a graded chain complex  
  \begin{center}
\begin{tikzcd} 
0 &  \Q \ar[l] \ar[r, leftarrow, "d_{0}"]& A = F_0 \ar[r, leftarrow, "d_{1}"]\ar[l]& \cdots   \ar[r, leftarrow, "d_{r-1}"] & F_{r-1} \ar[r, leftarrow, "d_{r}"] &F_r  
\end{tikzcd}
\end{center}
whose unaugmented portion (starting at $A$) will be denoted  $\K(A)$.
We define $F_p = F_p (A)$ to be the free $A$--module of rank ${r \choose p}$ (in 
particular, $F_0 = A$). The differential $d_0$ is simply the unital surjection sending all $x_i$ to zero. In order to describe the grading on $F_p$ and the 
differential $d_p \co F_p \to F_{p-1}$ for $p>0$, we adopt the following notation. For 
$p\geqs 1$, we will view $F_p$ as the free $A$--module on the set of 
formal symbols $\tilde{x}_{i_1} \sm \tilde{x}_{i_2} \sm \cdots \sm \tilde{x}_{i_p}$ with $1 \leqs i_1 < i_2 < \ldots < i_p \leqs r$. 
The grading on the submodule
$A\tilde{x}_{i_1} \sm \tilde{x}_{i_2} \sm \cdots \sm \tilde{x}_{i_p}\leqs F_p$ 
is defined so that the map $a\cdot \tilde{x}_{i_1} \sm \tilde{x}_{i_2} \sm \cdots \sm \tilde{x}_{i_p}\goesto a$ gives an isomorphism
$$A\tilde{x}_{i_1} \sm \tilde{x}_{i_2} \sm \cdots \sm \tilde{x}_{i_p} \isom \Susp^{|x_{i_1}|+\cdots+| x_{i_p}|} A$$
of graded $A$--modules.
The differential $d_p$ is defined on generators by
$$d_p(\tilde{x}_{i_1} \sm \tilde{x}_{i_2} \sm \cdots \sm \tilde{x}_{i_p}) = \sum_{j=1}^p (-1)^{j-1} x_{i_j} \cdot \tilde{x}_{i_1} \sm \cdots\sm \widehat{\tilde{x}_{i_j}} \sm \cdots \sm \tilde{x}_{i_p}.$$
The definition of the gradings ensures that $d_p$ preserves them, 
so the subgroups $$(F_p)_q \leqs F_p$$ 
consisting of homogeneous elements in grading $q$ form a subcomplex 
$\K(A)_q$ of $\K(A)$.

Since the sequence $x_1, \ldots, x_r$ is regular, the Koszul complex is 
exact~\cite[Theorem XXI.4.6]{Langalgebra}. Note that exactness in the 
ungraded sense immediately implies that the portion of the sequence in each 
homogeneous degree is also exact.

The next result will give a vanishing curve in our Eilenberg--Moore spectral sequences.

\begin{lem}\label{grading}
The Koszul resolution of $H^*(BG_r)$  satisfies $(F_p)_q = 0$ for 
$q < p(p+1)$ $($except in the case $G_r = \SO(2) \isom S^1$$)$.
\end{lem}
\begin{proof} 
The minimum shift in grading among the free summands in $F_p$ occurs for the summand corresponding to $x_1\sm x_2 \sm \cdots \sm x_p$, 
where the 
this shift is $\sum_{i=0}^p |x_i|$. By inspection, this sum is always smallest in the case $G_r = \U(r)$.
\end{proof}

We can now derive our stability bounds for the homology of spaces of commuting elements in the classical groups.

\begin{thm}\label{Hom-bound}
For each $n\geqs 1$, the sequences  
$$\{\Hom(\Z^n, G_r)_1\}_{r\geqs 1}, \,\,\,\,\{\Bcom  (G_r)_1\}_{r\geqs 1},
\,\,\,\,\textrm{and} \,\,\,\,
\{\Comm (G_r)_1\}_{r\geqs 1}$$
satisfy strong rational homological stability, and in homological degree $k$, stability holds once $r  - \lfloor \sqrt{r} \rfloor \geqs k$. 
\end{thm}

\begin{proof} We  work in cohomology. The arguments for $\Hom(\Z^n, -)$, $\Bcom(-) $, and $\Comm (-)$ are completely analogous, so we will focus on $\Hom (\Z^n, -)$.
We  consider the map of Eilenberg--Moore spectral sequences associated to the diagram of fibrations
  \begin{equation}
\begin{tikzcd} \label{fib-diag} 
\Hom(\Z^n, G_r)_1  \arrow[d] \arrow{rr}{i}
 & & \Hom(\Z^n, G_{r+1})_1\arrow[d]\\
	  (\Hom(\Z^n, G_r)_1)_{hG_r} \arrow[rd]  \arrow{rr}{j}
	&  & (\Hom(\Z^n, G_{r+1})_1)_{hG_{r}} \arrow[ld]  
	 \\
&BG_{r}  
	 \\
\end{tikzcd}
\end{equation}
where the horizontal maps $i$ and $j$ are induced by the standard inclusion.  
Consider the commutative diagram
\begin{center}
\begin{tikzcd}
H^k_{G_{r+1}} \left(\Hom(\Z^n, G_{r+1})_1\right)  \arrow{dr}  \arrow{d}{\iota} \\ 
H^k_{G_{r}} \left(\Hom(\Z^n, G_{r+1})_1\right)  \arrow{r}{j^*} 
	& H^k_{G_{r}} \left(\Hom(\Z^n, G_{r})_1\right), 
\end{tikzcd}
\end{center}
in which $\iota$ is the map from  Lemma~\ref{lem: EM}, and hence is an isomorphism for $r\geqs k/2$. By Theorem~\ref{thm: equiv.coh. stability Hom, Comm,Bcom}, the composite $j^*\circ \iota$ is an isomorphism for $r\geqs k$, and we conclude that $j$ induces an isomorphism in cohomology for $r\geqs k$.

To simplify notation,    
let $\H^*_r$ and $\H^*_{r+1}$ be the graded $H^*BG$--modules
$$\H^*_r = H_{G_r}^* \Hom(\Z^n, G_r)_1 \,\,\,\, \textrm{and}\,\,\,\, \H^*_{r+1} = H_{G_r}^* \Hom(\Z^n, G_{r+1})_1,$$
with module structures induced by the projections from the homotopy orbit spaces to $BG_r$. Denote the spectral sequences associated to the fibrations on the left and right of Diagram (\ref{fib-diag}) by $\{E_m^{p,q} (r), d_m = d_m (r)\}$ and $\{E_m^{p,q} (r+1), d_m  = d_m (r+1)\}$, respectively.

\begin{claim}\label{claim: stab-reg}
The map 
\begin{equation}\label{eqn: E2-map} E_2^{p,q} (r+1) \maps E_2^{p,q} (r)
\end{equation}
induced by (\ref{fib-diag}) is an isomorphism if $q\leqs r + p(p+1)$.
\end{claim}

\begin{proof} This map of $\Tor$ groups arises from 
 the map of graded chain complexes
\begin{equation}\label{kmap}\K(A)\otimes_{A} \H^*_{r+1} \maps \K(A)\otimes_{A} \H^*_r\end{equation}
induced by $j^*\co \H^*_{r+1}\to \H^*_r$. 
Say $q \leqs r + p(p+1)$.
Recall that $\K(A)_p$ is a direct sum of copies $A$, with gradings shifted upwards by at least $p(p+1)$. Hence $\K(A)\otimes_{A} \H^*_{r+1}$ and $\K(A)\otimes_{A} \H^*_{r}$ are direct sums of copies of $\H^*_{r+1}$ and $\H^*_r$, respectively, and the map (\ref{kmap}) is simply a direct sum of copies of the map $\H^*_{r+1}\to \H^*_r$, but again with grading shifted up by at least $p(p+1)$. In grading $q$, then, (\ref{kmap}) splits as a sum of maps of the form $j^* \co \H^l_{r+1} \to \H^l_{r}$, where $l\leqs q-p(p+1)\leqs r$, and this map is an isomorphism by Theorem \ref{thm: equiv.coh. stability Hom, Comm,Bcom}.
\end{proof}

A triple of integers $(p,q,m)$ with $m\geqs 2$, thought of as a point on page $m$ of the spectral sequence(s), will be called \e{stable} if the map 
\begin{equation}\label{eqn: Em-map} E_m^{p,q} (r+1) \maps E_m^{p,q} (r)
\end{equation}
is an isomorphism. 
 So Claim~\ref{claim: stab-reg} asserts that all triples $(p,q,2)$ with $q\leqs r + p(p+1)$ are stable, and we wish to prove that all points of the form $(p, q, m)$ with $q\leqs r-\sqrt{r}+p$ are stable. To simplify notation and terminology, we refer to $q-p$ as the (total) \e{cohomological degree} of the point $(p,q,m)$, and we refer to the points
$\{(p,q,m)\,:\, q=p+k\}$ as the line of cohomological degree $k$ (on page $m$).
Note that each differential out of the line of cohomological degree $k$ maps to the line of cohomological degree $k+1$.

\begin{claim}
On page $2+s$, all points of cohomological degree at most $r-s$ are 
stable ($s=0, 1, \ldots$).
\end{claim}
\begin{proof} We use induction on $s$. For $s=0$, this is a weaker statement than 
Claim~\ref{claim: stab-reg}. Assume the claim for some $s\geqs 0$, and consider a
point $(p,q,2+s+1)$ with cohomological degree at most $r-(s+1)$. 
We need to prove that 
 \begin{equation}\label{eqn: Em+1-map} E_{2+s+1}^{p,q} (r+1) \maps E_{2+s+1}^{p,q} (r)
\end{equation}
is an isomorphism. Let $m=2+s$. The map (\ref{eqn: Em+1-map}) is simply the map on homology induced by the 
map of chain complexes 
\begin{equation}\label{Em} 
\begin{tikzcd} 
E_{m}^{p-m,q-m+1} (r+1) \arrow[d] \arrow[r, leftarrow, "d_{m}"]   & 
	 E_{m}^{p,q} (r+1) \arrow[d]  \arrow[r, leftarrow, "d_{m}"]
	  & E_{m}^{p+m,q+m-1} (r+1)  \arrow[d]
	 \\
E_{m}^{p-m,q-m+1} (r) \arrow[r, leftarrow, "d_{m}"] & E_{m}^{p,q} (r) \arrow[r, leftarrow, "d_{m}"]
& E_{m}^{p+m,q+m-1} (r).
\end{tikzcd}
\end{equation}
The point $(p,q,m+1)$ has cohomological degree $q-p\leqs r-(s+1)$, and the points
$$(p-m,q-m+1, m),\,\,\, (p,q,m), \,\,\, \textrm{and} \,\,\,(p+m, q+m-1,m)$$ 
have cohomological degrees $(q-p)+1$, $q-p$, and $(q-p)-1$, respectively, 
so all three are
stable by the induction hypothesis.
Hence the vertical maps in (\ref{Em}) are isomorphisms, and it follows that the induced map in homology is an isomorphism as well.
\end{proof}

Letting $s=\lfloor \sqrt{r} \rfloor$, we find that all points of cohomological degree at most $r-\lfloor \sqrt{r} \rfloor$
are stable on page $m = 2+\lfloor \sqrt{r} \rfloor$.

Next, we claim that for $m\geqs 2+\lfloor \sqrt{r} \rfloor$, all differentials in and out of the 
lines of cohomological degree at most $r-\lfloor \sqrt{r} \rfloor$ are zero.
Recall that
by Lemma~\ref{grading}, all groups $E^{p,q}_m$ with $q<p(p+1)=p^2+p$ are zero.
The line of cohomological degree $k$ (that is, the line $q=p+k$) intersects the curve $q=p^2+p$ at $p=\sqrt{k}$, so 
all groups of the form $E_m^{p,p+k}$ with $p>\sqrt{k}$ are zero.
In particular, if $p>\sqrt{r}$ then
the groups $E_m^{p,q}$ with cohomological degree at most $r-\lfloor \sqrt{r} \rfloor$
 are zero 
(since in this situation we have $p> \sqrt{r} \geqs \sqrt{r-\lfloor \sqrt{r}\rfloor} \geqs \sqrt{q-p}$).
On pages $m\geqs 2+\lfloor \sqrt{r} \rfloor$, differentials map at least $2+\lfloor \sqrt{r} \rfloor>\sqrt{r}$ units in the horizontal direction. 
Hence on pages $m\geqs 2+\lfloor \sqrt{r} \rfloor$, all differentials out 
of non-zero groups with cohomological degree at most $r-\lfloor \sqrt{r} \rfloor$ map to trivial groups (in columns $p<0$), while all differentials into non-zero groups with cohomological degree at most $r-\lfloor \sqrt{r} \rfloor$ map out of trivial groups (in columns $p>\sqrt{r}$ and cohomological degree at most  $r-\lfloor \sqrt{r} \rfloor - 1$). This proves the claim.

It follows that for each integer $k$ with $0\leqs k \leqs r-\lfloor \sqrt{r} \rfloor$, the groups 
$$E_{2+\lfloor \sqrt{r} \rfloor}^{p,q} (r+1) \,\,\,\, \textrm{and} \,\,\,\, E_{2+\lfloor \sqrt{r} \rfloor}^{p,q} (r)$$ 
of cohomological degree  
$k$ form the associated graded groups of filtrations on $H^k (\Hom(\Z^n, G_{r+1}))$ and $H^k (\Hom(\Z^n, G_{r}))$ (respectively). The maps 
 \begin{equation*}  E_{2+\lfloor \sqrt{r} \rfloor}^{p,q} (r+1) \maps E_{2+\lfloor \sqrt{r} \rfloor}^{p,q} (r)
\end{equation*} 
 are
the induced maps between the associated graded groups of these filtrations, and we have shown that these maps are isomorphisms. It follows that 
$$H^k (\Hom(\Z^n, G_{r+1})_1)\maps H^k (\Hom(\Z^n, G_{r})_1)$$
is an isomorphism as well, completing the proof.
\end{proof}

\begin{rmk}\label{rmk: +1} The argument above in fact yields the slightly better bound 
$$r  - \lfloor \sqrt{r} \rfloor + 1 \geqs k,$$
except in the case $r=1$, $k=1$, and $G_1 = \SO(2)$.
\end{rmk}

\begin{rmk} The methods from Section~\ref{stab} may be used to show that the sequences 
$r\goesto H_k (G_r/T_r \cross T_r^n)$ extend to uniformly representation stable $\FI_W$--modules (this is similar to the arguments in Proposition~\ref{prop: stability J} and Theorem~\ref{thm: equiv.coh. stability Hom, Comm,Bcom}). By Theorem~\ref{quot-stab}, this implies homological stability for the sequences $r\goesto \Hom(\bbZ^n, G_r)$. However, this approach does not appear to yield a bound on the stable range, because we have limited information about the $\FI_W$--module $r\goesto H_k (G_r/T_r)$. For instance, we do not know a bound on its stable range. Similar comments apply to the other sequences considered in this section. 
\end{rmk}

\section{Nilpotent representations and noncompact Lie groups}\label{sec: nil}

In this section, we extend our stability results to certain noncompact Lie groups, and to finitely generated nilpotent discrete groups.

Let $G$ be a complex reductive affine algebraic group (that is, the complexification of a compact Lie group). For a discrete group $\pi$,
let $\mathfrak{X}_{\pi} (G)$ denote the $G$--character variety of 
$\pi$, defined as the GIT quotient $\Hom(\Z^n,G)/\!\!/G$ of $\Hom(\Z^n,G)$ by $G$. This space is homeomorphic to the subspace of closed orbits in $\Rep(\Z^n, G) = \Hom(\Z^n, G)/G$, and the inclusion is in fact a homotopy equivalence, and hence a homology isomorphism~\cite[Proposition 3.4]{FLR}.  We denote by 
$ \mathfrak{X}_{\pi} (G)_1$ the connected component of the trivial
representation.

These results in this section are based on the following result, as well as earlier work of Florentino--Lawton~\cite{florentino2014topology} and 
Pettet--Souto~\cite{pettet2013souto} in the abelian case.

\begin{thm}[{Bergeron~\cite{bergeron2015topology}}]\label{thm: bergeron}
Let $\Gamma$ be a finitely generated nilpotent group and let $G$ be the group of
complex or real points of a $($possibly disconnected$)$ reductive linear algebraic group,
defined over $\R$ in the latter case. If $K$ is a  maximal compact subgroup of $G$,
then there is a $K$--equivariant strong deformation retraction of $\Hom(\Gamma,G)$ onto
$\Hom(\Gamma,K)$. In particular, $\Hom(\Gamma,G)_1$ deformation retracts to
$\Hom(\Gamma,K)_1$, and $\mathfrak{X}_{\Z^n} (G)_1$ deformation retracts to
$\Rep(\Z^n, K)_1$
\end{thm}

We will mainly be interested in the free nilpotent groups, which we now define.
The descending central series of the free group $F_n$ is given by 
\begin{equation}\label{eqn: descending central series}
F_n = \Gamma^1 \rhd \Gamma^2 \rhd \Gamma^3 \rhd \cdots,
\end{equation}
where $\Gamma^2=[F_n,F_n]$ and inductively $\Gamma^{q+1}=[F_n,\Gamma^q]$, for all 
$q \geqslant 2$. The free nilpotent groups, then, are the quotients $F_n/\Gamma^q$.

We will also use the following interesting result of Bergeron and Silberman.
Their result works for any nilpotent group, but we state it only for $F_n/\Gamma^q.$

\begin{thm}[{Bergeron-Silberman~\cite{bergeron2016note}}]\label{thm: bergeron-silberman}
Let $K$ be a compact Lie group. Then the abelianization map
$F_n/\Gamma^q \maps \Z^n$ induces an inclusion
$$\Hom(\Z^n,K) \injects \Hom(F_n/\Gamma^q,K)$$
 for each $q\geqslant 2$, and on identity components this map is in fact a homeomorphism
 $$\Hom(\Z^n,K)_1 \srm{\homeo}\Hom(F_n/\Gamma^q,K)_1.$$
\end{thm}

In other words, if a homomorphism $F_n/\Gamma^q \to K$ lies in the path component of the trivial representation, then its image is in fact abelian.

Let $\{G_r\}_{r\geqs 1}$ denote one of the classical infinite families of 
compact, connected Lie groups -- namely $G_r = \SU(r)$, $\U(r)$, $\SO(2r+1)$,   
$\Sp(r)$, or  $\SO(2r)$, and let $G_r (\C)$  denote its complexification (explicitly, these groups are $\SL_r (\C)$, $\GL_r (\C)$, $\SO(2r+1, \C)$, $\Sp(2r, \C)$, and $\SO(2r, \C)$, respectively).
Let $T_r (\C) = T (G_r (\C)) \leqs G_r (\C)$ denote the complexification of $T_r\leqs G_r$.
 Since the standard inclusions $G_r\injects G_{r+1}$ and the inclusions $T_r\injects G_r$ are algebraic maps, they induce maps between complexifications, which restrict to maps $T (G_r (\C))\to T (G_{r+1} (\C))$. It is a standard fact that these inclusions in fact induce isomorphisms 
 $NT_r/T_r \srm{\isom} NT(G_r(\C))/T(G_r(\C))$ (in the semi-simple case, this follows from~\cite[\S 5.1.4, Problem 24]{Onishchik-Vinberg}). We will denote both Weyl groups by $W_r$.

Combining Theorems~\ref{thm: bergeron} and~\ref{thm: bergeron-silberman}, we see that for each of the classical sequences of Lie groups, and for each $n\geqs 1$, the inclusion
$$\Hom(\Z^n, G_r)_1\injects \Hom(F_n/\Gamma^q, G_r (\C))_1$$
is a homotopy equivalence, and the same holds for the character varieties. Combined with Theorem~\ref{Hom-bound} and Theorem~\ref{stability-wrt-r-Rep}, this yields the following corollary.

\begin{cor}\label{cor: nil complex}
Fix positive integers $n$ and $q$, with $q\geqslant 2$, and let $G_r$ be as above. The 
sequences 
$$r\goesto \Hom(F_n/\Gamma^q, G_r (\C))_1 \,\,\, \textrm{ and }\,\,\,
r\goesto\mathfrak{X}_{F_n/\Gamma^q} (G_r (\C))_1$$
are strongly rationally homologically stable, and in degree $k$, stability holds for 
$r  - \lfloor \sqrt{r} \rfloor \geqs k$ in the former case and for $r\geqs k$ in the latter case.
\end{cor}
 
The infinite-dimensional constructions from Section~\ref{sec: B(2,G) stability} also have nilpotent analogues.
The spaces
\begin{equation}\label{eqn: B_q-nil (G)}
B(q, G) :=|\Hom(F_\bullet/\Gamma^q,G)|
\end{equation}
were introduced in~\cite{adem2012commuting}, and were used to define \emph{nilpotent $K$--theory}~\cite{adem2015gomez,adem2017gomezlindtillman}.
(Note that the case $q=2$ corresponds to $\Bcom G$.) 
Since the spaces $\Hom(F_n/\Gamma^q,G)$
are not necessarily path-connected, we will focus on the
subspace $B(q, G)_1\subset B(q, G)$ defined by
$$
B(q, G)_1 :=|\Hom(F_\bullet/\Gamma^q,G)_1|.
$$
This gives filtrations of $BG$ as follows:
$$
\Bcom(G) = B(2,G) \subseteq B(3, G) \subseteq B(4, G) \subseteq \cdots \subseteq BG
$$
and 
$$
\Bcom(G)_1 = B(2,G)_1 \subseteq B(3, G)_1 \subseteq B(4, G)_1 \subseteq \cdots \subseteq BG.
$$

\begin{prop} \label{prop: Bcom-Comm}
Let $G$ be a compact connected Lie group. Then the inclusion $G\injects G(\C)$ induces homotopy equivalences
$$\Bcom (G)_1  \srm{\heq} B(q, G (\C))_1  \,\,\, \textrm{ and }\,\,\, \Comm(G)_1\srm{\heq}  \Comm(G (\C))_1$$
$($for each $q\geq 2$$)$.
\end{prop}
\begin{proof} Theorems~\ref{thm: bergeron} and~\ref{thm: bergeron-silberman} show that the inclusion of simplicial spaces
\begin{equation}\label{eqn: lwwe}\Bcom (G)_1  \injects B(q, G (\C))_1
\end{equation}
is a level-wise homotopy equivalence. The degeneracy maps for these simplicial spaces are inclusions of simplicial complexes
by the argument in Villarreal~\cite[Theorem 2.19]{villarreal2017cosimplicial} (see also \cite{hofmann2009triangulation} in the non-compact case). 
By Lillig's Union Theorem~\cite{Lillig}, this implies that these simplicial spaces are \e{proper}, and hence the map (\ref{eqn: lwwe}) is a homotopy equivalence by the results in~\cite[Appendix]{May-EGP}.

The  spaces $\Comm(G)_1$ and $\Comm(G(\C))_1$ are colimits of diagrams built from the spaces $\Hom(\bbZ^n, G)_1$ and $\Hom(\bbZ^n, G (\C))_1$, respectively, with maps induced by the coordinate projections of $\bbZ^{n+1}$ onto $\bbZ^{n}$. As these induced maps are algebraic maps between algebraic sets, they are cofibrations, and hence these colimits are homotopy equivalent to the corresponding homotopy colimits. The inclusions $\Hom(\bbZ^n, G)_1 \injects \Hom(\bbZ^n, G (\C))_1$ are homotopy equivalences~\cite{pettet2013souto}, so they induce a homotopy equivalence between the corresponding homotopy colimits.
\end{proof}

Our stability results for compact groups now yield:

\begin{cor} \label{cor: nil stable}
Fix a positive integer  $q\geqslant 2$, and let $G_r(\C)$ be as above. The
sequences 
$$r\goesto B(q, G_r (\C))_1
 \,\,\, \textrm{ and }\,\,\,
r\goesto \Comm(G_r (\C))_1 
$$
are satisfy strong rationally homological stability, and in degree $k$, stability holds for $r  - \lfloor \sqrt{r} \rfloor \geqs k$.
\end{cor}

Nilpotent analogues of $\Comm (G)$ were introduced in~\cite{cohen2016spaces}.
Define
\begin{equation}\label{eqn: X(q,G)}
\X(q,G) : = \bigg(\coprod_{n \geq 0} \Hom(F_n/\Gamma^q,G)\bigg)\bigg/\sim,
\end{equation}
where $\X(2,G)=\Comm(G)$, and the corresponding subspaces
$$
\X(q,G)_1 : = \bigg(\coprod_{n \geq 0} \Hom(F_n/\Gamma^q,G)_1\bigg)\bigg/\sim.
$$
We thereby obtain filtrations of $J(G)$ as follows:
$$
\Comm(G) = \X(2,G)  \subseteq \X(3,G) \subseteq \X(4,G) \subseteq \cdots \subseteq J(G)
$$
and
$$
\Comm(G)_1 = \X(2,G)_1 \subseteq \X(3,G)_1 \subseteq \X(4,G)_1 \subseteq \cdots \subseteq J(G).
$$

We now prove (weak) homological stability for these constructions.

\begin{cor}\label{cor: X(q,G)}
Fix a positive integer $q\geqslant 2$, and let $G_r(\C)$ be as above. Then there are isomorphisms
$$H_k \X(q, G_r (\C))_1 \isom H_k\X(q, G_{r+1} (\C))_1$$
for $r  - \lfloor \sqrt{r} \rfloor \geqs k$.
\end{cor}
\begin{proof} In light of Theorem \ref{thm: bergeron-silberman} and Theorem~\ref{Hom-bound}, it suffices to show that 
$$H_k \X(q, G_r (\C))_1 \isom H_k \Comm(G_r (\C))_1$$
for each $r\geqs 1$.
To begin, note that the stable splitting in \cite[Theorem 5.2]{cohen2016spaces} (stated there for compact groups) still holds for $G_r (\C)$ since (as discussed in the proof of Corollary~\ref{cor: nil stable}) the inclusions 
$$\Hom(F_n/\Gamma^q, G_r (\C))_1 \injects \Hom(F_{n+1}/\Gamma^q, G_r (\C))_1$$
are cofibrations.
That is, we have decompositions 
$$
\Sigma X(q,G_r(\C))_1 \simeq \Sigma \bigvee_{n\geqslant 1} \widehat{\Hom}(F_n/\Gamma^q,G_r(\C))_1,
$$
where $\widehat{\Hom}(F_n/\Gamma^q,G_r(\C))_1$ is the quotient of ${\Hom}(F_n/\Gamma^q,\GL_r(\C))_1$ 
by the subspace $S_{n,q}(G_r(\C))$ consisting of all the nilpotent $n$--tuples with at least one coordinate equal to 
the identity element of $G_r(\C)$.
The argument in~\cite[Theorem 4.1]{ramras2017hilbert} now shows that the natural map
$$\widehat{\Hom}(\Z^n,G_r(\C))_1\maps \widehat{\Hom}(F_n/\Gamma^q,G_r(\C))_1$$
is a homotopy equivalence, completing the proof.
\end{proof}

As an immediate consequence of Theorem~\ref{thm: bergeron} and Theorem~\ref{thm: stability for fixed G}, we also have stability with respect to the maps of discrete groups $F_n/\Gamma^q\maps F_{n+1}/\Gamma^q$.

\begin{cor}\label{cor: stability in n nil} Let $G$ be as in Theorem~\ref{thm: bergeron}. Then for each $q\geqs 2$, the sequences 
$$n\goesto \Hom(F_n/\Gamma^q, G)_1\,\,\, \textrm{ and }
\,\,\, n\goesto \mathfrak{X}_{F_n/\Gamma^q} (G)_1$$
are strongly rationally homologically stable, and in homological degree $k$, stability holds for $n\geqs k$.
\end{cor}

\noindent {\bf Real reductive groups.}
Since Theorems~\ref{thm: bergeron} applies to real reductive groups, we can also consider families of such groups whose maximal compact subgroups correspond to the classical sequences of compact Lie groups. For instance, the maximal compact subgroup of $\SL_r (\bbR)$ is $\SO(r)$, and hence we obtain strong homological stability (with the same bounds as above) for all four sequences 
$$\{\Hom(F_n/\Gamma^q, \SL_r(\bbR))_1\}_{r\geqs 1}, \,\,\,\{\mathfrak{X}_{F_n/\Gamma^q} (\SL_r(\bbR))_1\}_{r\geqs 1},$$
$$\{B(q, \SL_r(\bbR))_1\}_{r\geqs 1},\,\,\, \textrm{ and }\,\,\, \{\Comm (\SL_r(\bbR))_1\}_{r\geqs 1}.$$
A similar statement applies to the sequence of real symplectic groups $\Sp(2r, \bbR)$, which are reductive with maximal compact $\U(r)$. 

Finally, consider the indefinite groups $\U(s,r)$ and $\SO(s,r)$, which are real reductive with maximal compact subgroups $\U(s)\cross \U(r)$ and $\SO(s)\cross \SO(r)$, respectively. In these cases, we may stabilize with respect to either variable $r$ or (equivalently) $s$. In general, for Lie groups $G$ and $H$ there are homeomorphisms 
$$\Hom(\Z^n, G\cross H) \isom \Hom(\Z^n, G)\cross \Hom(\Z^n, H),$$
and similarly for the character varieties. Moreover, by Ebert--Randal-Williams~\cite[Theorem 7.2]{ERW}, geometric realization commutes with products of simplicial spaces up to weak homotopy equivalence, so the map
$$\Bcom (G\cross H) \maps \Bcom (G)\cross \Bcom (H)$$
is a weak homotopy equivalence, and hence an isomorphism in (co)homology.\footnote{Here it is important that we give $\Bcom (G)\cross \Bcom (H)$ the compactly generated topology associated to the product topology.}
The K\"unneth Theorem now shows that the sequences
$$\{\Hom(\Z^n, \U(s,r))\}_{r\geqs 1}, \,\,\,\{\mathfrak{X}_{\Z^n} (\U(s,r))\}_{r\geqs 1},\,\,\, \textrm{ and }\,\,\, \{\Bcom (\U(s,r))\}_{r\geqs 1}$$
are all strongly rationally homologically stable (with the same stability bounds), and similarly for  $\SO(s,r)$ in place of $\U(s,r)$ (at least after restricting to the connected components of the trivial representations).

\begin{rmk}\label{rmk: eqvt11} The results in this section can be extended to equivariant homology. For instance, consider the $G_r (\C)$--equivariant homology of $\Hom(\Z^n, G_r (\C))_1$.
Comparing the fibrations
$$(\Hom(\Z^n, G_r (\C))_1)_{hG_r (\C)}\maps BG_r (\C)$$
and
$$(\Hom(\Z^n, G_r)_1)_{hG_r}\maps BG_r,$$
one sees (using Theorem~\ref{thm: bergeron}) that the homotopy orbit spaces are in fact homotopy equivalent. The other cases are similar.
\end{rmk}

\section{Finite covers of Lie groups}\label{sec: covers}

In this section we show that passing to a finite cover of the underlying Lie group does not change
the homology of the various spaces of commuting elements considered in this article. In particular, this allows us to extend our stability results to the Spin groups, and to the projective unitary and general linear groups 
(although in the latter case, there are no stabilization maps, so we have only weak stability).

\begin{prop}\label{prop: fin cov} Let $p\co G\to H$ be a finite covering homomorphism between connected Lie groups, and assume that $G$ and $H$ are either compact or complex reductive affine algebraic groups.
Then for each $n\geqs 1$ and each $q\geqs 2$, the induced maps 
\begin{equation}\label{HX} \Hom(F_n/\Gamma^q, G)_1\maps \Hom(F_n/\Gamma^q, H)_1  \,\,\, \textrm{and}\,\,\,\mathfrak{X}_{F_n/\Gamma^q} (G)_1\maps \mathfrak{X}_{F_n/\Gamma^q} (H)_1\end{equation}
are $($rational$)$ homology isomorphisms, as are the maps
$$B(q, G)_1\maps B(q, H)_1 \,\,\, \textrm{and}\,\,\,
\Comm(G)_1\maps \Comm(H)_1.$$
Finally, when $G$ and $H$ are compact, the same holds for
$$\Bcom (G)_1/G \to \Bcom (H)_1/H\,\,\, \textrm{and}\,\,\, \Comm(G)_1/G \to \Comm(H)_1/H.$$
\end{prop}
\begin{proof} First we consider the case in which $G$ and $H$ are compact.
Let $T\leqs H$ be a maximal torus, and define $\wt{T} := p^{-1} (T)$.
We claim that $\wt{T}$ is a maximal torus in $G$, and that the induced map of Weyl groups is an isomorphism.
Note that $p$ induces an isomorphism of Lie algebras, and hence the maximal tori in $G$ and $H$ have the same rank.
 Since all finite covers of a torus are disjoint unions of tori (of the same rank as the base torus), the identity component   $p^{-1} (T)_0 \leqs p^{-1} (T)$  is a maximal torus in $G$, so it  suffices to show that $p^{-1} (T)$ is connected.
Observe that $\ker (p)\leqs G$ is discrete and normal, hence central by~\cite[Theorem 6.13]{HM06},
so $p^{-1} (T)$  centralizes $p^{-1} (T)_0$
(note that by an elementary covering space argument, $p^{-1} (T)$ is generated by $p^{-1} (T)_0$ together with $\ker (p)$). Since maximal tori in compact Lie groups are their own centralizers, this shows that $\wt{T}$  is a maximal torus of $G$. 
It follows that $p$ induces an isomorphism $\wt{W}\to W$ between the Weyl groups of 
$\wt{T}$ and $T$, since these groups are naturally isomorphic to the Weyl groups of the  root systems associated to $\wt{T}$ and 
$T$~\cite[Chapter 11]{Hall-Lie-groups}, and $p$ induces an isomorphism of Lie algebras.\footnote{Alternatively, surjectivity of $p$ implies that $p^{-1} (NT) = N\wt{T}$, and then the 9-lemma implies that $\wt{W}\to W$ is an isomorphism.}

To study the maps (\ref{HX}), first recall that the
Poincar\'{e} polynomial of these spaces are given by the formulas (\ref{eqn: Hom-PP}) and (\ref{eqn: Rep-PP}), which depend only on the Weyl group and its action on the Lie algebra of the maximal torus. Since $p$ induces an isomorphism $\wt{W} \srm{\isom} W$ and a Weyl--equivariant diffeomorphism $\wt{T}\to T$, the Poincar\'e polynomials (that is, the ranks of the rational homology groups) are unchanged under passage from $G$ to $H$. It will thus suffice to show that  the maps (\ref{HX}) are surjective on rational homology. 

By a result of Goldman~\cite[Lemma 2.2]{goldman1988topological}, 
the map of homomorphism spaces is a normal covering map with structure group $\Hom(F_n/\Gamma^q, \ker(p)) = \ker(p)^n$ (note that, as observed above, $\ker (p)$ is central in $G$ and in particular is abelian). Hence $\Hom(F_n/\Gamma^q, H)_1$ is the quotient of $\Hom(F_n/\Gamma^q, G)_1$ by the finite group $\ker(p)^n$, and by Lawton--Ramras~\cite[Lemmas 3.7 and 3.9]{Lawton-Ramras}, we have 
$$\mathfrak{X}_{F_n/\Gamma^q} (G)_1/\ker(p)^n \homeo \mathfrak{X}_{F_n/\Gamma^q} (H)_1$$
as well. Now Proposition~\ref{coh-quot} implies that the maps (\ref{HX}) are surjective on rational homology (for homomorphism spaces, this can also be seen by considering the transfer map).

The result for $B(q, G)_1\to B(q, H)_1$ now follows from the fact that a level-wise homology equivalence between proper simplicial spaces is a homology equivalence on realizations (May~\cite[Appendix]{May-EGP}). The assumption that $G$ and $H$ are compact ensures properness, as discussed in the proof of Proposition~\ref{prop: Bcom-Comm}.

Next we show that $\Comm(G)_1\to \Comm(H)_1$ is a homology equivalence.
Our choice of maximal tori implies that $p$ induces a commutative diagram of the form
\begin{equation}\label{eqn: phi'2}
\begin{tikzcd}  
G/\wt{T}\cross_{\wt{W}} J(\wt{T}) \arrow[d, "p_*"] \arrow[r, "\phi''"]   
	& \Comm(G)_1 \arrow[d, "p_*"]   
	 \\
H/T\cross_{W} J(T)  \arrow[r, "\phi''"]   
	& \Comm(H)_1.
\end{tikzcd}
\end{equation}
The horizontal maps in (\ref{eqn: phi'2}) are homology equivalences by Theorem~\ref{thm: phi}, so to show that $\Comm(G)_1\to \Comm(H)_1$ is a homology equivalence it suffices to prove the same for the map $G/\wt{T}\cross_{\wt{W}} J(\wt{T}) \to H/T\cross_{W} J(T)$ induced by $p$. 
By Proposition~\ref{coh-quot} and the K\"unneth Theorem, it suffices  to show that the maps 
\begin{equation}\label{J} H_* (J(\wt{T}))\to H_* (J(T))
\end{equation}
  and 
 \begin{equation}\label{/T} H_* (G/\wt{T})\to H_* ( H/T)
 \end{equation}
 induced by $p$ are homology equivalences.\footnote{To apply Proposition~\ref{coh-quot}, we need to establish that the Weyl group actions in (\ref{eqn: phi'2}) are good. This follows from the argument in the proof of Theorem~\ref{thm: stability for BcomG/G and ComG/G}.}

To see that (\ref{J}) is an isomorphism, first note that 
 $p\co \wt{T}\to T$ is a covering map between spaces with isomorphic homology. As discussed above, this implies that $p$ is a homology equivalence. Now the natural isomorphism $H_*(J(X)) \isom \mathcal{T} (\wt{H}_* (X))$ shows that (\ref{J}) is a homology equivalence as well.

To analyze (\ref{/T}), we construct a commutative diagram
\begin{center}
\begin{tikzcd}
G/\wt{T}  \arrow{r}{\wt{j}}  \arrow{d}{p}  & 
	B\wt{T} \arrow[d, "Bp"] 
	 \\
H/T  \arrow{r}{j} &
BT, 
\end{tikzcd}
\end{center}
where the horizontal maps are classifying maps for the bundles $G\to G/\wt{T}$ and $H\to H/T$. To obtain these compatible classifying maps, choose a point $e_0\in EG$ and let $\overline{e_0}$ denote its image in $EH$ under the map $EG\to EH$ induced by $p$. The inclusions $G/\wt{T} \subset EG/\wt{T}$, $[g]\goesto e_0\cdot g$ and $H/T \subset EH/T$, $[h]\goesto \overline{e_0}\cdot h$ yield the desired classifying maps.

By Proposition~\ref{prop: G/T}, the horizontal maps $\wt{j}$ and $j$ become isomorphisms in cohomology after modding out the ideal of positive degree Weyl--invariants on the right-hand side. The map $Bp\co B\wt{T}\to BT$ is Weyl--equivariant and a (co)homology equivalence (again by~\cite[Appendix]{May-EGP}), so it restricts to an isomorphism between these ideals. It follows that 
the map $G/\wt{T}\to H/T$ is an isomorphism in (co)homology, as desired.

The fact that $p$ induces homology isomorphisms
$$\Bcom (G)_1/G \to \Bcom (H)_1/H\,\,\, \textrm{and}\,\,\, \Comm(G)_1/G \to \Comm(H)_1/H$$
follows similarly, using Proposition~\ref{prop: homeo}.

Finally, we consider the case in which $G$ and $H$ are complex. First, note that if $K\leq H$ is a maximal compact subgroup, then since finite covers preserve compactness, $p^{-1} (K)$ is a maximal compact subgroup of $G$. Now if $F$ is any of the functors under consideration, we have a commutative diagram
\begin{center}
\begin{tikzcd}
F(p^{-1} K) \arrow{r} \arrow{d}  & 
	F(G) \arrow[d] 
	 \\
F(K) \arrow{r}&
F(H), 
\end{tikzcd}
\end{center}
in which the left-hand vertical map is a homology isomorphism, while the horizontal arrows are homotopy equivalences (see Theorem~\ref{thm: bergeron} and Proposition~\ref{prop: Bcom-Comm}), and it follows that the right-hand vertical map is a homology isomorphism as well.
\end{proof}

We end by discussing some examples in which Proposition~\ref{prop: fin cov} applies.

\begin{ex}\label{ex: Spin} For $r\geqs 3$, the Spin group $\Spin(r)$ is the universal covering group of $\SO(r)$, and since $\pi_1 (\SO(r)) = \Z/2$ for $r\geqs 3$, this is in fact a double covering. Let $p \co \Spin(r)\to \SO(r)$ be the covering map. The standard inclusions 
$$\SO(r)\injects \SO(r+1)$$ 
(block sum with the $1\cross 1$ identity matrix) induce maps $\Spin(r)\to \Spin(r+1)$, and it follows from Proposition~\ref{prop: fin cov} (and the earlier results in the article) that these maps induce homology isomorphisms after applying any of the functors considered in Proposition~\ref{prop: fin cov}, except possibly in the case of character varieties (see also Remark~\ref{rmk: SO}). Moreover, the stable ranges are the same as for the special orthogonal groups.
In the case of character varieties, the proof of Proposition~\ref{prop: fin cov} involves comparing Poincar\'{e} polynomials, so we obtain only weak stability in this case.
\end{ex} 

\begin{ex}\label{ex: proj} For each of the families of compact or complex Lie groups considered in this article, there is an associated family of projective groups obtained by modding out the centers (although in some cases the center is trivial). In each case, we obtain (weak) homological stability results for the various functors considered in Proposition~\ref{prop: fin cov}. Note that the standard inclusions \e{do not} map centers to centers, and hence we do not have maps between the projective groups inducing these homology isomorphisms.
\end{ex}

\begin{rmk}\label{rmk: eqvt12} As in Remark~\ref{rmk: eqvt11}, 
the results in this section extend to equivariant homology. For instance, if $p\co G\to H$ is a finite covering, to compare the $G$--equivariant homology of $\Hom(\Z^n, G)_1$ to the $H$--equivariant homology of $\Hom(\Z^n, H)_1$, we compare the 
 fibrations
$$(\Hom(\Z^n, G)_1)_{hG}\maps BG$$
and
$$(\Hom(\Z^n, H)_1)_{hH}\maps BH.$$
The induced map $Bp\co BG\to BH$ is a (rational) homology equivalence by~\cite[Appendix]{May-EGP}, and the map of fibers is as well (by Proposition~\ref{prop: fin cov}). Comparing the Serre spectral sequences for the two fibrations, we obtain the desired isomorphism in equivariant homology.
\end{rmk}


\begin{thebibliography}{10}


\bibitem{adem2007commuting}
Adem, A. and F. R. Cohen. 
\newblock Commuting elements and spaces of homomorphisms.
\newblock {\em Math. Ann.} 338 (2007), no. 3, 587--626.



\bibitem{adem2012commuting}
Adem, A., F. R. Cohen, and E. Torres Giese.
\newblock Commuting elements, simplicial spaces and filtrations of classifying
  spaces.
\newblock {\em Math. Proc. Cambridge Philos. Soc.} 152 (2012), no. 1, 91--114.



\bibitem{adem2015gomez}
Adem, A. and J. G{\'o}mez.
\newblock A classifying space for commutativity in {L}ie groups.
\newblock {\em Algebr. Geom. Topol.} 15 (2015), no. 1, 493--535.



\bibitem{adem2017gomezlindtillman}
Adem, A., J. G{\'o}mez, J. Lind, and U. Tillmann.
\newblock Infinite loop spaces and nilpotent {K}--theory.
\newblock {\em Algebr. Geom. Topol}. 17 (2017), no. 2, 869--893.



\bibitem{AGV}
Antol{\'\i}n-Camarena, O., S. Gritschacher, and B. Villarreal.
\newblock Classifying spaces for commutativity of low-dimensional {L}ie groups.
\newblock To appear in {\em Math. Proc. Cambridge Philos. Soc.} (Published online July 2019).  



\bibitem{arnol1969cohomology}
Arnol'd, V. I.
\newblock The cohomology ring of the colored braid group.
\newblock {\em Math. Notes} 5 (1969), no. 2, 138--140.



\bibitem{baird2007cohomology}
Baird, T. J.
\newblock {Cohomology of the space of commuting $n$--tuples in a compact {L}ie
  group}.
\newblock {\em Algebr. Geom. Topol}. 7 (2007), 737--754.



\bibitem{bergeron2015topology}
Bergeron, M.
\newblock The topology of nilpotent representations in reductive groups and
  their maximal compact subgroups.
\newblock  {\em Geom. Topol.} 19 (2015), no. 3, 1383--1407.



\bibitem{bergeron2016note}
Bergeron, M. and L. Silberman.
\newblock A note on nilpotent representations.
\newblock {\em J. Group Theory} 19 (2016), no. 1, 125--135.



\bibitem{BLR}
Biswas, I., S. Lawton, and D. Ramras.
\newblock Fundamental groups of character varieties: surfaces and tori.
\newblock {\em Math. Z.} 281 (2015), no. 1--2, 415--425.



\bibitem{Borel57}
Borel, A.
\newblock Sur la cohomologie des espaces fibr\'es principaux et des espaces
  homog\`enes de groupes de {L}ie compacts.
\newblock {\em Ann. of Math. (2)} 57 (1953), 115--207.



\bibitem{borel2002almost}
Borel, A., R. Friedman, and J. W. Morgan.
\newblock {\em Almost commuting elements in compact {L}ie groups}.
\newblock {\em Mem. Amer. Math. Soc.} 157 (2002), no. 747, x+136 pp. 



\bibitem{Bredon}
Bredon, G. E.
\newblock \e{Introduction to compact transformation groups.}
\newblock Pure and Applied Mathematics, Vol. 46.
\newblock {\em Academic Press, New York}, 1972.



\bibitem{bredon2012sheaf}
Bredon, G. E.
\newblock {\em Sheaf theory}, Second edition. Grad. Texts in Math., 170.
\newblock {\em Springer-Verlag, New York,} 1997.



\bibitem{BTD}
Br\"ocker, T. and T. tom Dieck.
\newblock {\em Representations of compact {L}ie groups}, 
{Grad. Texts in Mathematics}, 98.
\newblock {\em Springer-Verlag, New York,} 1995.
\newblock Translated from the German manuscript. Corrected reprint of the 1985
  translation.



\bibitem{church2012homological}
Church, T.
\newblock Homological stability for configuration spaces of manifolds.
\newblock {\em Invent. Math.} 188 (2012), no. 2, 465--504.



\bibitem{church2015fi}
Church, T., J. S. Ellenberg, and B. Farb.
\newblock F{I}--modules and stability for representations of symmetric groups.
\newblock {\em Duke Math. J.} 164 (2015), no. 9, 1833--1910.



\bibitem{church2013representation}
Church, T. and B. Farb.
\newblock Representation theory and homological stability.
\newblock {\em Adv. Math.} 245 (2013), 250--314.



\bibitem{cohen1972thesis}
Cohen, F. R. 
\newblock {\em The Cohomology of Braid Spaces}.
\newblock {PhD} thesis, The University of Chicago, 1972.



\bibitem{cohen2016spaces}
Cohen, F. R. and M. Stafa.
\newblock On spaces of commuting elements in {L}ie groups.
\newblock {\em Math. Proc. Cambridge Philos. Soc.} 161 (2016), no. 3, 381--407.



\bibitem{ERW}
Ebert, J. and O. Randal-Williams.
\newblock Semisimplicial spaces.
\newblock {\em  Algebr. Geom. Topol.} 19 (2019), no. 4, 2099--2150.



\bibitem{florentino2014topology}
Florentino, C. and S. Lawton.
\newblock Topology of character varieties of {A}belian groups.
\newblock {\em Topology Appl.} 173 (2014), 32--58.



\bibitem{FLR}
Florentino, C., S. Lawton, and D. Ramras.
\newblock Homotopy groups of free group character varieties.
\newblock {\em Ann. Sc. Norm. Super. Pisa Cl. Sci.} (5) 17 (2017), no. 1, 143--185.



\bibitem{FS2017}
Florentino, C. A. and J. A. M. Silva.
\newblock Hodge-deligne polynomials of abelian character varieties.
\newblock December 2017.  { \href{https://arxiv.org/abs/1711.07909}{arXiv:1711.07909}}.



\bibitem{geck2000characters}
Geck, M. and G. Pfeiffer.
\newblock {\em Characters of finite {C}oxeter groups and {I}wahori-{H}ecke algebras}, 
			London Mathematical Society Monographs. New Series, 21.
\newblock {\em The Clarendon Press, Oxford Univ. Press, New York,} 2000.



\bibitem{goldman1988topological}
Goldman, W. M.
\newblock Topological components of spaces of representations.
\newblock {\em Invent. Math.} 93 (1988), no. 3, 557--607.



\bibitem{gritschacher2018spectrum}
Gritschacher, S.
\newblock The spectrum for commutative complex {K}--theory.
\newblock {\em Algebr. Geom. Topol.} 18 (2018), no. 2, 1205--1249.



\bibitem{Grothendieck}
Grothendieck, A.
\newblock  Sur quelques points d'alg\`ebre homologique,
\newblock {\em Tohoku Math. J.} 9 (1957), no. 2, 119--221.



\bibitem{Hall-Lie-groups}
Hall, B. C.
\newblock {\em Lie groups, {L}ie algebras, and representations. An elementary
  introduction}, Second edition. Grad. Texts in Math., 222.
\newblock {\em Springer, Cham,} 2015



\bibitem{HM06}
 Hofmann, K. H. and S. A. Morris.
\newblock {\em The structure of compact groups. A primer for the student--a
  handbook for the expert}, Third edition, revised and augmented. De Gruyter Stud. in Math., 25.
\newblock  {\em De Gruyter, Berlin,} 2013.



\bibitem{hofmann2009triangulation}
Hofmann, K. R.
\newblock {\em Triangulation of locally semi-algebraic spaces}.
\newblock PhD thesis, University of Michigan, 2009.



\bibitem{kac2000Smilga}
Kac, V. G. and A. V. Smilga.
\newblock Vacuum structure in supersymmetric {Y}ang-{M}ills theories with any gauge group. 
\newblock  {\em The many faces of the superworld,} 185--234, {\em World Sci. Publ., River Edge, NJ,} 2000.



\bibitem{Langalgebra}
Lang, S.
\newblock {\em Algebra}, Revised third edition. Grad. Texts in Math., 211. 
\newblock {\em Springer-Verlag, New York,} 2002.



\bibitem{Lawton-Ramras}
Lawton, S. and D. Ramras.
\newblock Covering spaces of character varieties. {W}ith an appendix by
  {N}.-{K}. {H}o and {C}.-{C}. {M}. {L}iu.
\newblock {\em New York J. Math.} 21 (2015), 383--416.



\bibitem{Lillig}
Lillig, J.
\newblock A union theorem for cofibrations.
\newblock {\em Arch. Math. (Basel)} 24 (1973), 410--415.



\bibitem{MacDonald}
Macdonald, I. G.
\newblock{Symmetric products of an algebraic curve.}
\newblock {\em Topology} 1 (1962), 319--343.



\bibitem{May-GOILS}
May, J. P.
\newblock {\em The geometry of iterated loop spaces}.
\newblock Lecture Notes in Math., Vol. 271.
\newblock {\em Springer-Verlag, Berlin-New York,} 1972.



\bibitem{May-EGP}
May, J. P.
\newblock {$E_{\infty }$} spaces, group completions, and permutative categories.
\newblock {\em New developments in topology (Proc. Sympos. Algebraic Topology, Oxford, 1972)}, pp. 61-93. 
\newblock London Math. Soc. Lecture Note Ser., No. 11. 
\newblock {\em Cambridge Univ. Press, London,} 1974.
  
  

\bibitem{may1992simplicial}
May, J. P.
\newblock {\em Simplicial objects in algebraic topology}.
\newblock Chicago Lectures in Mathematics. {\em University of Chicago Press,} Chicago, IL, 1992.
\newblock Reprint of the 1967 original.



\bibitem{mccleary2001ssbook}
McCleary, J.
\newblock {\em A user's guide to spectral sequences.}
\newblock Second edition. Cambridge Stud. Adv. Math., 58.
\newblock {\em Cambridge Univ. Press, Cambridge,} 2001.



\bibitem{Michael57}
Michael, E. 
\newblock Another note on paracompact spaces.
\newblock {\em Proc. Amer. Math. Soc.} 8 (1957), 822--828.

\bibitem{Michael-selection}
Michael, E. 
\newblock Continuous selections. I.
\newblock  {\em Ann. of Math. (2)} 63 (1956), 361--382.

\bibitem{Milnor56}
Milnor, J.
\newblock Construction of universal bundles. {II}.
\newblock {\em Ann. of Math. (2)} 63 (1956), 430--436.


\bibitem{Toda-Mimura}
Mimura, M. and H. Toda.
\newblock {\em Topology of {L}ie groups. {I}, {II}},
\newblock Translated from the 1978 Japanese edition by the authors. 
\newblock Transl. Math. Monogr., 91.
\newblock {\em American Mathematical Society, Providence, RI,} 1991.



\bibitem{nLab-colim}
{nLab authors}.
\newblock colimits of paracompact {{H}}ausdorff spaces.
\newblock
  \url{http://ncatlab.org/nlab/show/colimits\%20of\%20paracompact\%20Hausdorff\%20spaces},
 Accessed February 2020.
\newblock
 {\em{ \href{http://ncatlab.org/nlab/revision/colimits\%20of\%20paracompact\%20Hausdorff\%20spaces/14}{Revision
  14}}}




\bibitem{Onishchik-Vinberg}
Onishchik, A.~L.  and \`E.~B. Vinberg.
\newblock {\em Lie groups and algebraic groups}.
\newblock Translated from the Russian and with a preface by D. A. Leites. 
\newblock Springer Series in Soviet Mathematics. {\em Springer-Verlag, Berlin,} 1990. 




\bibitem{pettet2013souto}
Pettet, A. and J. Souto.
\newblock Commuting tuples in reductive groups and their maximal compact subgroups.
\newblock {\em Geom. Topol.} 17 (2013), no. 5, 2513--2593.



\bibitem{Quillen}
Quillen, D.
\newblock Finite generation of the groups {$K_{i}$} of rings of algebraic integers.
\newblock {\em Algebraic K-theory, I: Higher K-theories 
			(Proc. Conf., Battelle Memorial Inst., Seattle, Wash., 1972)}, 
			pp. 179--198. Lecture Notes in Math., Vol. 341. 
\newblock {\em Springer, Berlin,} 1973.



\bibitem{ramras2017hilbert}
Ramras, D. A. and M. Stafa.
\newblock {Hilbert-Poincar\'e series of nilpotent representations in {L}ie groups}.
\newblock {\em Math. Z.}  292 (2019), no. 1-2, 591--610.



\bibitem{reeder1995cohomology}
Reeder, M.
\newblock {On the cohomology of compact {L}ie groups}.
\newblock {\em Enseign. Math. (2)}  41 (1995), no. 3-4, 181--200.



\bibitem{Schwarz}
Schwarz, G. W.
\newblock {Smooth functions invariant under the action of a compact {L}ie group.}
\newblock {\em Topology} 14 (1975), 63--68.

\bibitem{segal1968csss}
Segal, G.
\newblock Classifying spaces and spectral sequences.
\newblock {\em Inst. Hautes \'{E}tudes Sci. Publ. Math.} No. 34 (1968), 105--112.


  
\bibitem{Pazzis}
de Seguins Pazzis, C.
\newblock The geometric realization of a simplicial Hausdorff space is Hausdorff.
\newblock {\em Topology Appl.} 160 (2013), no. 13, 1621--1632.



\bibitem{Sella}
Sella, Y.
\newblock Comparison of sheaf cohomology and singular
cohomology.
\newblock   March 2016.
{  \href{https://arxiv.org/abs/1602.06674}{arXiv:1602.06674}}.




\bibitem{smith_Eilenberg-Moore-SS}
Smith, L.
\newblock {\em Lectures on the {E}ilenberg-{M}oore spectral sequence}.
\newblock Lecture Notes in Math., Vol. 134. 
\newblock {\em Springer-Verlag, Berlin-New York,} 1970.



\bibitem{stafa2017poincare}
Stafa, M.
\newblock Poincar\'e series of character varieties for nilpotent groups.
\newblock {\em J. Group Theory} 22 (2019), no. 3, 419--440.



\bibitem{Steenrod-convenient}
Steenrod, N. E.
\newblock A convenient category of topological spaces.
\newblock {\em Michigan Math. J.} 14 (1967), 133--152.



\bibitem{steenrod}
Steenrod, N. E.
\newblock Cohomology operations, and obstructions to extending continuous functions.
\newblock {\em Adv. Math.} 8 (1972), 371--416.




\bibitem{torres2008fundamental}
Torres Giese, E. and D. Sjerve.
\newblock Fundamental groups of commuting elements in {L}ie groups.
\newblock {\em Bull.  Lond. Math. Soc.} 40 (2008), no. 1, 65--76.



\bibitem{villarreal2017cosimplicial}
Villarreal, B.
\newblock Cosimplicial groups and spaces of homomorphisms.
\newblock {\em Algebr. Geom. Topol.}  17 (2017), no. 6, 3519--3545.



\bibitem{Wilson-Math-Z}
Wilson, J. C. H.
\newblock {FI$_W$}-modules and constraints on classical {W}eyl group characters.
\newblock {\em Math. Z.} 281 (2015), no. 1-2, 1--42.



\bibitem{wilson2014fiw}
Wilson, J. C. H.
\newblock {FI$_W$-modules and stability criteria for representations of classical Weyl groups}.
\newblock {\em J. Algebra}, 420 (2014), 269--332.



\bibitem{witten1998toroidal}
Witten, E.
\newblock Toroidal compactification without vector structure.
\newblock {\em J. High Energy Phys.} 1998, no. 2, Paper 6, 43 pp.

\end{thebibliography}

\end{document}